\documentclass[bj,preprint]{imsart}

\RequirePackage{amsthm,amsmath,amsfonts,amssymb,bbm,graphicx,xr,xargs}
\RequirePackage[numbers]{natbib}

\usepackage[shortlabels]{enumitem}

\RequirePackage[colorlinks,citecolor=blue,urlcolor=blue]{hyperref}

\usepackage{cleveref}

\startlocaldefs
\numberwithin{equation}{section}
\numberwithin{equation}{section}

\newcommand{\toi}{\to\infty}
\newcommand{\eind}{\stackrel{d}{=}}
 
\newcommand{\cleq}{\leq}

\newcommand{\dto}{%
	\mathrel{\vbox{\offinterlineskip\ialign{%
				\hfil##\hfil\cr
				$\scriptstyle d$\cr
				$\longrightarrow$\cr
			}}}}
\newcommand{\wto}{%
	\mathrel{\vbox{\offinterlineskip\ialign{%
				\hfil##\hfil\cr
				$\scriptstyle w$\cr
				$\longrightarrow$\cr
			}}}}
\newcommand{\vto}{%
	\mathrel{\vbox{\offinterlineskip\ialign{%
				\hfil##\hfil\cr
				$\scriptstyle v$\cr
				$\longrightarrow$\cr
			}}}}

\newcommand{\lless}{\prec} 
\newcommand{\lgreater}{\succ}
\newcommand{\lleq}{\preceq} 
\newcommand{\lgeq}{\succeq}

\newcommand{\pr}{\mathbb{P}} 
\newcommand{\ex}{\mathbb{E}}
\newcommand{\EE}{\mathbb{E}}
\newcommand{\ZZ}{\mathbb{Z}}
\newcommand{\PP}{\mathbb{P}}

\newcommand\ind[1]{\mathbbm{1}{\left\{#1\right\}}}

\newcommand\1[1]{\mathbbm{1}_{#1}}

\newcommand{\N}{\mathbb{N}}
\newcommand{\R}{\mathbb{R}} 
 
\newcommand{\Z}{\mathbb{Z}}

\newcommand{\bQ}{\boldsymbol{Q}}

\newcommand{\bX}{\boldsymbol{X}}
\newcommand{\bY}{\boldsymbol{Y}}
\newcommand{\bZ}{\boldsymbol{Z}}

\newcommand{\bTheta}{\boldsymbol{\Theta}}

\newcommand{\bx}{\boldsymbol{x}}
\newcommand{\by}{\boldsymbol{y}}
\newcommand{\tbx}{\tilde{\boldsymbol{x}}}

\newcommand{\bt}{\boldsymbol{t}}

\newcommand{\bT}{\boldsymbol{T}}
\newcommand{\bi}{\boldsymbol{i}}
\newcommand{\bj}{\boldsymbol{j}}
\newcommand{\bk}{\boldsymbol{k}}
\newcommand{\bo}{\boldsymbol{0}}
\newcommand{\bone}{\boldsymbol{1}}
\newcommand{\ba}{\boldsymbol{a}}
\newcommand{\bb}{\boldsymbol{b}}

\newcommand{\lo}{\tilde{l}_0}
\newcommand{\loo}{\tilde{l}_{0,0}}
\newcommand{\bz}{{\varepsilon}}

\newcommand{\s}{s} %%%score function
\newcommand{\ts}{\alpha^*} %%%%% index of regular variation for alignments, before it was theta star
\newcommand{\mus}{\mu^*} 
\newcommand{\tailcons}{C}

\newcommand{\bbou}{\mathcal{B}_b}
\newcommand{\F}{\mathcal{F}}

\newcommand{\spaceX}{\mathbb{X}} 

\newcommand{\x}{\spaceX}
\newcommand{\borel}{\mathcal{B}}
\newcommand{\bborel}{\mathcal{B}_b}

\newcommand{\bounded}{\bborel}

\newcommand{\mx}{\mathcal{M}(\spaceX)}

\newcommand{\mpx}{\mathcal{M}_p(\mathbb{X})}

\newcommand{\cbb}{CB_b(\spaceX)}
\newcommand{\cbbp}{CB_b^+(\spaceX)}

\newcommand{\cl}[1]{\overline{#1}}

\newcommand{\sub}{\subseteq}

\newcommand{\PPP}{\mathrm{PPP}}

\newcommand{\spectral}{K}

\newcommand{\dtilde}{\tilde{d}}
\newcommandx\sequence[3][2=\ZZ,3=t]{(#1_{#3})_{#3\in#2}}

\theoremstyle{plain}
\newtheorem{theorem}{Theorem}[section]
\newtheorem{lemma}[theorem]{Lemma}
\newtheorem{corollary}[theorem]{Corollary}
\newtheorem{proposition}[theorem]{Proposition}
\newtheorem{hypothesis}[theorem]{Assumption}
\theoremstyle{remark}
\newtheorem{definition}{Definition}[section]
\newtheorem{remark}{Remark}[section]
\newtheorem{example}{Example}[section]

\endlocaldefs

\begin{document}

\begin{frontmatter}
\title{Compound Poisson approximation for regularly varying fields with application to sequence alignment}
%\title{A sample article title with some additional note\thanksref{T1}}
\runtitle{Compound Poisson approximation for regularly varying fields}
%\thankstext{T1}{A sample of additional note to the title.}

\begin{aug}
\author{\fnms{Bojan} \snm{Basrak}\ead[label=e1,mark]{bbasrak@math.hr}}
\and
\author{\fnms{Hrvoje} \snm{Planini\'c}\ead[label=e2,mark]{planinic@math.hr}}
 \address{Department of Mathematics, Faculty of Science, University of Zagreb\\
 Bijeni\v{c}ka cesta 30, 10000 Zagreb, Croatia\\
  \printead{e1,e2}}
\end{aug}

\begin{abstract}
The article determines the asymptotic shape of the extremal clusters in stationary regularly varying random fields. To deduce this result, we present a general framework for the Poisson approximation of point processes on Polish spaces which appears to be of independent interest. %We also demonstrate some effective methods for establishing its conditions.
		We further introduce a novel and convenient concept of anchoring of the extremal clusters  for regularly varying sequences and fields. Together with the Poissonian approximation theory, this allows for a  concise description of the limiting behavior of random fields in this setting.
		We  apply this theory  to shed entirely new light on the classical problem of evaluating local alignments of biological sequences.	
\end{abstract}


\end{frontmatter}

%%%%%%%%%%%%%%%%%%%%%%%%%%%%%%%%%%%%%%%%%%%%%%
%%%% Main text entry area:
\section{Introduction}
{Developments in the theory of stationary regularly varying sequences have broadened our understanding of several key time series models,
	see for instance \cite{basrak:segers:2009,janssen:segers:2014,mikosch:wintenberger:2016} and references therein. 
	This theory extends to regularly varying random fields in a relatively straightforward manner, the main technical difficulty being the absence of a natural ordering on the higher-dimensional integer lattice. In parallel to the one-dimensional case, the extreme values in such a random field typically exhibit local clustering.
Characterizing the limiting behavior of those extreme clusters is one of the main goals of our study. }

{In order to deal with this question, we first present a new theory of Poisson approximation for point processes on general Polish spaces which seems of independent interest. Next, we introduce a novel concept of anchoring. 
	This notion is original, and we think, illuminating and bound to be useful even in the well understood time series setting. 
	Using it, we deduce several
	results concerning  compound Poisson limit approximations for extremes of stationary regularly varying random fields.
	}

{Finally,  these methods allow us to revisit the classical problem of local sequence alignments. In particular, we give a new geometric interpretation for the asymptotic behavior of the scores in local alignments of i.i.d.\ sequences. Our main result in this context is given as Theorem~\ref{thm:PPconv_Alignments_intro} below.
}

\subsection{Regularly varying random fields}
We say that a  real-valued random field $\bY=(Y_{i,j}:i,j\in\Z)$ represents
the tail field (or the tail process) of a (strictly) stationary real-valued random field $(X_{i,j}: i,j\in\Z)$ if it appears as the limit in
\begin{equation*}	
\left( u^{-1}X _{i,j} \right)_{i,j \in \{-m,\dots,m\}} \;\big| \; |X_{0,0}|>u  \dto (Y_{i,j})_{i,j \in \{-m,\dots,m\}}\, ,
\end{equation*} 
 for every $m\in \N$  as $u \toi$. {Note that in this introduction we consider random fields indexed over the two-dimensional integer lattice, while we actually develop the theory for integer lattices of arbitrary dimension $d\in\N$.}
%The tail field (or the tail process) of a (strictly) stationary $\R$--valued random field $(X_{i,j}: i,j\in\Z)$ is an $\R$-valued random field $\bY=(Y_{i,j}:i,j\in\Z)$
%which for every $m\in \N$ and as $u \toi$ appears as the limit in
%\begin{equation}\label{eq:TailF}
%\left( u^{-1}X _{i,j} \right)_{i,j \in \{-m,\dots,m\}} \;\big| \;|X_{\bo}|>u  \dto (Y_{i,j})_{i,j \in \{-m,\dots,m\}}\, .
%\end{equation} 

{The notion of the tail process for stationary time series was introduced in \cite{basrak:segers:2009}.
% and has broaden our understanding of several key time series models,
%	see for instance\ \cite{janssen:segers:2014,mikosch:wintenberger:2016} and references therein. 
	In \Cref{sub:tail} we extend this theory to random fields. This extension is relatively straightforward but some issues arise due to the absence of a natural ordering on $\Z^2$, see \Cref{sub:exist_tail}.  As in the one-dimensional case, the existence of the tail process is equivalent to $(X_{i,j})$ being regularly varying, that is, to having all of its finite-dimensional distributions multivariate regularly varying. 
	%However, a nontrivial extension was to obtain minimal sufficient conditions for the existence of the tail process, see \Cref{thm:tail_process}.
	}

%\tcb{Consider a stationary regularly varying $\R$--valued random field $(X_{i,j}: i,j\in\Z)$, i.e.\  field whose all finite-dimensional distributions are multivariate regularly varying. This property is equivalent to the existence of the so-called tail field (or the tail process)  $\bY=(Y_{i,j}:i,j\in\Z)$ which for every $m\in \N$ and as $u \toi$ appears as the limit in
%\begin{equation}\label{eq:TailF}
%\left( u^{-1}X _{i,j} \right)_{i,j \in \{-m,\dots,m\}} \;\big| \;|X_{\bo}|>u  \dto (Y_{i,j})_{i,j \in \{-m,\dots,m\}}\, ,
%\end{equation} 
%% using the product topology on the space $\R ^{\ZZ^d}$; 
%% the tail process for time series was introduced by Basrak and Segers~\cite{basrak:segers:2009}, in 
%see Section~\ref{sec:RegVarF} below. 
%%and also Wu and Samorodnitsky~\cite{samorodnitsky:wu:2018}. 
%}

One of our main goals is to describe the limiting extremal behavior of $(X_{i,j})_ {i,j \in \{1,\dots, n\}}$ as $n\toi$ relying on the theory of point processes; cf. \Cref{sec:basic_setup} where we recall the definition of a point process on a general state space and the related notion of vague convergence. 
%For simplicity, in the rest of this subsection assume $d=2$.
{
The limiting extremal behavior can be deduced easily if $X_{i,j}$'s are i.i.d., see Resnick~\cite{resnick:1987}. On the other hand, in the general case where extreme values tend to appear in clusters,  it is often useful to decompose $(X_{i,j})_ {i,j \in \{1,\dots, n\}}$ into (smaller) blocks of size $r_n^2$ for some intermediate sequence $(r_n)_n$ such that $\lim_{n\toi} r_n =\infty$ but with $\lim_{n\toi}r_n/n =0$. } 
More precisely, define the blocks as  $r_n^2$-dimensional random vectors
\begin{equation}\label{eq:blocks_intro}
\bX_{n,\bi}:=(X_{i,j}: (i,j) \in J_{n,\bi}) \, ,
\end{equation}
for $\bi=(i_1,i_2) \in I_n :=\{1,\dots,k_n\}^2$ where $k_n=\lfloor n / r_{n} \rfloor$  and
\begin{equation}\label{eq:blocks_indices_intro}
J_{n, \bi}=\{(i_1-1)r_n+1,\dots, i_1 r_n\}\times\{(i_2-1)r_n+1,\dots, i_2 r_n\} \, . 
\end{equation}
%For every $\bi \in I_n$ we will consider $\bi r_n$, which is the upper-right end index in $J_{n,\bi}$, as the position of the block $\bX_{n,\bi}$. 
%We note that any other (even random) index from $J_{n,\bi}$ could be used instead.
 
%%%% or just
%These blocks are denoted by $\bX_{n, \bi}$, $\bi \in I_n :=\{1,\dots,k_n\}^2$, where $k_n^2=\lfloor n / r_{n} \rfloor^2$ is the total number of blocks see \Cref{sub:convToComp} for details.
%Following \cite{basrak:planinic:soulier:2018}, 
{One can add zeros around these blocks and consider them as elements of the (infinite-dimensional) space of all arrays $(x_{i,j})_{i,j\in \Z} \in \R^{\Z^2}$ which vanish to 0 in all directions  but where we,  for technical reasons explained in \Cref{rem:why_lo}, do not distinguish between arrays which are equal up to a shift. This space is denoted by $\lo$ and can be seen as a quotient space, see \Cref{subsub:lo} for a precise definition where we also endow
 $\lo$ with the metric generated by the norm $\|(x_{i,j})_{i,j}\|=\max_{i,j} |x_{i,j}|$.
% Note that by the shift-equivalence property this embedding is unambiguous. 
%\tcr{On $\lo$ we consider the metric generated by the norm $\|(x_{i,j})_{i,j}\|=\max_{i,j} |x_{i,j}|$.}

 {In \Cref{thm:PPconv} we show that under some standard weak dependence conditions on the field $(X_{i,j})$ and for a sequence of positive numbers $(a_n)_n$ satisfying $
\lim_{n\toi} n^2\pr(|X_{0,0}|>a_n)= 1$, 
\begin{align}\label{eq:PPconvinLo_intro}
\sum_{\bi\in I_n} \delta_{\left(\bi /k_n,\bX_{n,\bi}/a_n\right)} \dto\sum_{k\in \N}\delta_{\left(\bT_k, P_k (Q^k_{i,j})_{i,j\in \Z}\right)} \, , \, n\toi \, ,
\end{align}  
in the space of point measures on $[0,1]^2 \times (\lo\setminus\{\bo\})$ where {$\bo$ is the array consisting only of 0's,} and 
\begin{enumerate}[(i)]
	\item $\sum_{k\in \N}\delta_{(\bT_k,P_k)}$ is a Poisson point process on $[0,1]^2\times(0,\infty)$
	with intensity measure $\vartheta d\bt \times \alpha y^{-\alpha-1}dy$ for some constant $\vartheta>0$; 
	%$\mbox{Leb}$ denotes the Lebesgue measure on $\R^2$ and $d(-y^{-\alpha})$ the measure on $(0,\infty)$ with density $\alpha y^{-\alpha-1}dy$.
	\item $(Q^k_{i,j})_{i,j\in \Z}, \: k\in \N$ is a sequence of i.i.d.\ random fields independent of
	$\sum_{k\in \N}\delta_{(\bT_k, P_k)}$. 
	%in particular satisfying $\|(Q^k_{i,j})_{i,j}\|=1$ for all $k\in \N$.
\end{enumerate}	
%  Then the point processes $N_n' = \sum_{\bi\in I_n} \delta_{(\bi/k_n,\bX_{n,\bi})}$ converge in distribution in $\mathcal{M}_p([0,1]^d \times \lo\setminus\{\bo\})$
%to a Poisson point process with intensity measure $Leb\times\nu$. Moreover, the limit has the same distribution as the point process
%\begin{equation}\label{eq:PPPrepr}
%N'=\sum_{i\in \N}\delta_{(\bT_i, P_i(Q^i_{\bj})_{\bj} )} 
%\end{equation}
%  where

As usual, the vague topology used in (\ref{eq:PPconvinLo_intro}) controls only the blocks $\bX_{n,\bi}$ whose maximal value $\|\bX_{n,\bi}\|$ exceeds a  sufficiently high threshold, see \Cref{sec:basic_setup} and \Cref{sub:convToComp} for the technical details. 
	For a schematic represention of the limit in (\ref{eq:PPconvinLo_intro}) on a particular class of regularly varying fields see the right side of \Cref{fig:SeqAli1} and  the discussion after \Cref{thm:PPconv_Alignments_intro}. Note that the spatial location of the block $\bX_{n,\bi}$ in (\ref{eq:PPconvinLo_intro}) satisfies $\bi / k_n \approx \bi r_n /n$ for large $n$ with $\bi r_n$ being the upper-right end index in $J_{n,\bi}$ from (\ref{eq:blocks_indices_intro}).

%The topology used for this convergence is the standard vague topology but with bounded sets in $[0,1]^2 \times (\lo\setminus\{\bo\})$ being those which are bounded away from the set $[0,1]^2\times \{\bo\}$ with $\bo$ denoting the zero array; again, see Section \ref{sec:basic_setup} for details. Effectively, this means that convergence in (\ref{eq:PPconvinLo_intro}) controls only blocks $\bX_{n,\bi}$ whose maximal value $\|\bX_{n,\bi}\|$ is large. }

%Detailed description of this convergence result together with 

{In the time series setting, the limit in (\ref{eq:PPconvinLo_intro}) appeared already in \cite[Theorem 3.6]{basrak:planinic:soulier:2018}. The novelty of our paper in this context is twofold. First, the link between the tail process $\bY$ and the key ingredients of the limit in (\ref{eq:PPconvinLo_intro}), constant $\vartheta$ 
%(which represents the extremal index of the field $\left(\, |X_{\bj}| \,\right)_{\bj \in \Z^d}$) 
and the distribution of $(Q^k_{i,j})_{i,j\in \Z}$, is described in detail  using the {novel} notion of anchoring, see \Cref{subs:anchor}. {We think that this notion sheds new light even on known results in the time series setting.}  
Second, we show that the convergence in (\ref{eq:PPconvinLo_intro}) can be seen in the light of the classical Poisson convergence principle going back to Grigelionis. For that purpose, in \Cref{sec:ComPoApp} we present  a general Poissonian approximation theorem for point processes on Polish spaces constructed from points which satisfy a suitable asymptotic (in)dependence condition. 
Moreover, we give sufficient conditions for this theorem to hold in the spirit of~\cite{arratia:goldstein:gordon:1989}.
These results seem to be of independent interest and  related to those obtained by Schuhmacher~\cite{schuhmacher:2005} using  the Chen-Stein method. We,  however, rely on the Laplace functionals of point processes.}

{Finally,  
	the continuous mapping theorem and (\ref{eq:PPconvinLo_intro}) jointly  yield
\begin{align}
\label{eq:ppconvinR_intro}
\sum_{i,j=1}^n \delta_{((i,j) / n , X_{i,j}/a_n)} \dto \sum_{k\in \N} \sum_{i,j\in \Z}\delta_{\left(\bT_k, P_k Q^k_{i,j} \right) } \, , \, n\toi \, , 
\end{align}
 in the simpler (and more familiar) space of point measures on $[0,1]^2 \times (\R \setminus\{0\})$, see \Cref{cor:PPconv_back_to_R}. Observe that the limit in \eqref{eq:ppconvinR_intro} has a form of  a Poisson cluster (or a compound Poisson) process.}

\subsection{Local sequence alignment}\label{sec:locSeqAlig_intro}

{Because of its importance  in molecular biology, 
 the local alignment problem  was studied extensively both  from a probabilistic and applied perspective,} see for instance \cite{arratia:goldstein:gordon:1989,dembo:karlin:zeitouni:1994,hansen:2006} and references therein. Since it represents one of the main motivations for our study,  we  explain here its key ingredients and our main result in that context.

Let $(A_i)_{i\in\N}$ and $(B_i)_{i\in\N}$ be two independent i.i.d.\ sequences taking values in a finite alphabet $E$. Also, let $A$ and $B$ be independent random variables distributed as $A_1$ and $B_1$, respectively. 
For a fixed score function $\s:E\times E\to \R$ and for all $i,j\in \N$ and $m=0,1, \dots ,i\wedge j$ (where $i\wedge j:=\min\{i,j\}$), let 
\[
S_{i,j}^m=\sum_{k=0}^{m-1} \s(A_{i-k},B_{j-k}) 
%, \; i,j \in\Z, \; m\in \N \; 
\, 
\]
 be the score of aligning segments $A_{i-m+1},\dots, A_{i}$ and $B_{j-m+1},\dots,B_j$.
Further, for all $i,j\in \N$ define
\begin{align}\label{eq:original_scores}
S_{i,j}=\max\{ S_{i,j}^m : 0\leq m \leq i\wedge j\} \, . 
\end{align}
From a biological perspective it is essential to understand the extremal distributional properties of the random matrix $(S_{i,j}: 1\leq i,j \leq n)$ as $n\toi$. The following simple assumption is standard in this context, cf. Dembo et al.~\cite{dembo:karlin:zeitouni:1994}. 
\begin{hypothesis}\label{hypo:negative_drift}
The distribution of $\s(A,B)$ is nonlattice, i.e.\ $\pr(s(A,B)\in \delta \Z)<1$ for all $\delta>0$,  and satisfying
\begin{equation}\label{eq:negative_drift}
\ex[\s(A,B)]<0 \quad \text{ and } \quad \pr(\s(A,B)>0)>0 \, .
\end{equation}
\end{hypothesis}

The lattice case is excluded for simplicity in the sequel. It is known to be conceptually similar, although technically more involved. Note further that, like~\cite{dembo:karlin:zeitouni:1994} and~\cite{hansen:2006}, we consider only gapless local alignments.

%In order to obtain the asymptotic result, we will impose the condition (E') of Dembo et al.~\cite{dembo:karlin:zeitouni:1994}, which is essentially an assumption on the extremal dependence structure of the field $(S_{i,j})$. But first, we need to introduce some additional notation.

Denote by $\mu_A$ and $\mu_B$ the distributions of $A$ and $B$, respectively and assume
 for simplicity that $\mu_A(e),\,\mu_B(e) >0$ for each letter $e$ in the alphabet $E$. By
Assumption~\ref{hypo:negative_drift} there exists a
 unique strictly positive solution $\ts$  of the Lundberg equation
\begin{equation*}%\label{eq:cgf}
%\kappa(\ts):=\log \ex[e^{\ts s (A,B)}]=0 \,.
m(\ts):=\ex[e^{\ts s (A,B)}]=1 \,.
\end{equation*}

Let $\mus$ be the (exponentially tilted) probability measure on $E\times E$ given by 
\begin{align}\label{eq:tilted_measure}
\mus(a,b)=e^{\ts s(a,b)} \mu_A(a)\mu_B(b)\;, \; a,b\in E \, .
\end{align}
 
For two probability measures $\mu$ and $\nu$ on a finite set $F$, denote by $H(\nu|\mu)$ the relative entropy of $\nu$ with respect to $\mu$, i.e.
\[
H(\nu|\mu)=\sum_{x\in F} \nu(x) \log \frac{\nu(x)}{\mu(x)} \, .
\]

Dembo et al.~\cite{dembo:karlin:zeitouni:1994} introduce one final condition on the tilted probability measure $\mus$.
%It essentially restricts extremal dependence within the field $(S_{i,j})$ in a way which seems biologically  meaningful and  exactly suits  their, as well as our, asymptotic analysis of the field.
\begin{hypothesis}[Condition (E') in \cite{dembo:karlin:zeitouni:1994}]\label{hypo:E'}
It holds that
\begin{equation}\label{eq:E'}
H(\mus | \mu_A\times \mu_B)> 2\left\{H(\mus_A | \mu_A) \vee H(\mus_B | \mu_B) \right\} ,
\end{equation}
where $\mus_A$ and $\mus_B$ denote the marginals of $\mus$.
\end{hypothesis}

Note that (\ref{eq:E'}) holds automatically if $\mu_A=\mu_B$ and if the score function $s$ is symmetric (i.e.\ $s(a,b)=s(b,a)$) but not of the form $s(a,b)=s(a)+s(b)$,
%\tcb{(i.e.\ $s(a,b)$ really depends on the closeness of the letters $a$ and $b$)}
 see \cite[Section 3]{dembo:karlin:zeitouni:1994a}.

Under Assumptions \ref{hypo:negative_drift} and \ref{hypo:E'}, Dembo et al.~\cite{dembo:karlin:zeitouni:1994} (see also Hansen~\cite{hansen:2006}) showed that the distribution of the maximal local alignment score
$
M_n=\max_{1\leq i,j \leq n} S_{i,j} \, ,
$
asymptotically follows a Gumbel distribution. More precisely, as $n\toi$, for a certain constant $K^*>0$,
\begin{equation} \label{eq:Gumb}
\pr\left(M_n - \frac{2\log(n)}{\ts}\leq x\right)\to e^{-K^* e^{-\ts x}} \, , \, x\in \R \, .
\end{equation}

{Observe that the field $(S_{i,j})$ consists of dependent random variables. For instance, simple arguments can be given (cf.\ \eqref{eq:Lind} below) showing that any extreme score, i.e.\ score exceeding a given large threshold, will be followed by a run of extreme scores along the diagonal. This phenomenon is illustrated in Figure~\ref{fig:SeqAli1}.
% for both real life and simulated sequences.
 The approach of \cite{dembo:karlin:zeitouni:1994} is based on showing that the number of such extreme clusters, as both the sample size and the threshold tend to infinity, becomes asymptotically Poisson distributed. }

%The asymptotic behavior of $M_n$ is then just a simple corollary. 
%\hrnote{U recenziji su pitali da kazemo koji je zakljucak ovog grafa. Mozda mozemo ostaviti samo simulirani slucaj?}
%\begin{figure}[h]\label{fig:SeqAli1}
%\centering
%		\includegraphics[height=6cm]{SeqAli34c.pdf}
%	\caption{Heatmap of large local scores for alignments of two simulated sequences (left) and  two regions of human and fruit-fly genome (right). Each sequence is 1000 nucleotides long.}
%\end{figure}
\begin{figure}[h]
\centering
		\includegraphics[scale=0.5]{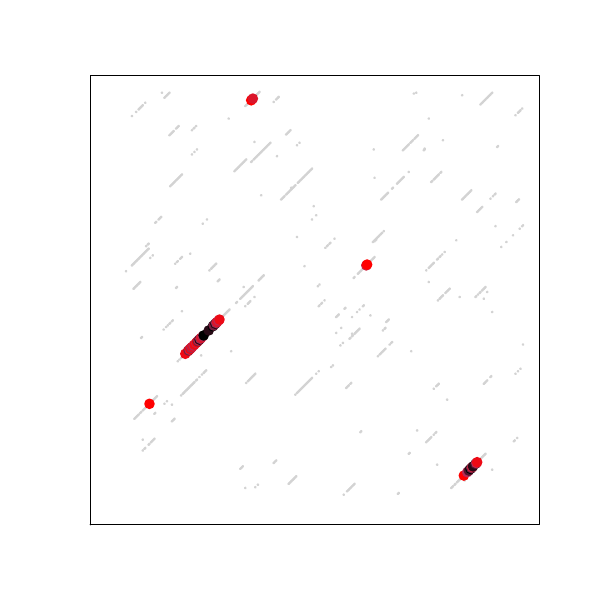}
				\includegraphics[scale=0.5]{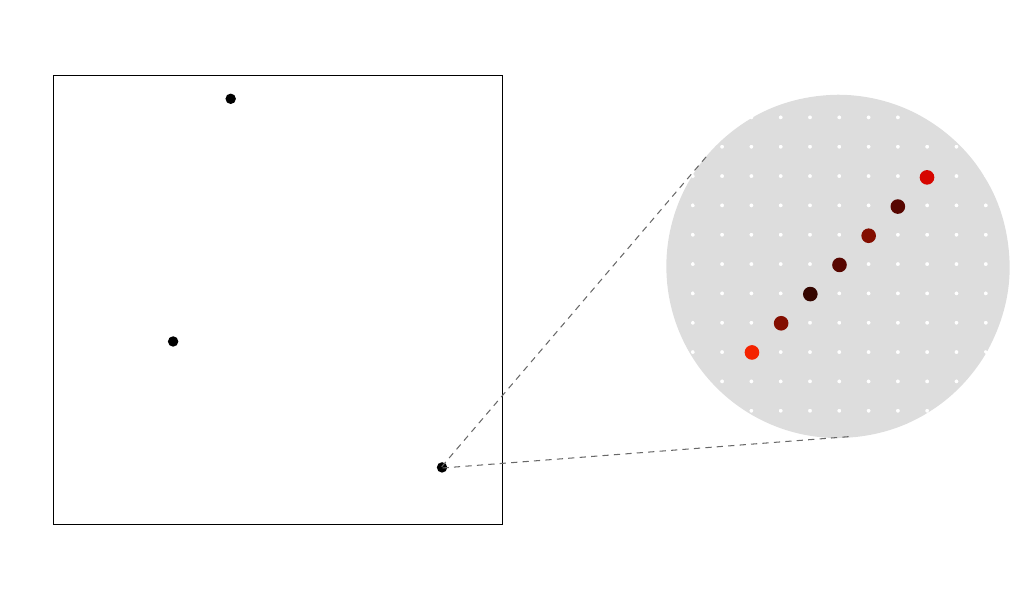}
	\caption{Heatmap of the local scores $S_{i,j}$, $i,j=1,\dots, n$, exceeding a prespecified threshold for two simulated sequences of length $n=500$ (on the left).
			Schematic representation of the limit as $n \toi$ with the clusters of values above the given threshold collapsing to a single point  marked with the corresponding tail field (on the right), see the discussion after \Cref{thm:PPconv_Alignments_intro}. }
	\label{fig:SeqAli1}
\end{figure}

In the sequel,  we show that one can give a much more detailed information about the structure within the extreme clusters. 
In particular, following the method below one can deduce the asymptotic distribution of arbitrary functionals of the upper order statistics of the field $(S_{i,j})$.

Observe first that for each $i,j\in \N$, $S_{i,j}$ can be seen as the maximum of a truncated random walk $(S_{i,j}^m)_{m=0,\ldots , i \wedge j}$ which by
 (\ref{eq:negative_drift}) has negative drift.
 It can be rigorously shown, see Remark \ref{rem:originalScores}, that in all our asymptotic considerations this truncation and the related edge effects can be ignored. Therefore we 
 assume throughout that the sequences $(A_i)$ and $(B_i)$ extend over all integers $i \in \ZZ$. This makes scores $S_{i,j}^m$  well defined for all $i,j\in\Z$ and $m\geq 0$, and consequently we update the original field of scores $(S_{i,j})$ as  follows
% 
% 
%  Since, intuitively speaking, the probability that such a random walk  exceeds a big treshold for large $m$ is small, the asymptotic behavior of $M_n$ will be identical to the behavior of the maximum score with original scores $S_{i,j}$ replaced by
\begin{align}\label{eq:stationary_scores}
S_{i,j}=\sup\{ S_{i,j}^m : m\geq 0 \} \, , \, i,j \in \Z \, .  
\end{align}

By construction, the field $(S_{i,j})$ is stationary. Moreover, by the classical Cram\'er-Lundberg theory, Assumption \ref{hypo:negative_drift} implies that the tail of $S_{i,j}$ is asymptotically exponential, or more precisely
\begin{equation}\label{eq:exp_tail}
\pr(S_{i,j}>u)\sim \tailcons e^{-\ts u} \, , \, \text{as } \, u\toi\, ,
\end{equation}
for some  $\tailcons>0$.
% and $\ts>0$ being the unique positive constant satisfying the Lundberg equation
%\begin{equation}\label{eq:cgf}
%\kappa(\ts):=\log \ex[e^{\ts s(A,B)}]=0 \, .
%\end{equation}
Note that, in the language of extreme value theory, marginal distribution of the field $(S_{i,j})$ belongs to the maximum domain of attraction of the Gumbel distribution. In this light, the limiting result  \eqref{eq:Gumb} may  not be very surprising, but its proof  remains quite involved  due to the clustering of extremal scores of the field $(S_{i,j})$. Observe that the field $(S_{i,j})_{i,j\in \Z}$ satisfies the following simple (Lindley) recursion along any diagonal, namely
\begin{equation} \label{eq:Lind}
  S_{i,j} = \left( S_{i-1, j-1} + \bz _{i,j} \right)_+\,,
\end{equation}
where random variables $\bz_{i,j} = s(A_i,B_j)$ have negative mean.
%where we have assumed that sequences $(A_i)$ and $(B_i)$ are defined over all integers so the scores $S_{i,j}^m$ are well defined for all $m\in \N$.
%This heuristics will be made precise later.

{Our main result in this context strengthens \eqref{eq:Gumb} to a convergence in distribution of point processes based on the $S_{i,j}$'s. 
The key observation is that under Assumptions \ref{hypo:negative_drift} and \ref{hypo:E'} the transformed field 
\begin{align*}
X_{i,j}=e^{S_{i,j}} \, , \, i,j\in\Z
\end{align*} 
admits a tail process $(Y_{i,j} : i,j\in \Z)$, hence it is regularly varying; see \Cref{prop:multivariate_reg_var}. Its tail process satisfies 
$$Y_{i,j}=0 \,  , \; i\not = j \, .$$
%as $u \toi$,
%\begin{equation}\label{eq:SpTailFexpS}
%\left( {e^{S_{i,j}}}/{ e^{S_{0,0}}}  : (i,j) \in \ZZ^2 \right) \;\big| \; S_{0,0} >u  \dto (\Theta_{i,j})\,,
%\end{equation}
%for some field $(\Theta_{i,j})_{(i,j) \in \ZZ^2}$
%with $\Theta_{i,j} = 0$ for all $i\not = j$. 
Moreover, the distribution of $Y_{m,m}$'s can be described in detail using two auxiliary independent i.i.d.\ sequences
$ (\bz_i)_{i \geq 1}$ and $ (\bz^*_i)_{i \geq 1}$ whose distributions correspond to
the distributions of $s(A,B)$ under the product measure $\mu_A \times \mu_B$ and 
under the tilted measure $\mus$ from (\ref{eq:tilted_measure}), respectively:
if $S^\bz _ 0 = 0$ and
\begin{align*}
S^\bz_m=\begin{cases}
\phantom{-} \sum_{i=1}^m \bz_i \,, &  m \geq 1 \, , \\
 - \sum_{i=1}^{-m}  \bz^*_i \,,  &  m \leq -1 \, ,
\end{cases}
%%\label{eq:defOfSm}
%   S^\bz_m &= \sum_{i=1}^m \bz_i\,,  m \geq 1 \\   
%   S^\bz_m &= - \sum_{i=1}^{-m}  \bz^*_i\,,  m \leq -1\,,
\end{align*}
then 
$$Y_{m,m} = Y_{0,0} e^{S^\bz_m} \, , \; m \in \ZZ \, , $$  
where $Y_{0,0}$ is Pareto distributed with index $\ts$, i.e.\ $\pr(Y_{0,0}>y)=y^{-\ts}$ for all $y\geq 1$, and independent of $(S^\bz_m)_m$. }
{
To state our main result denote by $\Theta_{i,j}=Y_{i,j}/Y_{0,0}$, $i,j\in \Z$, the so-called spectral tail field of $(X_{i,j})$, so that 
\begin{equation}\label{eq:align_spectralPr_intro}
\Theta_{m,m}=e^{S^\bz_m} \mbox{ for } m \in \ZZ\,, \quad\mbox{ and } \quad  \Theta_{i,j}=0 \mbox{ for } i\neq j\,.
\end{equation}
Take an arbitrary sequence of positive integers $(r_n)$ such that $\lim_{n\toi} r_n = \infty$ and $\lim_{n\toi}r_n/n^\epsilon\to 0$ for all $\epsilon>0$ and recall the blocks $\bX_{n,\bi}$ defined in (\ref{eq:blocks_intro}).}
{\begin{theorem}
\label{thm:PPconv_Alignments_intro}
Under Assumptions \ref{hypo:negative_drift} and \ref{hypo:E'}, 
%for any sequence of positive integers $(r_n)$ such that $\lim_{n\toi} r_n = \infty$ and $\lim_{n\toi}r_n/n^\epsilon\to 0$ for all $\epsilon>0$,
\begin{align}\label{eq:PPconv_lo_alignm_intro}
\sum_{\bi\in I_n} \delta_{\left(\bi /k_n, \, \bX_{n,\bi}/n^{2/\ts}\right)} \dto \sum_{k\in \N}\delta_{\left(\bT_k, \, P_k(Q^k_{i,j})_{i,j \in \Z} \right)} 
\end{align}  
in the space of point measures on $[0,1]^2 \times (\lo\setminus\{\bo\})$ where 
\begin{enumerate}[(i)]
  \item $\sum_{k\in \N}\delta_{(\bT_k,{P}_k)}$ is a Poisson point process on $[0,1]^2\times (0,\infty)$
   with intensity measure $\vartheta \tailcons d\bt \times \ts y^{-\ts-1} dy$ {where $C$ is the constant from (\ref{eq:exp_tail})} and 
   \begin{align*}
   \vartheta=\PP ( \sup_{m \geq 1}S^\bz_m + \Gamma  \leq  0) \, ,
   \end{align*}
   for an exponential random variable $\Gamma$ with parameter $\ts$ independent of $(S_m^{\bz})$;
%   \[
%   \vartheta=\int_{-\infty}^0 \pr\left(\sup_{m \geq 1}S^\bz_m \leq z \right)\ts e^{\ts z}dz \in (0,1] \; ; 
%   \]
  \item $({Q}^{k}_{i,j})_{i,j\in \Z}, \: k\in \N$ are i.i.d.\ random fields independent of
    $\sum_{k\in \N}\delta_{(\bT_k,{P}_k)}$ and with common distribution equal to the distribution of $(\Theta_{i,j})_{i,j\in \Z}$ 
    in \eqref{eq:align_spectralPr_intro}, but conditionally on the underlying random walk $(S^\bz_m)_m$ being negative for $m<0$ and nonpositive for $m>0$.
  \end{enumerate}
\end{theorem}}

 An interpretation of the theorem can be given through Figure~\ref{fig:SeqAli1}. On the left, we plot the  scores exceeding a prespecified threshold for two simulated independent sequences of length $n=500$ from the uniform distribution on a four letter alphabet. The grey dots correspond to the scores exceeding 50\% of the maximal score $M_n$, while the other dots represent points over 75\%$M_n$
 (they are colored from red to black, with the darker color indicating a higher score). {In this simulation, for illustration purposes, we  score a  match by $\sqrt 3$ and a mismatch by $-1$.} 
The picture on the right schematically illustrates the limit of the leading clusters of (exponentially transformed) high scores grouped into blocks which, after a rescaling, collapse to a single point (at position $T_k$ say) which is then marked by its maximum and
 the shape of the cluster (denoted by $P_k$ and $(Q^{k}_{i,j})_{i,j\in \Z}$ say), see also the discussion after \Cref{rem:explanation}. In this case, the random fields $(Q^{k}_{i,j})$ are concentrated on the diagonal because of  \eqref{eq:align_spectralPr_intro}.

 Taking logarithms, from (\ref{eq:PPconv_lo_alignm_intro}) one can deduce the convergence
 \begin{align}\label{eq:StildePQ}
 \sum_{i,j=1}^n \delta_{\left(\tfrac{(i,j)}{n},\: S_{i,j}- \tfrac{2\log(n)}{\ts}\right)} 
\dto 
\sum_{k\in \N} \sum_{m\in \Z} \delta_{(\bT_k, \log(P_k) + \log(Q^{k}_{m,m})  )} \, 
 \end{align}
in the space of point measures on $[0,1]^2 \times \R$ with a suitable vague topology, see  \Cref{cor:sequence_alignment} for details.
%\tcb{Taking logarithms, from (\ref{eq:PPconv_lo_alignm_intro}) one can deduce the convergence
%\[
%\sum_{i,j=1}^n \delta_{\left(\tfrac{(i,j)}{n},\: S_{i,j}- \tfrac{2\log(n)}{\ts}\right)} 
%\dto 
%\sum_{k\in \N} \sum_{m\in \Z} \delta_{(\bT_k, \tilde{P}_k + \tilde{Q}^{k}_m  )} \, 
%\] 
%in the space of point measures on $[0,1]^2 \times \R$ with a suitable vague topology,
% where
%  \begin{enumerate}[(i)]
%  \item $\sum_{k\in \N}\delta_{(\bT_k,\tilde{P}_k)}$ is a Poisson point process on $[0,1]^2\times \R$
%   with intensity measure $\vartheta \tailcons \mbox{Leb} \times \ts e^{-\ts u} du$; 
%%   where, for an exponential random variable $\Gamma$ with parameter $\ts$ independent of $(S_m^{\bz})$,
%%   \begin{align*}
%%   \vartheta=\PP ( \sup_{m \geq 1}S^\bz_m + \Gamma  \leq  0) \, ;
%%   \end{align*}
%%   \[
%%   \vartheta=\int_{-\infty}^0 \pr\left(\sup_{m \geq 1}S^\bz_m \leq z \right)\ts e^{\ts z}dz \in (0,1] \; ; 
%%   \]
%  \item $(\tilde{Q}^{k}_{m})_{m\in \Z}, \: k\in \N$ are i.i.d.\ two-sided $\R$-valued sequences, independent of
%    $\sum_{k\in \N}\delta_{(\bT_k,\tilde{P}_k)}$ and with common distribution equal to the distribution of the random walk $(S^\bz_m)_m$ conditioned on staying negative for $m<0$ and nonpositive for $m>0$.
%  \end{enumerate}
%See  \Cref{cor:sequence_alignment} for details.}\tcr{Too much?}  
In particular, 
%since 
%$$\pr(\sum_{k\in \N} \1{\{\log(P_k)>x\}}=0)=e^{-\vartheta C e^{-\ts x}}$$
this 
 yields \eqref{eq:Gumb} at once with the following new expression for the key constant therein
 $$
  K^* = \vartheta \tailcons \, .
 $$
 Note that $\vartheta$ is the so-called extremal index of the field $(S_{i,j})$, cf. \Cref{rem:extremal_index}.
 %Note that $\vartheta$ is the so-called extremal index of the field $(S_{i,j})$, 
 The same expression for $\vartheta$ appears in a different context in de Haan et al.~\cite[Section 3]{dehaan:resnick:rootzen:devries:1989} together with a suggested algorithm for its numerical computation.
 Moreover, the constant $\tailcons$ arising from \eqref{eq:exp_tail} is frequently encountered in the literature;  for various expressions 
of $C$ we refer to  \cite[Part C, XIII.5]{asmussen:2003}. Thus, in principle, for i.i.d. sequences (as in Altschul et al.~\cite{altschul:2001} for instance) the
 constants $K^*$ and $\alpha^*$ in  \eqref{eq:Gumb} do not have to be estimated since they can be directly determined from the marginal distribution of the letters and the scoring function $s$.
Note also that the distribution of random walks conditioned to stay negative (or positive) is discussed in detail by Tanaka~\cite{tanaka:1989} and Biggins~\cite{biggins:2003}.} 

Finally, \Cref{thm:PPconv_Alignments_intro} has some specific implications for the interpretation of real biological sequence alignments. First of all,
observe that the number of $\log(P_k)$'s above a given threshold $x$ in (\ref{eq:StildePQ}) is Poisson distributed, while the overshoots of
$\log(P_k) -x$ are i.i.d.\ and have an exponential distribution.
This fact gives a theoretical underpinning to the use of the peaks-over-a-threshold approach to the modeling of local alignments  in which the number of clusters (islands) of scores above a high threshold is modeled by  a Poisson random variable and where the local extremes of these clusters exceed a given threshold by a random amounts which are independent and exponentially distributed. For an application of this idea in two different contexts see
 Altschul et al.~\cite{altschul:2001} and Hansen~\cite{hansen:2009}.
%Finally, \Cref{thm:PPconv_Alignments_intro} has some specific implications for the interpretation of real biological sequence alignments. First of all, it gives a theoretical underpinning to the use of the peaks--over--a--threshold approach for the modeling of local alignments (see Altschul et al.~\cite{altschul:2001} and Hansen~\cite{hansen:2009}),  in which the number of clusters (islands) of scores above a high threshold is modeled by  a Poisson \tcb{random variable} and where the local extremes \tcb{of these clusters} exceed a given threshold by a random amount which is approximately exponentially distributed.
 Moreover, if one connects the $k$ leading nonoverlapping clusters of high scores in the direction of the alignment,  one can incorporate gaps into the alignment
 and approximate $p$-values of such extended and possibly penalized local alignments
  (this would go into the direction of  Siegmund and Yakir~\cite{siegmund:yakir:2000,siegmund:yakir:2003}, cf. also  Metzler et al. \cite{metzler:grossman:wakolbinger:2002} where our deduced limit is simply assumed). Finally, zooming in into individual clusters, the theorem allows one to study the structure of subsequences 
  $(A_{i-k},\ldots, A_i)$  and  $(B_{j-k},\ldots, B_j)$ in a cluster of high scores,
  to see if it agrees with the predicted theoretical distribution of such a cluster given a very close alignment.  Each of these issues arguably deserves a detailed study and a real--life data illustration, but that would exceed the scope of our paper. }

% The distribution of random walks conditioned to stay negative (or positive) is discussed in detail by Tanaka~\cite{tanaka:1989} and Biggins~\cite{biggins:2003}. Consequently, one can apply those results to simulate and precisely describe
%the distribution of  $\tilde{Q}^{k}$'s.
%Putting all these ingredients together, one can use  \Cref{thm:sequence_alignment} to give a  probabilistic and geometric
%	intepretation of the plot in   Figure~\ref{fig:SeqAli1}. Note that this type of limiting behaviour was conjectured already by Metzler et al. \cite{metzler:grossman:wakolbinger:2002} who suggested a marked Poissonian model of local alignments with essentially {the} same features.

%\tcb{U gornjoj vezi s POT metodom, tu (uz logaritmiranje) mislis na konvergenciju
%\begin{align*}
%\sum_{\bi\in I_n} \delta_{M_{n,\bi}/a_n \epsilon} 1_{M_{n,\bi}>a_n\epsilon}\dto \sum_{i=1}^K \delta_{Z_i}
%\end{align*}
%gdje su $M_{n,\bi}$ maksimumi blokova, a $K\sim \mbox{Poiss}(\vartheta C \epsilon^{-\ts})$ nezavisna od njd niza $(Z_i)_i$ t.d.\ $Z_i\sim \mbox{Pareto}(\ts)$ ?
%}

\subsection{Organization of the paper}

{The rest of the article is organized as follows --- in Section \ref{sec:ComPoApp}, 
we present a general type of a Poissonian approximation theorem which allows one to study point processes constructed from general random fields with values in a Polish space under an appropriate dependence assumption. We also find sufficient conditions for such a dependence assumption to hold.
%using the classical theorem of Grigelionis (see \cite[Corollary 4.25]{kallenberg:2017}), we present  a general Poissonian approximation theorem for point processes on Polish spaces based on points which satisfy a suitable asymptotic (in)dependence condition. Also, we give sufficient conditions in the spirit of~\cite{arratia:goldstein:gordon:1989} under which such dependence assumption is satisfied. 
%These results are similiar to those obtained by Schuhmacher~\cite{schuhmacher:2005}. However, while \cite{schuhmacher:2005} uses the Stein's method, our approach is based on exploiting the multiplicative strucuture of Laplace functionals of point processes. Ideas similar to ours can be found already in Banys~\cite{banys:1980}.
%as well as Banys \cite{banys:1980} and Arratia et al. \cite{arratia:goldstein:gordon:1989}, we present a general type of Poissonian approximation theorem, which allows one to study point processes constructed from general random fields with values in a Polish space under appropriate dependence assumption. We also find sufficient conditions for such a dependence assumption to hold.
	Section \ref{sec:RegVarF} presents the point process convergence theory for stationary regularly varying random fields indexed over $\Z^d$ with $d\in \N$, complementing and extending the theory from the case $d=1$. In particular, we introduce the  notion of the tail field/process and point out at the subtleties of this  extension arising from the fact that  there is no unique natural ordering of the points in the $d$-dimensional lattice, for $d \geq 2$. Moreover, a special attention is dedicated to the notion of anchoring which clarifies the link between the tail process and the components $\vartheta$ and $(Q^k_{i,j})$ of the limiting point process from (\ref{eq:PPconvinLo_intro}).
%of the cluster of extremes, i.e.\ choosing a reference point in the cluster of large exceedances. Here we show that the choice of the anchoring point 
%	%(or the choice of group order in $\Z^d$) 
%	does not influence  our main results.
%%As an important class of examples we study stationary fields admitting $m$--dependent approximation, see Section \ref{subs:m-depApprox}.	
	Section \ref{sec:PrThmAl} is entirely dedicated to the alignment problem and the proof of Theorem~\ref{thm:PPconv_Alignments_intro}. 
	%{In Section \ref{subs:m-depApprox} we study stationary fields which are approximable by  $m$-dependent regularly varying fields.} 
	 Finally, in Section \ref{sec:App}
	we give the proofs of \Cref{thm:tail_process} from Section~\ref{sec:RegVarF} and several auxiliary results used in  Section \ref{sec:PrThmAl}\,.} Some proofs and arguments which 
are straigthforward generalizations of the existing results can be found in \cite{thesis}.

\section{On (compound) Poisson approximation in general Polish spaces} \label{sec:ComPoApp}

For the general theory of point processes on Polish spaces and the so-called vague convergence  see e.g.\ Kallenberg~\cite{kallenberg:2017} or Resnick~\cite{resnick:1987}. Note that even though the latter reference considers only point processes on a locally compact state space,  most of the results transfer directly to the general Polish case. However, as proposed in \cite{basrak:planinic:2019}, we use a slight modification of the definition of vague convergence.
%\tcr{needed?: Our main motivation for studying point processes on a possibly infinite-dimensional space (and also the reason for the "compound" in the title) actually comes from the problem of obtaining a compound Poisson or Poisson cluster limit for point processes based on a large class of stationary random fields, see~\Cref{sec:RegVarF} below.}

\subsection{Basic setup and the notion of vague convergence}\label{sec:basic_setup}
Let $\x$ be a Polish space. Denote by $\borel(\x)$ the Borel $\sigma$-field on $\x$ and choose a subfamily $\bounded(\x)\subseteq \borel(\x)$ of sets, called \textit{bounded} (Borel) sets of $\x$. When there is no fear of confusion, we will simply write $\borel$ and $\bounded$. We say that a Borel measure $\mu$ on $\x$ is \textit{locally} (or \textit{boundedly}) finite if $\mu(B)<\infty$ for all $B\in \bbou$. The space of all such measures is denoted by $\mx=\mathcal{M}(\mathbb{X}, \bbou)$. 

For measures $\mu,\mu_1,\mu_2,\ldots \in \mx$, we say that $\mu_n$ converge \textit{vaguely} to $\mu$  and denote this by $\mu_n\vto \mu$, if as $n\toi$,
\begin{equation*}%\label{eq:vague_convergence}
\mu_n(f)=\int f d\mu_n \to \int f d\mu= \mu(f) \, ,
\end{equation*}  
for all bounded and continuous real-valued functions $f$ on $\x$ with support being a bounded set. Denote by $\cbb$ the family of all such functions and by $\cbbp$ the subset of all nonnegative functions in $\cbb$.

In the sequel we assume that the family of bounded sets $\bbou$ satisfies the following properties and in that case say that $\bbou$ \textit{properly localizes} $\x$.
\begin{enumerate}[(i)]
\item $A\sub B \in \bounded$ for a Borel set $A\sub \x$ implies $A\in \bounded$, and $A,B\in \bounded$ implies $A\cup B\in \bounded$.
\item For each $B \in \bounded$ there exists an open set $U\in \bounded$  such that $\cl{B} \subseteq U$, where  $\cl{B}$ denotes the closure of $B$ in $\x$. 
\item There exists a sequence  $(K_m)_{m\in \N}$ of bounded Borel sets which cover $\x$ and such that every $B\in \bounded$ is contained in $K_m$ for some $m\in \N$.
\end{enumerate}

Moreover, the sequence $(K_m)_{m\in \N}$ can always be chosen to consist of open sets satisfying 
\begin{equation*}%\label{eq:proper_boundedness}
 	\overline{K}_m\subseteq K_{m+1} \; ,\;  \text{for all} \; m\in \N.
 	\end{equation*}
Any such sequence $(K_m)$ is called a \textit{proper localizing sequence}.	 

By  the theory of Hu~\cite[Section V.5]{hu:1966}, properties (i)-(iii) are equivalent to the existence of a metric on $\x$ which generates the topology of $\x$ and such that the corresponding family of metrically bounded Borel subsets of $\x$ is precisely $\bounded$. Since this is exactly the framework of \cite[Chapter 4]{kallenberg:2017}, the theory developed therein directly applies. In particular, by \cite[Theorem 4.2]{kallenberg:2017}, the topology on $\mx$ inducing the notion of vague convergence, called the vague topology, is again Polish, see also \cite[Section 3]{basrak:planinic:2019}. 

Note that by choosing a different family of bounded sets one changes the space of locally finite measures and the related notion of vague convergence. 
\begin{example}\label{exa:M0Convergence} 
Let $(\spaceX',d')$ be a complete and separable metric space and $\mathbb{C}\subseteq \mathbb{X}'$ a closed set. 
Assume that $\spaceX$ is of the form $\spaceX=\spaceX'\setminus \mathbb{C}$ equipped with the subspace topology and
 set $\bounded$ to be the class of all {Borel} sets $B\subseteq \spaceX$ such that for some $\epsilon>0$, $d'(x,\mathbb{C})>\epsilon$ for all $x\in B$, where   $d'(x,\mathbb{C})=\inf\{d'(x,z):z\in \mathbb{C} \}$. In words, $B$ is bounded if it is bounded away from $\mathbb{C}$. Such $\bounded$ properly localizes $\x$ and one can take $K_m=\{x\in \spaceX: d'(x,\mathbb{C}) >1/m\}$, $m\in \N$, as a proper localizing sequence. The corresponding notion of convergence coincides with the so-called $\mathbb{M}_{\mathbb{O}}$-convergence from Lindskog et al.~\cite{lindskog:resnick:roy:2014} and is frequently used in extreme value theory.
\end{example}

%In particular, the theory includes the notions of weak convergence of finite measures, vague convergence of Radon measures on a locally compact space, or the so-called $\mathbb{M}_{\mathbb{O}}$-convergence which is frequently used in extreme value theory (including this paper); for details see \cite[Examples 2.3-2.5]{basrak:planinic:2019}.

Denote by
$\delta_{x}$ the Dirac measure concentrated at $x\in\mathbb{X}$.	 A \textit{(locally finite) point measure} on $\x$ is a locally finite measure $\mu\in \mx$ which is of the form $\mu=\sum_{i=1}^K \delta_{x_i}$ for some $K\in \{0,1,\dots\}\cup\{\infty\}$ and (not necessarily distinct) points $x_1,x_2,\dots,x_K$ in $\spaceX$. Denote by $\mpx$ the space of all point measures on $\x$ and endow it with the vague topology. Vague convergence of point measures is equivalent to the convergence of points in (almost) all bounded Borel sets of $\x$, see \cite[Proposition 2.8]{basrak:planinic:2019} for details.

A point process on $\x$ is a random element of the space $\mpx$ with respect to the Borel $\sigma$-algebra. 
We denote convergence in distribution by $\dto$. 
Recall, for point processes $N, N_1, N_2 ,\dots$, convergence of Laplace functionals $\ex[e^{-N_n(f)}]\to \ex[e^{-N(f)}]$ for all $f\in \cbbp$ is equivalent to $N_n\dto N$ in $\mpx$, see \cite[Theorem 4.11]{kallenberg:2017}. 
%It is often useful if one can restrict to a smaller family of functions.
%
%Denote by $disc(f)$ the set of all discontinuity points of a function $f:\x \to \R$.
\begin{definition} \label{defn:conv_det_fam}
We say that a family $\F\subseteq \cbbp$  is (point process) {convergence determining} if, for any point processes $N, N_1, N_2 ,\dots$, convergence  $\ex[e^{-N_n(f)}]\to \ex[e^{-N(f)}]$ for all $f\in \F$ implies that $N_n\dto N$ in $\mpx$.
%
%We say that a family $\F$ of nonnegative, bounded and measurable functions of $\x$ with bounded support is (point process) \textit{convergence determining} if, for point processes $N, N_1, N_2 ,\dots$, convergence of Laplace functionals $\ex[e^{-N_n(f)}]\to \ex[e^{-N(f)}]$ for all $f\in \F$ such that $N(disc(f))=0$ a.s., implies that $N_n\dto N$ in $\mpx$.
\end{definition}

For example,  one can take the subfamily $\F\subseteq\cbbp$ of functions which are Lipschitz continuous with respect to a suitable metric, see \cite[Proposition 4.1]{basrak:planinic:2019}.

\subsection{General Poisson approximation}
Let $(I_n)_{n\in \N}$ be a sequence of 
%\tcb{at most countably infinite} index sets
finite index sets but such that $\lim_{n\toi}|I_n|=\infty$, where $|I_n|$ denotes the number of elements in $I_n$.
For each $n\in \N$, let $(X_{n,i}: i\in I_n)$ be a family of 
%\tcb{identically distributed} 
random elements in a topological space $\x'$. Assume that 
there exists a Polish subset $\mathbb{X}$ of $\mathbb{X}'$ (e.g.\ as in \Cref{exa:M0Convergence}) with a family of bounded Borel sets $\bborel=\bborel(\x)$ such that, as $n\toi$,
\begin{align}\label{eq:nullArrayCond}
\sup_{i\in I_n}\pr(X_{n,i}\in B) \to 0 \, , \;  B\in \bbou \, .
\end{align}
The central theme of this section is convergence in distribution in $\mpx$ of the  point processes 
\begin{equation*}%\label{eq:N_n}
N_n=\sum_{i\in I_n}\delta_{X_{n,i}} %\ind{X_{n,i}\in \mathbb{X}}
\, , \; n\in \N \, ,
\end{equation*}
restricted to the space $\x$. For a   locally finite measure $\lambda$ on $\x$ denote by $\PPP(\lambda)$ the distribution of a Poisson point process on $\x$ with intensity measure $\lambda$.

Observe that if for each $n\in \N$, $(X_{n,i}: i\in I_n)$ were independent, (\ref{eq:nullArrayCond}) would imply that measures $\delta_{X_{n,i}}$ on $\x$, $n\in \N, i\in I_n$ form a null-array (see \cite[p.\ 129]{kallenberg:2017})
%In this case, if $N_n\dto N$ in $\mpx$ for some $N$, then $N$ is necessarily a Poisson process. Indeed, by \cite[Theorem 4.22]{kallenberg:2017} $N$ is infinitely divisible and since, by construction of $N_n$, its L\'evy measure  is concentrated on the set $\{\delta_x : x\in \x\}$. Moreover, by the so-called Grigelionis theorem (see \cite[Corollary 4.25]{kallenberg:2017}), this convergence holds with $N\eind \PPP(\lambda)$ for $\lambda\in \mx$ if and only if  $\ex[N_n(\,\cdot\,)]=\sum_{i\in I_n} \pr(X_{n,i}\in \cdot)\vto \lambda$ in $\mx$.
and by the so-called Grigelionis theorem (see \cite[Corollary 4.25]{kallenberg:2017}), for $\lambda\in \mx$, convergence $N_n\dto N \sim \PPP(\lambda)$ holds in $\mpx$  if and only if  
$$\ex[N_n(\,\cdot\,)]=\sum_{i\in I_n} \pr(X_{n,i}\in \cdot)\vto \lambda$$
 in $\mx$.

%Recall, for each $n\in \N$, a family of random elements $(X_{n,i} :  i \in I_n)$ in $\x'$ is given. Further, we will assume that $X_{n,i}$'s are such that $\sup_{i\in I_n}\pr(X_{n,i}\in B) \to 0$ as $n\toi$ for all $B\in \bbou$ and that the intensity measures of point processes $N_n=\sum_{i\in I_n}\delta_{X_{n,i}}$ on $\x$ vaguely converge to a measure $\lambda$ in $\mx$, i.e.\ that $\ex[N_n(\,\cdot\,)]=\sum_{i\in I_n} \pr(X_{n,i}\in \cdot)\vto \lambda$.

%If in addition, for each $n\in\N$, the family $(X_{n,i}: i\in I_n)$ were independent, the famous result of Grigelionis (see \cite[Corollary 4.25]{kallenberg:1983}) applied to the null array $(\delta_{X_{n,i}}:n\in \N ,\: i\in I_n )$ would imply that point processes   
%$N_n$ converge in distribution to a Poisson point process with intensity measure $\lambda$. In what follows the distribution of such  process will be denoted by $\PPP(\lambda)$.

In general, one can still obtain the same Poisson limit if the asymptotic distributional behavior of $N_n$'s is indistinguishable from its independent version. 

More precisely, let for each $n\in \N$, $(X_{n,i}^*:i\in I_n)$ be independent random elements such that for all $i\in I_n$, $X_{n,i}^*$ is distributed as $X_{n,i}$, and denote by $N_n^*=\sum_{i\in I_n}\delta_{X_{n,i} ^*}$ the corresponding point processes on $\x$. Further, let $\F$ be a class of measurable and nonnegative functions on $\x$ with {bounded support.} We say that the family $(X_{n,i}:n\in \N , \: i\in I_n)$ is \textit{asymptotically $\F$-independent} ($AI(\F)$) if
 \begin{equation*}
 \left|\ex\left[e^{-N_n(f)}\right]-\ex\left[e^{-N_n^*(f)}\right]\right| =\big| \ex\left[e^{-\sum_{i\in I_n}f (X_{n,i})}\right] - \prod_{i\in I_n} \ex\left[ e^{
      - f (X_{n,i}) } \right] \big| \to 0 , \; \text{as} \; n\toi,
  \end{equation*} 
for all $f\in \F$, where  we set  $f(x)=0$ for all $x\in \mathbb{X}' \setminus \mathbb{X}$.  
To obtain  meaningful results we will require that the functions in $\F$ determine convergence in distribution in $\mpx$ in the sense of \Cref{defn:conv_det_fam}. 
{Since $N_n^*\dto N \sim \PPP(\lambda)$  implies convergence $\EE[e^{-N_n^*(f)}]\to\EE[e^{-N(f)}]$ for all $f\in \cbbp$, the following result is now immediate.}

% \hrnote{See in Kallenberg if it is possible to have (local) total variation distance.}
\begin{theorem}\label{thm:gen_poisson_apprx}
Assume that (\ref{eq:nullArrayCond}) holds and that there exists a measure $\lambda\in \mathcal{M}(\mathbb{X})$  
such that, 
%\begin{enumerate}
%\item[(i)] For all $m\in\N$
%\begin{equation}\label{eq:nullarray}
%\lim_{n\toi} \sup_{i\in I_n} \pr(X_{n,i}\in K_m)=0 \; .
%\end{equation}
%\item[(ii)]  
%%There exists a \tcb{non zero} measure $\lambda\in \mathcal{M}(\mathbb{X})$ such that, as $n\toi$,
as $n\toi$,
\begin{equation}\label{eq:independentblocksconvergence}
\sum_{i\in I_n} \pr(X_{n,i}\in \cdot )\vto \lambda \; .
\end{equation}
%%\item[(iii)] The family $(X_{n,i}:n\in \N , \: i\in I_n)$ is $AI(\F)$.
%\end{enumerate}
Then for any convergence determining family $\F\subseteq \cbbp$,  $N_n\dto N \sim \PPP(\lambda)$  in $\mathcal{M}_p(\mathbb{X})$ if and only if $(X_{n,i}:n\in \N , \: i\in I_n)$ is $AI(\F)$. 
\end{theorem}
%\begin{proof}
%%The result is immediate since (\ref{eq:independentblocksconvergence}) implies  convergence $N_n^*\dto N \sim \PPP(\lambda)$ which in turn implies convergence $\EE[e^{-N_n^*(f)}]\to\EE[e^{-N(f)}]$ for all $f\in \cbbp$.  
%%As a consequence of the continuous mapping theorem (see \cite[Lemma 4.12]{kallenberg:2017}), $\EE[e^{-N_n^*(f)}]\to\EE[e^{-N(f)}]$ holds for all nonnegative and bounded functions $f$ with bounded support for which $N(disc(f))=0$ almost surely. In particular it holds for all $f\in \cbbp$ and the result is now obviou 
%\end{proof}
%\begin{proof}
%Since by \eqref{eq:nullarray} the random point measures $(\delta_{X_{n,i}^*}:n\in \N ,\: i\in I_n )$ form a null array (see \cite[Section 4.3]{kallenberg:2017}), it follows from \cite[Corollary 4.25]{kallenberg:2017} that \eqref{eq:independentblocksconvergence} is equivalent to $N_n^*\dto N \sim \PPP(\lambda)$ in $\mathcal{M}_p(\mathbb{X})$.
%As a consequence of the continuous mapping theorem (see \cite[Lemma 4.12]{kallenberg:2017}), $\EE[e^{-N_n^*(f)}]\to\EE[e^{-N(f)}]$ for all nonnegative and bounded functions $f$ with bounded support for which $N(disc(f))=0$ a.s.
% %almost everywhere continuous with respect to $\lambda$. 
% The result now follows since we assumed that the family $\F$ is convergence determining and that $(X_{n,i}:n\in \N , \: i\in I_n)$ is $AI(\F)$.
%\end{proof}
\begin{remark}
{Observe that we have assumed that $\F$ consists only of continuous functions. However, $N_n^*\dto N $ actually implies $\EE[e^{-N_n^*(f)}]\to\EE[e^{-N(f)}]$ for all nonnegative and bounded functions $f$ with bounded support for which $N(disc(f))=0$ almost surely, where $disc(f)$ denotes  the set of all discontinuity points of $f$ (see \cite[Lemma 4.12]{kallenberg:2017}).} Consequently, if (\ref{eq:independentblocksconvergence}) holds and $N \sim \PPP(\lambda)$,
in the necessary and sufficient condition for $N_n\dto N$, one can allow $\F$ 
 to be a sufficiently rich class  of functions $f$ which are not necessarily continuous, 
% \tcm{but satisfy $(X_{n,i}:n\in \N , \: i\in I_n)$ is $AI(\F)$ and $N(disc(f))=0$ almost surely instead.} 
%% where $disc(f)$ denotes  the set of all discontinuity points of $f$. 
e.g.\ $\F$ could consist of nonnegative simple functions with bounded support, see \cite[Theorem 4.11]{kallenberg:2017} for details.
\end{remark}

\begin{remark}
Assume that (\ref{eq:nullArrayCond}) holds and that $X_{n,i}$'s are $AI(\F)$ for some convergence determining family $\F$. In this case, if $N_n$ converge in distribution to some limit, $N$ say, then $N$ is necessarily a Poisson process. Indeed, since also $N_n^*\dto N$,  by \cite[Theorem 4.22]{kallenberg:2017} $N$ is infinitely divisible and moreover, by the construction of $N_n^*$, its so-called L\'evy measure (see \cite[p.\ 89]{kallenberg:2017})  is concentrated on the set $\{\delta_x : x\in \x\}$ which implies that $N$ is Poisson.
\end{remark}

Observe that the assumption $AI(\F)$ implies   that  $X_{n,i}$, $i\in I_n$ asymptotically behave as if they were independent, but only on the bounded sets of the space $\mathbb{X}$. The key fact here is that all functions in $\F$ have bounded support, so for every fixed $f\in  \F$, $N_n(f)$ is unaffected by the behavior of $X_{n,i}$'s outside of a fixed bounded set. Sufficient condition for $AI(\F)$ to hold is given in Proposition \ref{prop:AI_condition} below.

{First we state a stationary version of the previous result, cf.\ \cite[Proposition 3.21]{resnick:1987}.
For $d\in \N$} consider   the space $[0,1]^d \times \mathbb{X}$ with respect to the product topology and with {$B'\in \borel([0,1]^d \times \mathbb{X})$ being bounded if the set $\{x\in \x : (\bt,x)\in B' \text{ for some } \bt\in [0,1]\}$ is bounded in $\x$.}

%Consider now the space $[0,1]^d \times \mathbb{X}$ with respect to the product topology and with bounded sets being those which are contained in $K_m'=[0,1]^d\times K_{m}$ for some $m\in \N$. Let $Leb$ denote the Lebesgue measure  on $[0,1]^d$. As a simple consequence of the previous theorem, one obtains an important special case,  cf.\ {\cite[Proposition 3.21]{resnick:1987}}. 

\begin{corollary}\label{cor:poissonapprox_stationarycase}
Assume that $I_n=\{1,2,\dots,k_n\}^d \subseteq \Z^{d}$ for some $d\in\N$ with $k_n\toi$ and that $(X_{n,\bi}: \bi\in I_n)$ are identically distributed for every $n\in \N$. If there exists a measure $\nu\in \mathcal{M}(\mathbb{X})$
% and a convergence determining family $\F'$ on $[0,1]^d \times \mathbb{X}$ 
such that,
%\begin{enumerate}[(i)]
%\item 
as $n\toi$,
 \begin{equation}\label{eq:independentblocksconvergence_idcase}
k_n ^d \pr(X_{n,\bone}\in \cdot )\vto \nu \, ,
\end{equation}
%\item The family $((\bi/k_n, X_{n,\bi}): n\in \N, \bi \in I_n)$ is $AI(\F')$.
%\end{enumerate}
then for any convergence determining family $\F'$ on $[0,1]^d \times \mathbb{X}$,  
$$N_n'=\sum_{\bi\in I_n}\delta_{(\bi/k_n,X_{n,\bi})} \dto N' \sim \PPP(\mbox{Leb} \times \nu)$$ 
in $\mathcal{M}_p([0,1]^d \times \mathbb{X})$ if and only if $((\bi/k_n, X_{n,\bi}): n\in \N, \bi \in I_n)$ is $AI(\F')$, where $\mbox{Leb}$ denotes the Lebesgue measure on $[0,1]^d$.
\end{corollary}
%In addition to the above assumptions, suppose that there exists a neighborhood structure $(B_n(\bi) : \bi\in I_n)$, a \tcb{non zero} measure $\nu\in \mathcal{M}(\mathbb{X})$ and a convergence determining family for $\nu$, say $\mathcal{F}_\nu$, closed under multiplication with a positive scalar, such that, as $n\toi$,
%%Assume that there exists a \tcb{non zero} measure $\nu\in \mathcal{M}(\mathbb{X})$ such that,as $n\toi$, 
%\begin{equation}\label{eq:independentblocksconvergence_idcase}
%k_n ^d \pr(X_{n,\bone}\in \cdot )\vto \nu
%\end{equation}
%and $\lim_{n\toi }b_{n,1}^m=\lim_{n\toi }b_{n,2}^m=\lim_{n\toi } b_{n,3}(f)=0$ for all $m\in\N$ and every $f\in \mathcal{F}_{\nu}$. 
%
%Then the point processes $N_n'=\sum_{\bi\in I_n}\delta_{(\bi/k_n,X_{n,\bi})}$ converge in distribution in $\mathcal{M}_p([0,1]^d \times \mathbb{X})$ to a Poisson point process with intensity measure $Leb\times\nu$.
%\end{theorem}
\begin{proof} We simply apply Theorem \ref{thm:gen_poisson_apprx} to random elements $X_{n,\bi}':=(\bi/k_n, X_{n,\bi})$, $n\in \N,\bi \in I_n$. 
Take an arbitrary $B'\in \bborel([0,1]^d \times \mathbb{X})$ and define $B=\{x\in \x : (\bt,x)\in B' \text{ for some } \bt\in [0,1]\}$. Since $B\in \bborel(\x)$, \eqref{eq:independentblocksconvergence_idcase} and \cite[Lemma 4.1(iv)]{kallenberg:2017} imply that
\[
\limsup_{n\toi} \sum_{\bi \in I_n} \pr(X_{n,\bi}'\in B')=\limsup_{n\toi} k_n^d \pr(X_{n,\bone}\in B)\leq \nu(\overline{B})<+\infty \; .
\]
Hence, \eqref{eq:nullArrayCond} holds since $k_n\toi$. 
 
 Further, note that for arbitrary $\ba=(a_1,\dots, a_d)$ and $\bb=(b_1,\dots, b_d)$ in $[0,1]^d$ such that $a_j\leq b_j$ for all $j=1,\dots,d$ and a set $B\in\mathcal{B}_b$ such that $\nu(\partial B)=0$, \eqref{eq:independentblocksconvergence_idcase} implies that as $n\toi$,
\begin{align*}
\sum_{\bi\in I_n} \pr(X_{n,\bi}'\in (\ba,\bb]\times B) = \frac{1}{k_n^d} \prod_{j=1}^d \lfloor k_n(b_j - a_j)\rfloor \cdot k_n^d \pr(X_{n,\bone}\in B)\to \prod_{j=1}^d (b_j - a_j) \cdot \nu(B) \;
\end{align*}  
By \cite[Lemma 4.1]{kallenberg:2017}, this implies that $\sum_{\bi\in I_n} \pr(X_{n,\bi}'\in \cdot)\vto \mbox{Leb}\times \nu$  in $\mathcal{M}([0,1]^d \times \mathbb{X})$,  i.e.\ (\ref{eq:independentblocksconvergence}) holds with $\lambda=\mbox{Leb}\times \nu$. 
\end{proof}

%In the rest of the section, we focus on giving sufficient conditions under which $X_{n,i}$'s satisfy the asymptotic independence assumption. 

\subsection{Sufficient conditions for asymptotic $\F$-independence}
For each $i\in I_n$, choose a subset of the index set $B_n(i)\subseteq I_n$ containing $i$, and call it the \textit{neighborhood of dependence} of $i\,$.  Intuitively, it will be {beneficial to choose $B_{n}(i)$ as small as possible,} but such that $X_{n,i}$ is (nearly) independent of all $X_{n,j}$ for $j\notin B_n(i)$.
%Intuitively, one thinks of $X_{n,j}$ being dependent on $X_{n,i}$ for all $j\in B_n(i)$, and  independent or \textit{nearly} independent for all $j\in B_n^c(i)=I_n\setminus B_n(i)$. 

Select an arbitrary ordering of the elements in $I_n$. Without loss of generality, we will assume that $I_n=\{1,2,\dots,m_n\}$ where $m_n\toi$ as $n\toi$. For all $i\in I_n$ {partition $\{i+1,\dots,m_n\}$ into}
$\tilde{B}_n(i):=\{j\in B_n(i) : j>i\}$ and $\tilde{B}^c_n(i):=\{j\notin B_n(i) : j>i\}$. Further, fix an arbitrary sequence $(K_m)_{m\in \N}\subseteq \bborel$ of $\bborel$ such that for every $B\in \bborel$, $B\subseteq K_m$ for some $m\in \N$.

For a given neighborhood structure $(B_n(i) : n\in \N, \, i\in I_n)$ and for all $m,n\in \N$ define 
\begin{align*}
b_{n,1}^m&=\sum_{i\in I_n}\sum_{j\in \tilde{B}_n(i)} \pr(X_{n,i}\in K_m)\cdot \pr(X_{n,j}\in K_m)\, , \\
b_{n,2}^m &=\sum_{i\in I_n}\sum_{j\in \tilde{B}_n(i)} \pr(X_{n,i}\in K_m ,X_{n,j}\in K_m) \, .
\end{align*}
%\begin{align*}
%b_{n,1}^m&=\sum_{i\in I_n}|B_n(i)| \pr(X_{n,i}\in K_m)^2\;, \\
%b_{n,2}^m &=\sum_{i\in I_n}\sum_{j\in B_n(i)\setminus\{i\}} \pr(X_{n,i}\in K_m ,X_{n,j}\in K_m)
%\end{align*}
Furthermore, for all $n\in\N$ and an arbitrary nonnegative measurable function $f$ on $\mathbb{X}$ 
%denote $\mu_{n,i}(f)=\ex[e^{-f(X_{n,i})}]$ and 
define 
\begin{align*}
b_{n,3}(f)=\sum_{i\in I_n}\big|\ex\big[e^{-f(X_{n,i})} \prod_{j\in \tilde{B}^c_n(i)} e^{-f(X_{n,j})}\big]-\ex\big[e^{-f(X_{n,i})}\big]\cdot \ex\big[\prod_{j\in \tilde{B}^c_n(i)} e^{-f(X_{n,j})}\big]\big| \; .
\end{align*}
%\[
%b_{n,3}(f)=\sum_{i\in I_n} \ex\big| \ex\big[ e^{-f(X_{n,i})} \mid \sigma(X_{n,j}:j\in \tilde{B}^c_n(i))\big]-\ex\big[e^{-f(X_{n,i})}\big]\big| \; ,
%\]

\begin{proposition}\label{prop:AI_condition}
Let $f$ be a nonnegative measurable function on $\x$ with bounded support. If $m\in \N$ is such that the support of $f$ is contained in $K_m$, then {for all $n\in \N$},
\begin{equation*}
\big| \ex\left[e^{-\sum_{i\in I_n}f (X_{n,i})}\right] - \prod_{i\in I_n} \ex\left[ e^{
      - f (X_{n,i}) } \right] \big| \leq b_{n,1}^m + b_{n,2}^m + b_{n,3}(f)
\end{equation*}
for all $n\in \N$. In particular, if there exists a neighborhood structure $(B_n(i) : n\in \N, \, i\in I_n)$ such that for all $m\in \N$ and every $f\in \F$  
\begin{equation*}
\lim_{n\toi }b_{n,1}^m=\lim_{n\toi }b_{n,2}^m=\lim_{n\toi } b_{n,3}(f)=0 \; ,
\end{equation*}
then the family $(X_{n,i}:n\in \N , \: i\in I_n)$ is $AI(\F)$.
\end{proposition}

\begin{proof}
The proof is an adaptation of argument in Nakhapetyan~\cite[Lemma 3]{nakhapetyan:1988}, though the main idea goes back to~\cite[Theorem 4]{banys:1980}. Since $e^{-f}$ is positive and bounded by $1$ it follows that
   \begin{align*}
     &\big| \ex\left[e^{-\sum_{i\in I_n}f (X_{n,i})}\right] - \prod_{i\in I_n} \ex\left[ e^{
      - f (X_{n,i}) } \right] \big|   \\
      & \leq\sum_{i=1}^{m_n -1}\big|\ex\big[e^{-f(X_{n,i})}\prod_{j=i+1}^{m_n} e^{-f(X_{n,j})}\big]-\ex\big[e^{-f(X_{n,i})}\big]\cdot\ex\big[\prod_{j=i+1}^{m_n} e^{-f(X_{n,j})}\big]\big| =: \sum_{i=1}^{m_n -1} \varepsilon_i \; .
  \end{align*} 
Fix now an arbitrary $i\in \{1,\dots,m_n-1\}$. After writing 
\[
\prod_{j=i+1}^{m_n} e^{-f(X_{n,j})}=\prod_{j\in \tilde{B}_n(i)} e^{-f(X_{n,j})}\prod_{j\in \tilde{B}^c_n(i)} e^{-f(X_{n,j})} \; ,
\] 
one can easily check that
\begin{align*}
\varepsilon_i& \leq \Big|\ex\big[e^{-f(X_{n,i})}\cdot\big(\prod_{j\in \tilde{B}_n(i)} e^{-f(X_{n,j})}-1\big)\prod_{j\in \tilde{B}^c_n(i)} e^{-f(X_{n,j})}\big]\\
&-\ex\big[e^{-f(X_{n,i})}\big]\cdot\ex\big[\big(\prod_{j\in \tilde{B}_n(i)} e^{-f(X_{n,j})}-1\big)\prod_{j\in \tilde{B}^c_n(i)} e^{-f(X_{n,j})}\big]\Big| \\
&+ \big|\ex\big[e^{-f(X_{n,i})} \prod_{j\in \tilde{B}^c_n(i)} e^{-f(X_{n,j})}\big]-\ex\big[e^{-f(X_{n,i})}\big]\cdot \ex\big[\prod_{j\in \tilde{B}^c_n(i)} e^{-f(X_{n,j})}\big]\big|\; .
\end{align*}
  Note that the first summand on the right hand side of the previous inequality equals
  \begin{align*}
   &\big|\ex\big[\big(e^{-f(X_{n,i})}-1\big)\cdot\big(\prod_{j\in \tilde{B}_n(i)} e^{-f(X_{n,j})}-1\big)\prod_{j\in \tilde{B}^c_n(i)} e^{-f(X_{n,j})}\big]\\
&-\ex\big[\big(e^{-f(X_{n,i})}-1\big)\big]\cdot\ex\big[\big(\prod_{j\in \tilde{B}_n(i)} e^{-f(X_{n,j})}-1\big)\prod_{j\in \tilde{B}^c_n(i)} e^{-f(X_{n,j})}\big]\big| \; , 
  \end{align*}
 and since $e^{-\sum_k f(x_k)}-1\neq 0$ implies that $f(x_k)>0$, and hence $x_k\in K_m$, for at least one $k$, we obtain that
 \begin{align*}
\varepsilon_{i} &\leq \pr\big(X_{n,i}\in K_{m},  \bigcup_{j\in \tilde{B}_n(i)} \{X_{n,j}\in K_{m}\}\big)+\pr\big(X_{n,i}\in K_{m}\big) \cdot \pr\big(\bigcup_{j\in \tilde{B}_n(i)} \{X_{n,j}\in K_{m}\}\big) \\
& +\big|\ex\big[e^{-f(X_{n,i})} \prod_{j\in \tilde{B}^c_n(i)} e^{-f(X_{n,j})}\big]-\ex\big[e^{-f(X_{n,i})}\big]\cdot \ex\big[\prod_{j\in \tilde{B}^c_n(i)} e^{-f(X_{n,j})}\big]\big|\; .
\end{align*}

%Let $(X_{n,i}^*:n\in \N, \: i\in I_n)$ and $(X_{n,i}:n\in \N , \: i\in I_n)$ be defined on the same probability space and independent. We can then bound the last term on the right hand side of the above inequality by 
%\begin{align*}
%  &\ex\big| \ex\big[ e^{-f(X_{n,i})}-e^{-f(X_{n,i}^*)} \mid \sigma(X_{n,j}:j\in \tilde{B}^c_n(j))\big]\big|\\
%  & = \ex\big| \ex\big[ e^{-f(X_{n,i})} \mid \sigma(X_{n,j}:j\in \tilde{B}^c_n(i))\big]-\ex\big[e^{-f(X_{n,i})}\big]\big| \; .
%\end{align*}
Hence, 
\[
\big| \ex\left[e^{-\sum_{i\in I_n}f (X_{n,i})}\right] - \prod_{i\in I_n} \ex\left[ e^{
      - f (X_{n,i}) } \right] \big| \leq \sum_{i=1}^{m_n - 1}\varepsilon_i \leq b_{n,1}^m + b_{n,2}^m + b_{n,3}(f) \; .
      \]
\end{proof}

\begin{remark}\label{rem:b_3'}
Recall, $(X_{n,i}^*:i\in I_n)$ are independent random elements such that for all $i\in I_n$, $X_{n,i}^*$ is distributed as $X_{n,i}$. Further,
let $(X_{n,i}^*: i\in I_n)$ and $(X_{n,i}: i\in I_n)$ be defined on the same probability space and independent. We can then bound $b_{n,3}(f)$ by 
\begin{align*}
  b_{n,3}(f)&\leq \sum_{i\in I_n}\ex\big| \ex\big[ e^{-f(X_{n,i})}-e^{-f(X_{n,i}^*)} \mid \sigma(X_{n,j}:j\in \tilde{B}^c_n(i))\big]\big|\\
  & = \sum_{i\in I_n}\ex\big| \ex\big[ e^{-f(X_{n,i})} \mid \sigma(X_{n,j}:j\in \tilde{B}^c_n(i))\big]-\ex\big[e^{-f(X_{n,i})}\big]\big| \; .
\end{align*}
Since for any $f\in \cbbp$ the function $1-e^{-f}$ is also an element $\cbbp$ and further bounded by $1$, it follows that 
\begin{equation*}
%b_{n,3}'(f):=
\sum_{i\in I_n}\ex\big| \ex\big[ f(X_{n,i}) \mid \sigma(X_{n,j}:j\in \tilde{B}^c_n(i))\big]-\ex\big[f(X_{n,i})\big]\big| \to 0
\end{equation*}
for all $f\in \cbbp$ which are bounded by $1$ implies that $b_{n,3}(f)\to 0$ for all $f\in \cbbp$. 
%The same argument holds with $\cbbp$ replaced by $LB_b^+(\x,\rho)$ where $\rho$ is a metric on $\x$.
%%By the same argument it follows that for any metric $\rho$ on $\x$, $b_{n,3}'(f)\to 0$ for all $f\in LB_b^+(\x,\rho)$ bounded by $1$ implies $b_{n,3}(f)\to 0$ for all $f\in LB_b^+(\x,\rho)$.
\end{remark}

\begin{remark}
The concept of neighborhoods implicitly appears already in Banys~\cite[Theorem 4]{banys:1980}. There, essentially the same sufficient conditions for convergence of $N_n$ to a Poisson point process are given but with, in our notation, neighborhoods of the form  $\tilde{B}_n(i)=\{i+1,\dots,i+r_n\}$ and $\tilde{B}^c_n(i)=\{i+r_n+1,\dots,m_n\}$ for all $i\in I_n$ where $(r_n)_{n\in \N}$ is a sequence of nonnegative integers.
%where it was essentially shown that $N_n$ converge  in distribution to a Poisson point process with intensity $\lambda$ if, in addition to (\ref{eq:nullarray}) and (\ref{eq:independentblocksconvergence}), for all $m\in \N$ and every $f\in \cbbp$ bounded by $1$, 
% \begin{equation*}
%\lim_{n\toi }b_{n,1}^m=\lim_{n\toi }b_{n,2}^m=\lim_{n\toi } b_{n,3}'(f)=0 \; ,
%\end{equation*}
%where $\tilde{B}_n(i)=\{i+1,\dots,i+r_n\}$ and $\tilde{B}^c_n(i)=\{i+r_n+1,\dots,m_n\}$ for all $i\in I_n$ and some sequence of nonnegative integers $(r_n)_{n\in \N}$. 
The proof is similar to ours and even though it is stated only for the case when $\x$ is locally compact, it transfers directly to the case of a general Polish space.
\end{remark}

\begin{remark}
Similar results were also obtained by Schuhmacher~\cite[Theorem 2.1]{schuhmacher:2005}, but with a completely different approach, using the Chen-Stein method. As a consequence, Schuhmacher even provides bounds on the convergence in the so-called Barbour-Brown distance $d_2$. However, this result  does not directly imply our results,
%since,
%assuming $\x$ is itself bounded as in \cite[Theorem 2.1]{schuhmacher:2005}, it essentially requires that 
%\[
%\sum_{i\in I_n} \ex\sup_{f}\big|\ex\big[ f(X_{n,i}) \mid \sigma(X_{n,j}:j\notin B_n(i))\big]-\ex\big[f(X_{n,i})\big]\big|\to 0
%\]
%where the supremum is taken over all $f\in LB_b^+(\x,\rho)$ which are bounded by $1$ and with Lipschitz constant not bigger than $1$, and $\rho$ being any metric on $\x$ generating the topology.
%%roughly speaking, it requires some uniformity in the term $b_{n,3}'(f)$.
 see \cite[Remark 2.4(b)]{schuhmacher:2005} for the comparison to the result of Banys~\cite{banys:1980} which is also relevant to our case.
\end{remark}

\begin{example}
For Bernoulli random variables $X_{n,i}$
	such that   $\lim_{n\toi} \sup_{i\in I_n} \pr(X_{n,i}=1)=0$ and $\lim_{n\toi}\sum_{i\in I_n} \pr(X_{n,i}=1) =\lambda\in (0,\infty)$, one can set $\mathbb{X}'=\{0,1\}$ and $\mathbb{X}=K_m=\{1\}$ for all $m\in \N$. 
	Using Theorem \ref{thm:gen_poisson_apprx} together with Proposition \ref{prop:AI_condition} and Remark \ref{rem:b_3'}, we recover the result of Arratia et al.~\cite[Theorem 1]{arratia:goldstein:gordon:1989} on convergence in distribution of $\sum_{i\in I_n} X_{n,i}$ to a Poisson random variable with intensity $\lambda$, but without the bound on the distance in total variation.

\end{example}

\section{Regularly varying fields}\label{sec:RegVarF}
%\label{sec:tailField}

%Let  $\bX=(X_{\bi}:\bi\in\Z^d)$ for $d\in \N$ be a stationary random field. 
%For nonempty and finite $I\subseteq \Z^d$ one of numerous definitions of multivariate regular variation of $\bX_I$ is the following. $\bX_{I}$ is multivariate regularly varying with index $\alpha>0$ if for an arbitrary norm $\|\cdot\|$ on $\R^{|I|}$ there exists a random element $\bTheta^{(I)}$ in the corresponding unit sphere in $\R^{|I|}$ such that 
%\begin{align}\label{eq:standardRV}
%%\left( \frac{\|\bX_{I}\|}{u}, \frac{\bX_{I}}{\|\bX_{I}\|}\right)
%u^{-1} \bX_{I}
%\, \big| \, \|\bX_{I}\|> u \dto 
%%(Y, \bTheta^{(I)}) 
%Y\bTheta^{(I)}
%\, , \; \text{as } u\toi \, ,
%\end{align}
%where $Y$ is Pareto distributed  with index $\alpha$ and independent of $\bTheta^{(I)}$, see \cite[Proposition 2.1]{segers:meinguet:2010}.

%by, roughly speaking, conditioning in (\ref{eq:standardRV})  on $|X_i|>u$ for fixed $i\in I$. Because of stationarity this an 

\subsection{The tail field}\label{sub:tail}

Consider a (strictly) stationary $\R$-valued random field $\bX=(X_{\bi}:\bi\in\Z^d)$ with $d\in \N$. 
For every finite and nonempty subset of indices $I\subseteq \Z^d$, denote by $\bX_{I}$  the $\R^{|I|}$-valued random vector $(X_{\bi}: \bi \in I)$, i.e.\ $\bX_{I}$'s represent finite-dimensional distributions of $\bX$. 

We say that a random field $\bY=(Y_{\bi}:\bi\in\Z^d)$ is the \textit{tail field} (or \textit{tail process}) of $\bX$, if for all finite and nonempty $I\subseteq \Z^d$, 
\begin{align}\label{eq:conv_to_tail}
u^{-1} \bX_{I}
\, \big| \, |X_{\bo}|> u \dto 
%(Y, \bTheta^{(I)}) 
\bY_{I}
\, , \; \text{as } u\toi \, ,
\end{align}
where $\bo=(0,\dots,0)\in \Z^d$. Here and in the rest of the paper, $A(u)\mid B(u) \dto C$ as $u\toi$ for a family of random elements $A(u),C$ and events $B(u)$, $u>0$, means that the law of $A(u)$ conditionally on $B(u)$ converges weakly as $u\toi$ to the law of $C$.

Note that in (\ref{eq:conv_to_tail}) we implicitly assume that $\pr(|X_{\bo}|>u)>0$ for all $u>0$. Observe, taking $I=\{\bo\}$ in (\ref{eq:conv_to_tail}) yields that $\lim_{u\toi}\pr(|X_{\bo}|>uy)/\pr(|X_{\bo}|>u)=\pr(|Y_{\bo}|>y)$ for all except at most countably many $y\in [1,\infty)$. By standard arguments (see \cite[Theorem 1.4.1]{bingham:goldie:teugels:1987} and the discussion before it), this implies that $u\mapsto \pr(|X_{\bo}|>u)$ is a regularly varying function with index $-\alpha$ for some $\alpha>0$, i.e.\
\begin{align*}
%\label{eq:RVof|X_0|}
\lim_{u\toi}\frac{\pr(|X_{\bo}|>uy)}{\pr(|X_{\bo}|>u)}=y^{-\alpha} \, , \, y>0\, .
\end{align*}  
In particular, $\pr(|Y_{\bo}|>y)=y^{-\alpha}$ for all $y\geq 1$, i.e.\ $|Y_{\bo}|$ is Pareto distributed with index $\alpha$.

{\begin{remark}
For notational convenience, in this paper we only consider $\R$-valued random fields. All the results in this section extend easily to the case of $\R^n$-valued random fields with $n\in\N$ by simply replacing the absolute value $|\cdot|$
with an arbitrary norm $\|\cdot\|$ on $\R^n$.
\end{remark}}

\subsubsection{Existence of the tail field}\label{sub:exist_tail}
A family of indices $\mathcal{I}\subseteq \Z^d$ is said to be \textit{encompassing} if for every finite and nonempty $I\subseteq \Z^d$ there exists at least one $\bi^*\in I$ such that $I-\bi^*\subseteq \mathcal{I}$. Note that necessarily $\bo\in \mathcal{I}$.

If $d=1$, the set of nonnegative (or nonpositive) integers is an example of such family. 
More generally, assume that $\lleq$ is an arbitrary total order on $\Z^d$ which is translation-invariant in the sense that for all $\bi,\bj$ and $\bk$ in $\Z^d$, $\bi \lleq \bj$ implies $\bi + \bk \lleq \bj + \bk$. Then  the set $\Z^d_{\lgeq}=\{\bi \in \Z^d : \bi \lgeq \bo\}$ is clearly encompassing. Indeed, simply set $\bi^*\in I$ to be the (unique) minimal element of the finite set $I$ with respect to $\lleq$. We refer to such orders as \textit{group orders} on $\Z^d$.
 
In particular, the lexicographic order on $\Z^d$, denoted by $\lleq_{l}$, is a group order. Recall, for indices $\bi=(i_1,\dots,i_d),\bj=(j_1,\dots, j_d)\in\Z^d$,  $\bi\lless_{l} \bj$ if $i_k < j_k$ for the first $k$ where $i_k$ and $j_k$ differ, and $\bi\lleq_{l} \bj$ if $\bi\lless_{l} \bj$ or $\bi= \bj$.

The following result extends \cite[Theorem 2.1]{basrak:segers:2009} which treats the case $d=1$; the proof is postponed to \Cref{subs:RVequiv}.

\begin{theorem}\label{thm:tail_process}
For a stationary random field $\bX=(X_{\bi}:\bi\in\Z^d)$ and $\alpha>0$, the following three statements are equivalent:
\begin{enumerate}[(i)]
\item\label{it:tail_process1}
All finite-dimensional distributions of $\bX$ are multivariate regularly varying with index $\alpha$; 
%$\bX$ is jointly regularly varying with index $\alpha$, i.e.\ for every finite and nonempty $I\subseteq \Z^d$, the random vector $\bX_{I}$ is multivariate regularly varying with index $\alpha$;
\item\label{it:tail_process3} 
The field $\bX$ has a tail field $\bY=(Y_{\bi} : \bi\in \Z^d)$ with $\pr(|Y_{\bo}|\geq y)=y^{-\alpha}$ for $y\geq 1$.
%%There exists a random field $\bY=(Y_{\bi}:\bi\in\Z^d)$ of random variables with $\pr(|Y_{\bo}|\geq y)=y^{-\alpha}$ for $y\geq 1$ such that, as $u\toi$, 
%%\[
%%u^{-1}\bX \;\big| \;|X_{\bo}|>u \dto \bY  \: \; \text{in} \; \R^{\Z^d} \, .
%%\]
\item\label{it:tail_process2} There exists an encompassing $\mathcal{I}\subseteq \Z^d$ and a family of random variables $(Y_{\bi}: \bi \in \mathcal{I})$ with $\pr(|Y_{\bo}|\geq y)=y^{-\alpha}$ for $y\geq 1$, such that for all finite and nonempty $I\subseteq \mathcal{I}$, 
\begin{align}\label{eq:conv_to_tail_forward}
%(u^{-1}X_{\bi}:\bi \lgeq \bo) 
u^{-1}\bX_{I}
\,\big| \,|X_{\bo}|>u \dto (Y_{\bi})_{\bi \in I} \, , \; \text{as } u\toi \, .
\end{align}
\end{enumerate}
\end{theorem}

Recall that for finite $I\subseteq \Z^d$, $\bX_I$ is multivariate regularly varying with index $\alpha>0$ if for some norm $\|\cdot\|$  on $\R^{|I|}$ there exists a random vector  on $\R^{|I|}$, say $\bTheta^{(I)}$, such that $\|\bTheta^{(I)}\|=1$ and
\begin{align*}
(u^{-1}\|\bX_I\|,\|\bX_I\|^{-1} \bX_I ) \, \big| \, \|\bX_{I}\|> u \dto (Y, \bTheta^{(I)}) 
\, , \; \text{as } u\toi \, ,
\end{align*}
where $Y$ is independent of $\bTheta^{(I)}$ and satisfies $\pr(Y>y)=y^{-\alpha}$ for $y\geq 1$.
%For background on multivariate regular variation see \cite{resnick:1987, resnick:2007}. 

The equivalence between \ref{it:tail_process1} and \ref{it:tail_process3} 
%(cf.\ (\ref{eq:jointconv_spectral}) below) 
explains why fields admitting a tail process will simply be called \textit{regularly varying}. We refer to the corresponding $\alpha$ as the \textit{(tail) index} of the field.

%Note that, in view of stationarity, it is not surprising that \ref{it:tail_process1} and \ref{it:tail_process3} are equivalent.  What is slightly more difficult to understand is that \ref{it:tail_process2} is enough for \ref{it:tail_process3}. The proof of \Cref{thm:tail_process} below actually shows  that in \ref{it:tail_process2}, instead of $\Z^d_{\lgeq}$, it is sufficient to take any $J\subseteq \Z^d$ such that for every nonempty and finite $I\subseteq \Z^d$ there exists at least one $\bi^*\in I$ such that $I-\bi^*\subseteq J$.
\begin{remark}
While writing the paper, we learned of a parallel study by Wu and
Samorodnitsky~\cite{samorodnitsky:wu:2020}  who also consider regularly varying fields but with the emphasis on the various notions of the "extremal indices" in this context and the application of the theory to the Brown-Resnick random fields.
 They show by an example that for $d\geq 2$ existence of the limit of $u^{-1}\bX_{I} \,\big| \,|X_{\bo}|>u$ for all finite  $I\subseteq \mathcal{I}$ when $\mathcal{I}$ is an orthant in $\Z^d$, is not sufficient for regular variation of $\bX$ and hence existence of the tail field. This made us reconsider an earlier (incorrect) version of \Cref{thm:tail_process} and eventually led to a proper extension of \cite[Theorem 2.1(ii)]{basrak:segers:2009}.
  
%One may ask whether it is sufficient in \ref{it:tail_process2} above to find the limit of $(u^{-1}X_{\bi}:\bi \in J) \;\big| \;|X_{\bo}|>u$ for other smaller sets of indices $J\subseteq \Z^d$ such as orthants or  other cones. 
%Corresponding statement is 
%not true in general.
%While writing the paper, we learned of parallel study by
%Samorodnitsky and Wu \cite{samorodnitsky:wu:2018}  who also consider regularly varying fields with emphasis on various notions of extremal index in this context.
% They show by an example that it is not sufficient to take $J$ equal to a quadrant in $\Z^$.
%  However,  the proof of \Cref{thm:tail_process} shows that one can take any $J\subseteq \Z^d$ such that for every non empty and finite $I\subseteq \Z^d$ there exists at least one $\bi^*\in I$ such that $I-\bi^*\subseteq J$.
%%Corresponding statement is 
%%not true in general, we learnt about a nice counterexample to this from Samorodnitsky and Wu  (personal communication). \tcb{However, by examining the proof of \Cref{thm:tail_process}, one can take any $J\subseteq \Z^d$ such that for every non empty and finite $I\subseteq \Z^d$ there exists at least one $\bi^*\in I$ such that $I-\bi^*\subseteq J$.}
%%%In particular, one can replace the lexicographic order in \Cref{thm:tail_process} with any total order on $\Z^d$ which is translation--invariant.
\end{remark}

\subsubsection{The spectral tail field}
Consider now the space $\R^{\Z^d}$ equipped with the product topology and the corresponding Borel $\sigma$-algebra. One can then rephrase (\ref{eq:conv_to_tail}) simply  as 
\begin{align*}%\label{eq:conv_to_tail_RZd}
u^{-1} \bX \, \big| \, |X_{\bo}|>u \dto \bY \: \; \text{in} \; \R^{\Z^d} \, ,
\end{align*} 
see e.g.~\cite[p.~19]{billingsley:1968}. 
%Recall the \textit{spectral tail field} $\bTheta=(\Theta_{\bi}:\bi\in\Z^d)$ of $\bX$ defined by
%$\Theta_{\bi}=Y_{\bi}/|Y_{\bo}|, \bi\in \Z^d$. Continuous mapping theorem applied to (\ref{eq:conv_to_tail_RZd}) yields that 
%\begin{align*}
%\left( \frac{|X_{\bo}|}{u}, \frac{\bX}{|X_{\bo}|}\right) \, \big| \, |X_{\bo}|>u \dto (|Y_{\bo}|, \bTheta) \: \; \text{in} \; (1,\infty)\times\R^{\Z^d} \, ,
%\end{align*}
%which by regular variation of $|X_{\bo}|$ implies that $|Y_{\bo}|$ and $\bTheta$ are independent.
The \textit{spectral tail field} $\bTheta=(\Theta_{\bi}:\bi\in\Z^d)$ of $\bX$ is defined by
$\Theta_{\bi}=Y_{\bi}/|Y_{\bo}|\,, \bi\in \Z^d$. Note that $|\Theta_{\bo}|=1$. Moreover, the spectral field $\bTheta$ is independent of $|Y_0|$ and satisfies 
\begin{align*}%\label{eq:conv_spectral}
|X_{\bo}|^{-1}\bX \;\big| \;|X_{\bo}|>u \dto \bTheta  \: \; \text{in} \; \R^{\Z^d} \, ,
\end{align*}
see \cite[Proposition 2.2.3]{thesis}.

Even though the tail field is typically not stationary, regular variation and stationarity  of the underlying random field $\bX$ yield specific distributional properties of $\bTheta$ (and hence of $\bY$) summarized by the so-called time-change formula: for every integrable (in the sense that one of the expectations below exists) or nonnegative measurable function $h:\R^{\Z^d}\to \R$ and all $\bj \in\Z^d$,
\begin{equation}\label{eq:time-change0}
\ex\left[h\left((\Theta_{\bi-\bj})_{\bi\in\Z^d}\right)\ind{\Theta_{-\bj}\neq 0}\right]=\ex\left[h\left((\Theta_{\bi}/|\Theta_{\bj}|)_{\bi\in\Z^d}\right)|\Theta_{\bj}|^{\alpha}\ind{\Theta_{\bj}\neq 0}\right] \; .
\end{equation}

In the case of time series, (\ref{eq:time-change0}) appears in \cite{basrak:segers:2009} and the proof is easily extended to the case of random fields, {see \cite[Theorem 3.2]{samorodnitsky:wu:2020}}. Alternatively, one can arrive at (\ref{eq:time-change0}) following the approach of  \cite{planinic:soulier:2018} who use the so-called tail measure of $\bX$ introduced in \cite{owada:samorodnitsky:2012}, see also \cite{dombry:hashorva:soulier:2018}. 

%In particular, a special, but very useful, case of \cite[Lemma 2.1]{planinic:soulier:2018} is
%\begin{equation}\label{eq:time-change1}
%\ex\left[h\left((Y_{\bi})_{\bi\in\Z^d}\right)\ind{|Y_{-\bj}|>1}\right]=\ex\left[h\left((Y_{\bi + \bj})_{\bi\in\Z^d}\right)\ind{|Y_{\bj}|>1}\right] \; .
%\end{equation}
%for all $\bj \in\Z^d$ and all $h$ as above.

\begin{remark}\label{rem:convToSpectralIsSuff}
Let $\bX$ be a stationary random field and $\alpha>0$. If  $\lim_{u\toi}\pr(|X_{\bo}|>uy)/\pr(|X_{\bo}|>u)\to y^{-\alpha}$ for all $y>0$ 
%is a regularly varying random variable with index $\alpha>0$.
%To verify that $\bX$ is a regularly varying (with index $\alpha>0$) it is enough to show that 
and for some encompassing $\mathcal{I}\subseteq \Z^d$ there exist random variables $(\Theta_{\bi} : \bi \in \mathcal{I})$ such that for all finite and nonempty $I\subseteq \mathcal{I}$, $|X_{\bo}|^{-1}\bX_{I} \;\big| \;|X_{\bo}|>u \dto (\Theta_{i})_{i\in I}$, then $\bX$ is regularly varying with index $\alpha$; combine the proof of \cite[Corollary 3.2]{basrak:segers:2009} and \Cref{thm:tail_process}.
If $\mathcal{I}\neq \Z^d$, the distribution of the whole spectral process $\bTheta$ is then determined by (\ref{eq:time-change0}) and the tail field of $\bX$ is given by $\bY=Y\bTheta$ where $Y$ is independent of $\bTheta$ and satisfies $\pr(Y\geq y)=y^{-\alpha}$ for $y\geq 1$.
\end{remark}

\subsection{Convergence to a compound Poisson process}\label{sub:convToComp}
Denote by $\cleq$ the componentwise order on $\Z^d$, thus  for $\bi=(i_1,\dots,i_d),\bj=(j_1,\dots, j_d)\in\Z^d$,
$\bi\cleq \bj$ if $i_k\leq j_k$ for all $k=1,\dots, d$.
Take a sequence of positive integers $(r_n)$ such that $\lim_{n\toi}r_n=\lim_{n\toi}n/r_n=\infty$ and let $k_n=\lfloor n / r_{n} \rfloor$. For each $n\in\N$, decompose $\{1,\dots,n\}^d$ into blocks $J_{n,\bi}$, $\bi \in I_n :=\{1,\dots,k_n\}^d$, of size $r_n^d$ by
\begin{equation}\label{eq:blocks_indices}
J_{n, \bi}=(\bj\in \Z^d :(\bi-\bone)\cdot r_n + \bone\cleq \bj \cleq \bi \cdot r_n) \, .
\end{equation}
In this section we apply the Poisson approximation theory from \Cref{sec:ComPoApp} to the point processes based on the (increasing) blocks 
\begin{equation*}%\label{eq:blocks}
\bX_{n,\bi}:=
%\bX_{n,\bi}(r_n):=
%(X_{\bj}: \bj \in J_{n,i}), 
\bX_{J_{n,\bi}},
\; \bi \in I_n \, .
\end{equation*}
Following \cite{basrak:planinic:soulier:2018}, we first introduce a suitable space for the $\bX_{n,\bi}$'s; for details see \cite[Subsection 2.3.1]{thesis}.

\subsubsection{A space for blocks - $\lo$}
\label{subsub:lo}
Let $l_0$ be the space of all $\R$-valued arrays on $\Z^d$ converging to zero in all directions, i.e.\ $l_0 = \{(x_{\bi})_{\bi\in \Z^d} : \lim_{|\bi|\to\infty}|x_i|= 0\}$, where  $|\bi|=\max_{k=1,\dots,d}|i_k|$ for $\bi=(i_1,\dots,i_d)\in \Z^d$.  On $l_0$ consider the
uniform norm
\begin{align*}
\|\bx\| = \max_{\bi\in\Z^d} |x_{\bi}| \, , \, \bx=(x_{\bi})_{\bi\in \Z^d} \, ,
\end{align*}
which makes $l_0$ into a separable Banach space. Also, denote by $\bo\in l_0$ the array consisting only of 0's.
%Indeed, $l_0$ is the closure of all rational--valued arrays on $\Z^d$ with at most finitely many non--zero terms in the Banach space of all bounded $\R$--valued arrays on $\Z^d$.  
%Note that we can now write (\ref{eq:convTozero}) as $\pr(\bY\in l_0)=1$.

%Define the family of shift operators $B^{\bk}$, $\bk \in \Z^d$, on $\R^{\Z^d}$ by $B ^{\bk}(x_{\bi})_{\bi} = (x_{\bi+\bk})_{\bi}$ and
Introduce an equivalence relation $\sim$ on $l_0$ by letting $\bx \sim \by $ for $\bx,\by\in l_0$ if for some $\bj\in \Z^d$, $y_{\bi}=x_{\bi+\bj}$ for all $\bi\in \Z^d$. In the sequel, we consider the quotient space $
\lo = l_0/\sim$ of shift-equivalent arrays.
Observe, for $\tbx\in \lo$ and an arbitrary $\bx = (x_{\bi})_{\bi}\in \tbx$, $\tbx=\{(x_{\bi+\bj})_{\bi} : \bj\in \Z^d\}$.
Further, metric $\dtilde : \lo\times\lo\longrightarrow [0,\infty)$ defined by
\begin{align}\label{eq:dtilde}
\dtilde(\tilde{\bx},\tilde{\by}) 
=\inf\{\|\bx-\by\|:\bx\in\tilde{\bx},\by\in\tilde{\by}\} 
%=\inf_{\bk,\bk'\in\Z^d}\|B^{\bk}\bx-B^{\bk'}\by\|_\infty 
\, , \, \tilde{\bx},\tilde{\by}\in \lo \, ,
\end{align}
makes $\lo$ a separable and complete metric space. Note that 
%by (\ref{eq:lo_metric_property}), 
for $\tilde{\bx}, \tilde{\bx}_1 , \tilde{\bx}_2,\dots \in \lo$,   $\dtilde(\tilde{\bx}_n, \tilde{\bx})\to 0$ as $n\toi$ if and only if for some, and then for every, $\bx \in \tilde{\bx}$ there exists $\bx_n \in \tilde{\bx}_n$, $n\in \N$, such that $\|\bx_n - \bx\| \to 0$.

%For any finite $I\subseteq \Z^d$ we will consider a finite array $\bx\in \R^{I}$ as an element of $\lo$ by simply adding infinitely many zeros around $\bx$ and then mapping this element of $l_0$ into its equivalence class. 

In what follows, on $l_0$ and $\lo$  consider their respective Borel $\sigma$-algebras $\mathcal{B}(l_0)$ and $\mathcal{B}(\lo)$. Call a function $h$ on $l_0$ \textit{shift-invariant} if $h((x_{\bi + \bj})_{\bi})=h((x_{\bi})_{\bi})$ for all $(x_{\bi})_{\bi}\in l_0, \bj\in \Z^d$. Note that $\mathcal{B}(l_0)$ coincides with the trace $\sigma$-algebra of $l_0$ in $\R^{\Z^{d}}$ considered with respect to its cylindrical $\sigma$-algebra, and a function $\tilde{h}$ on $\lo$ is measurable if and only if the function $\bx \mapsto \tilde{h}(\tbx)$ is a (shift-invariant) measurable function on $l_0$.
%Further, call a set $B\subseteq l_0$  \textit{shift-invariant} if $\bx\in B$ implies that $B^{\bk} \bx \in B$ for all $\bk\in \Z^d$. 

%Since topologies of $l_0$ and $\lo$ are Polish and the corresponding quotient map $\pi:l_0 \to \lo$ is continuous, \cite[Corollary A.2.5]{berberian:1988} implies that for every $\tilde{B}\subseteq \lo$, $\tilde{B}\in \mathcal{B}(\lo)$ if and only if $\pi^{-1}(\tilde{B})\in \mathcal{B}(l_0)$. In other words,
% $\mathcal{B}(\lo)$ coincides with the family of sets $\pi(B)$, for $B\in \mathcal{B}(l_0)$ which are shift-invariant. Moreover, a function $\tilde{h}$ on $\lo$ is measurable if and only if $\tilde{h}=h \circ \pi$ for some shift-invariant measurable function $h$ on $l_0$. 

\subsubsection{The point process of blocks}

%Recall the space $\lo$ of shift-equivalent arrays converging to zero defined in \Cref{subsub:lo}. 
Consider now the space $\loo:=\lo \setminus\{\bo\}$ with a Borel subset $B\subseteq \loo$ being bounded if for some $\epsilon>0$, $\|\bx\|>\epsilon$ for all $\bx\in B$. In other words, bounded sets are those which are bounded away from $\bo$ w.r.t.\ the metric~$\dtilde$ defined in (\ref{eq:dtilde}).

We will consider the finite block $\bX_{n,\bi}$ as an element of $\lo$ by simply adding infinitely many zeros around $\bX_{n,\bi}$ and then mapping the resulting element of $l_0$ into its equivalence class in $\lo$.
\begin{remark}\label{rem:why_lo}
One can regard blocks as elements of the simpler space $l_0$ but since we are interested in clusters of high-threshold exceedances in these blocks one would also need to specify a reference exceedance (an \textit{anchor}, see \Cref{subs:anchor} below) around which the block is centered, that is, which exceedance is put at position $\bo$. This introduces additional technical difficulties. For example, one natural choice for the anchor is the first  maximum of the block (e.g.\ w.r.t.\ the lexicographic order on $\Z^d$). In this case one encounters continuity issues since it is possible that the limiting cluster can with positive probability have two exceedances of the exactly same magnitude  (e.g.\ take a moving average process from \Cref{exa:moving_averages} below which has at least two identical non-zero coefficients). Consequently, to deduce the limiting behavior of the extremal clusters with this choice of an anchor one would need to exclude such cases by e.g.\ imposing a suitable condition on the tail process.

Another choice for the anchor could be the first exceedance over a (high) threshold but this (i) is dependent on the choice of the threshold, and (ii) in this case one does not have the nice polar decomposition of the limiting cluster, see \Cref{lem:spectral_Z} and \Cref{rem:shift_invariance} below. 

 On the other hand, the use of $\lo$ is immune to these issues and allows one to develop a general point process convergence theory while keeping all the relevant information about the structure within the extremal clusters, see \cite{basrak:planinic:2019} for an application of the theory to the study of sums and records times of regularly varying time series.
\end{remark}

Define the point process of blocks
\begin{align*}%\label{eq:PPofBlocks}
N_n' = \sum_{\bi\in I_n} \delta_{(\bi/k_n,\bX_{n,\bi}/a_n)} \, , \, n\in\N \, ,
\end{align*}
in $\mathcal{M}_p([0,1]^d \times \loo)$,
where the sequence $(a_n)$ is chosen such that 
\begin{equation*}%\label{eq:a_n}
\lim_{n\toi} n^d\pr(|X_{\bo}|>a_n)= 1 \, .
\end{equation*}
To obtain the convergence of $N_n'$ we will apply \Cref{cor:poissonapprox_stationarycase}.

For each $n\in \N$, denote  $J_{r_n}:=\{1,\dots,r_n\}^d=J_{n,\bone}$ and let $\bX_{r_n}:=\bX_{J_{r_n}}$ represent the common distribution of the blocks $\bX_{n,\bi}$, $\bi \in I_n$. Under this notation, the condition  \eqref{eq:independentblocksconvergence_idcase} reduces to the existence of a measure $\nu$ in $\mathcal{M}(\loo)$ satisfying
\begin{align}\label{eq:v_ntov}
k_n^d \pr(a_n^{-1}\bX_{r_n}\in \cdot)\vto \nu \, , \, \text{as } n\toi \, .
\end{align}

Property (\ref{eq:AC}) below provides one sufficient condition for this convergence to  hold. It appears in the time series literature under the name finite mean cluster size condition or the anticlustering condition. 

\begin{hypothesis}\label{hypo:AC}
There exists a sequence of positive integers $(r_n)_{n}$ satisfying $r_n\toi$, $r_n/n\to 0$, and for every $u>0$,
\begin{align}\label{eq:AC}
\lim_{m\toi}\limsup_{n\toi}\pr\left(\max_{m< |\bi|\leq r_n}|X_{\bi}|>a_nu\; \Big| \; |X_{\bo}|>a_nu\right)=0 \, .
\end{align}
\end{hypothesis}

As shown in \Cref{prop:intensity_convergence} below, for a sequence $(r_n)$ satisfying (\ref{eq:AC}),
the convergence in \eqref{eq:v_ntov} holds 
with the limiting measure $\nu$ of the form 
 \begin{equation*}
 \nu(\, \cdot \, )=\vartheta\int_0^\infty \pr(y\tilde{\bQ} \in \cdot) \alpha y^{-\alpha - 1} dy    \, , 
\end{equation*}
for some $\vartheta\in (0,1]$ and $\tilde{\bQ}$ being a random element in $\lo$ satisfying $\|\tilde{\bQ}\|=1$ almost surely. In the following we first describe $\vartheta$ and $\tilde{\bQ}$ in terms of the tail field of $\bX$ using the concept of anchoring. 

\subsubsection{Anchoring the tail process}\label{subs:anchor}
%\label{sub:AnchoredTailProcess}
From now on we will restrict our attention to tail fields $\bY=(Y_{\bi})_{\bi \in \Z^d}$ which satisfy 
\begin{align*}%\label{eq:convTozero}
\pr(\lim_{|\bi|\toi}|Y_{\bi}|=0)= \pr(\bY\in l_0)=1 \, .
\end{align*}
For example, this is true whenever the underlying random field $\bX=(X_{\bi})_{\bi \in \Z^d}$ satisfies Assumption \ref{hypo:AC} 
%below 
(cf.~\cite[Proposition 4.2]{basrak:segers:2009}).

%We start by briefly introducing two spaces of arrays used in the sequel; for details see \cite[Subsection 2.3.1]{thesis}.

%Recall that $|Y_{\bo}|>1$  so in particular $\|\bY\|>1$. 
We say that a measurable function $A:\{\bx \in l_0 : \|\bx\|>1 \} \to \Z^d$ is an \textit{anchoring function} if 
\begin{enumerate}[(i)]
\item $A((x_{\bi})_{\bi \in \Z^d})=\bj$ for some $\bj \in \Z^d$ implies that $|x_{\bj}|>1$;
\item For each $\bj\in \Z^d$, $A((x_{\bi+\bj})_{\bi})=A((x_{\bi})_{\bi })-\bj$.
\end{enumerate}

In words, $A$ picks one of the finitely many $x_{\bi}$'s which are larger than one in absolute value in a way which is translation covariant. Observe, for an arbitrary group order on $\Z^d$, the following are an examples of an anchoring function.
\begin{itemize}
\renewcommand\labelitemi{--}
\item \textit{first exceedance}: $A^{fe}((x_{\bi})_{\bi})=\min \{\bj\in \Z^d : |x_{\bj}|>1\}$,
\item \textit{last exceedance}: $A^{le}((x_{\bi})_{\bi})=\max \{\bj\in \Z^d : |x_{\bj}|>1\}$,
\item \textit{first maximum}: $A^{fm}((x_{\bi})_{\bi})=\min \{\bj\in \Z^d : |x_{\bj}|=\|(x_{\bi})_{\bi}\|\}$.
\end{itemize}

We will exploit the following property of the tail field which is implied solely by the stationarity of $\bX$, and can be seen as a special case of the time-change formula.

\begin{lemma}
For every bounded measurable function $h:\R^{\Z^d}\to \R$ and all $\bj\in \Z^d$,
\begin{equation}\label{eq:time-change1}
\ex\left[h\left((Y_{\bi})_{\bi\in\Z^d}\right)\ind{|Y_{\bj}|>1}\right]=\ex\left[h\left((Y_{\bi - \bj})_{\bi\in\Z^d}\right)\ind{|Y_{-\bj}|>1}\right] \, .
\end{equation}
\end{lemma} 
 
\begin{proof}
Assume in addition that $h$ is continuous with respect to the product topology on $\R^{\Z^d}$. Then, since $\pr(|Y_{\bj}|=1)=\pr(|Y_{\bo}|\cdot |\Theta_{\bj}|=1)=0$ for all $\bj\in \Z^d$, the definition of the tail process and stationarity of $(X_{\bi})$ imply
\begin{align*}
\ex\left[h\left((Y_{\bi})_{\bi}\right)\ind{|Y_{\bj}|>1}\right]&=\lim_{u\toi} \ex\left[h\left((u^{-1}X_{\bi})_{\bi}\right)\ind{|X_{\bj}|>u} \mid |X_{\bo}|>u\right] \\
&=\lim_{u\toi} \frac{\ex\left[h\left((u^{-1}X_{\bi})_{\bi}\right)\ind{|X_{\bj}|>u, |X_{\bo}|>u}\right]}{\pr(|X_{\bo}|>u)} \\
&=\lim_{u\toi} \frac{\ex\left[h\left((u^{-1}X_{\bi-\bj})_{\bi}\right)\ind{|X_{\bo}|>u, |X_{-\bj}|>u}\right]}{\pr(|X_{\bo}|>u)} \\
&=\ex\left[h\left((Y_{\bi - \bj})_{\bi}\right)\ind{|Y_{-\bj}|>1}\right] \, .
\end{align*}
Since finite Borel measures on a metric space are determined by integrals of continuous and bounded functions, this yields (\ref{eq:time-change1}). 
\end{proof} 
\begin{remark}
Using the already mentioned tail measure of $\bX$, one can give a one-line proof of the previous result, see~\cite[Lemma 2.2]{planinic:soulier:2018}.
\end{remark}
 
\begin{lemma}
Assume that $\pr(\bY\in l_0)=1$. Then for every anchoring function $A$
\begin{align*}
\pr(A(\bY)=\bo)>0 \, .
\end{align*}
\end{lemma}

\begin{proof}
Assume that $\pr(A(\bY)=\bo)=0$. Applying \eqref{eq:time-change1} yields
\begin{align*}
1&=\sum_{\bj \in \Z^d} \pr(A(\bY)=\bj) = \sum_{\bj \in \Z^d } \pr(A(\bY)=\bj, |Y_{\bj}|>1) \\
&= \sum_{\bj \in \Z^d } \pr(A((Y_{\bi - \bj})_{\bi})=\bj, |Y_{-\bj}|>1)=\sum_{\bj \in \Z^d }  \pr(A(\bY)=\bo , |Y_{-\bj}|>1) =0 \; . \\ 
\end{align*}
Hence, $\pr(A(\bY)=\bo)>0$.
\end{proof}

If $\pr(\bY\in l_0)=1$, for any anchoring function $A$ we define the \textit{anchored  tail process}  of $\bY$ (with respect to $A$) as any random element of $l_0$, denoted by  $\bZ^A=(Z^A_{\bi}: \bi \in \Z^d)$, which satisfies 
\begin{align*}
\bZ^A \eind  \bY \;\big| \; A(\bY)=0 \, .
\end{align*}
Also, define $\bQ^A=(Q^A_{\bi}: \bi \in \Z^d)$ by 
$\bQ^A = Z^A / \|\bZ^A\|$ and call it the \textit{anchored spectral  tail process}  (with respect to $A$).

%It turns out that the anchored tail processes are just randomly shifted versions of one another. 
%\tcb{The following results shows that, up to a random shift, the distribution of the anchored tail process does not depend on the anchoring function.}

\begin{lemma}
Assume that $\pr(\bY\in l_0)=1$ and let $A,A'$ be two anchoring functions. Then
\begin{align*}
\pr(A(\bY)=\bo)=\pr(A'(\bY)=\bo) 
\end{align*}
and
\begin{align*}
\bZ^{A} \eind \bZ^{A'} \,  \, \text{in } \lo \, .
\end{align*}
\end{lemma}

\begin{proof}
Let $h:l_0 \to [0,\infty)$ be an arbitrary measurable and shift-invariant function. Using \eqref{eq:time-change1} and shift-invariance of $h$ we obtain
\begin{align*}
\ex[h(\bY) \ind{A(\bY)=\bo}] &= \sum_{\bj \in \Z^d} \ex[h(\bY) \ind{A(\bY)=\bo, A'(\bY)=\bj, |Y_{\bj}|> 1}] \\
&= \sum_{\bj \in \Z^d} \ex[h(\bY) \ind{A(\bY)=-\bj ,  A'(\bY)=\bo}] \\
& = \ex[h(\bY) \ind{A'(\bY)=\bo}] \; .
\end{align*}
Taking $h \equiv 1$ yields the first statement, and then the second one follows immediately by the construction of the space $\lo$.
%Finally, if $h$ is in addition shift-invariant, then
%\begin{align*}
%\ex[h(\bZ^{A'})]=\ex\left[h((Z_{\bi+A'(\bZ^{A})}^{A})_{\bi})\right] =  \ex\left[h((Z_{\bi}^{A})_{\bi})\right]\, ,
%\end{align*}
%which yields the last statement.
\end{proof}

%On $\Z^d$ we will also consider the lexicographic order, i.e.\ for indices $\bi=(i_1,\dots,i_d),\bj=(j_1,\dots, j_d)\in\Z^d$, write $\bi\lless \bj$ if $i_k < j_k$ for the first $k$ where $i_k$ and $j_k$ differ. Moreover, if $\bi\lless \bj$ or $\bi= \bj$, we write 
% $\bi\lleq \bj$.
%Note that $\lleq$ makes $\Z^d$ a totally ordered space. In particular,  any finite subset of $\Z^d$ has a unique minimal and maximal element.

If $\pr(\bY\in l_0)=1$, denote by $\vartheta$ the common value of $\pr(A(\bY)=\bo)$, i.e.\ for an arbitrary anchoring function $A$ set
\begin{align}\label{eq:extremal_index}
\vartheta=\pr(A(\bY)=\bo) \, .
\end{align}
In particular, for any group order $\lleq$ on $\Z^d$, using the first/last exceedance as an anchor yields, 
\begin{equation*}%\label{eq:extremal_index_exceedance}
\vartheta=\pr(\sup_{\bj \lless \bo} |Y_{\bj}|\leq 1) = \pr(\sup_{\bj \lgreater \bo} |Y_{\bj}|\leq 1) \, 
\end{equation*}
since $\pr(|Y_{\bo}|>1)=1$.
Also, 
\begin{equation}\label{eq:extremal_index_maximum}
\vartheta=\pr(A^{fm}(\bY)=0)= \pr(A^{fm}(\bTheta)=0) \, .
\end{equation}
Observe here that the function $A^{fm}$ remains  well defined on the whole set $l_0$ without $ \bo $. Under suitable dependence conditions, $\vartheta$ turns out to be the extremal index of the field $(|X_{\bj}|)_{\bj}$, see \Cref{rem:extremal_index} below (cf.~also \Cref{rem:typical_cluster}).

Furthermore, second part of the previous result  
shows that the distribution of the anchored tail process, when viewed as an element in $\lo$, does not depend on the anchoring function. Hence, there exists a random element in $\lo$, denoted by $\tilde{\bZ}$, which satisfies
\begin{align*}%\label{eq:Q's_definition}
\tilde{\bZ} \eind \bZ^A  \quad \text{in} \; \: \lo 
%\quad \mbox{ and } \quad \bQ\eind \bQ^A\,,
\end{align*}
for all anchoring functions $A$; simply take your favorite anchoring function $A$ and let $\tilde{\bZ}$ be the equivalence class of $\bZ^A$ in $\lo$. Moreover, let $\tilde{\bQ}= \tilde{\bZ}/\|\tilde{\bZ}\|$, so in particular $\tilde{\bQ} \eind \bQ^{A}$ in $\lo$ for any anchor $A$. We will also refer to $\tilde{\bZ}$ and $\tilde{\bQ}$ as the \textit{anchored tail process} and the \textit{anchored spectral tail process}, respectively.

%%\begin{remark}
%Note that for any representative $\bZ = (Z_{\bi})_{\bi\in \Z^d}$ of $\tilde{\bZ}$ and any anchor $A$,
%\begin{align*}
%(Z_{\bi + A(\bZ)})_{\bi} \eind \bZ^A \, \, \text{in } l_0 \, .
%\end{align*}
%%\end{remark}

Under an appropriate assumption, the distribution of the anchored tail process $\tilde{\bZ}$ represents the distribution of the asymptotic cluster of exceedances of the underlying field $\bX$ and one can think of it as a "typical" cluster of exceedances; see \Cref{rem:typical_cluster} below. On the other hand, due to the conditioning, the distribution of the tail process $\bY$ exhibits bias towards clusters with more exceedances. This Palm-like relationship between the typical cluster and the tail process is made formal in the following result and has links with the recent work of Sigman and Whitt~\cite{sigman:whitt:2019} who studied Palm distributions of marked point processes on $\Z$.
%, cf.\ \cite[Formulas (19) and (26)]{sigman:whitt:2019}.

A random element $\boldsymbol{R}=(R_{\bi})_{\bi \in \Z^d}$ in $l_0$ is called a \textit{representative} of a random element $\tilde{\boldsymbol{R}}$ in $\lo$ if  $\boldsymbol{R} \eind \tilde{\boldsymbol{R}}$ in $\lo$. 
 In particular, for any anchoring function $A$, $\bZ^A$ and $\bQ^A$ become representatives of $\tilde{\bZ}$ and $\tilde{\bQ}$, respectively.

{\begin{proposition}\label{prop:Palm1}
Assume that $\pr(\bY\in l_0)=1$ and let $\bZ=(Z_{\bi})_{\bi\in \Z^d}$ be any representative of $\tilde{\bZ}$. Then for every measurable and shift-invariant function $h:l_0 \to [0,\infty)$,
\begin{align}\label{eq:palm_basic}
\ex[h(\bY)]= \vartheta \ex\left[h(\bZ) \cdot \sum_{\bk \in \Z^d}\ind{|Z_{\bk}|>1}\right] \, .
\end{align}
\end{proposition}}
%\begin{remark}
%In particular, (\ref{eq:palm_basic}) holds when $(Z_{\bi})_{\bi}=(Z_{\bi}^A)_{\bi}$ for any anchoring function $A$.
%\end{remark}
\begin{remark}\label{rem:thetaEquals1overExpClusterSize}
Taking $h\equiv 1$ in (\ref{eq:palm_basic}) yields
\begin{align*}
\vartheta = \frac{1}{\ex\left[\sum_{\bk \in \Z^d}\ind{|Z_{\bk}|>1}\right]} \, ,
\end{align*}
and since $\vartheta>0$ this implies $\ex\left[\sum_{\bk \in \Z^d}\ind{|Z_{\bk}|>1}\right]<\infty$. 
\end{remark}
%\begin{remark}
%Formula (\ref{eq:palm_basic}) was motivated by the recent work of Sigman and Whitt~\cite{sigman:whitt:2019} who studied Palm distributions of marked point processes on $\Z$, cf.\ \cite[Formulas (19) and (26)]{sigman:whitt:2019}. 
%\end{remark}

\begin{proof}[Proof of Proposition \ref{prop:Palm1}]
Fix an arbitrary anchoring function $A$. By the definition of $\bZ^A$,  (\ref{eq:time-change1}) and shift-invariance of $h$ we get
\begin{align*}
\vartheta \ex\left[h(\bZ^A) \cdot \sum_{\bk \in \Z^d}\ind{|Z_{\bk}^A|>1}\right] & = \sum_{\bk \in \Z^d} \ex[h(\bY)\ind{|Y_{\bk}|>1, A(\bY)=0}] \\
&=\sum_{\bk \in \Z^d} \ex[h(\bY)\ind{|Y_{-\bk}|>1, A(\bY)=-\bk}] \\
& = \sum_{\bk \in \Z^d} \ex[h(\bY)\ind{A(\bY)=-\bk}] = \ex[h(\bY)] \, .
\end{align*}
The claim for an arbitrary representative $(Z_{\bi})_{\bi}$ of $\tilde{\bZ}$ now follows since the function $\bx\mapsto h((x_{\bi})_{\bi}) \cdot \sum_{\bk \in \Z^d} \ind{|x_{\bk}|>1}$ on $l_0$ is shift-invariant.
\end{proof}

The next result shows that the polar decomposition of the tail process carries over to the anchored tail process, and gives a representative of the anchored spectral tail process $\tilde{\bQ}$ only in terms of the original spectral tail process $\bTheta$.

%Observe that using $A^{fm}$ as anchor implies that
%\begin{align*}
%(\|\bZ\|, \bQ) & \eind (\|\bY\|, \bY/\|\bY\|) \; \big| \; A^{fm}(\bY)=\bo \\
%& = (|Y_{\bo}|, \bTheta) \; \big| \; A^{fm}(\bTheta)=\bo \, .
%\end{align*}
%%Note that conditionally on $A^{fm}(\bY)=\bo$ (or equivalently on $A^{fm}(\bTheta)=\bo$), $\|\bY\|$ is equal to $|Y_{\bo}|$ and therefore $\bY/\|\bY\| = \bTheta$. 
\begin{lemma}\label{lem:spectral_Z}
Assume that $\pr(\bY\in l_0)=1$. Then $\pr(\|\tilde{\bZ}\| \ge y)=y^{-\alpha}$ for all $y\geq 1$, and $\|\tilde{\bZ}\|$ and $\tilde{\bQ}$ are independent.
%\begin{itemize}
%\renewcommand\labelitemi{--}
%\item 
%\item $\|\bZ\|$ and $\bQ$ are independent.
%\end{itemize}
Moreover, 
\begin{align}\label{eq:Q's}
\tilde{\bQ} \eind \bTheta \: \big| \: A^{fm}(\bTheta)=\bo  \; \text{in} \;\: \lo \, .
\end{align}
\end{lemma}

\begin{proof}
Using $A^{fm}$ as anchor implies that in $\lo$,
\begin{align*}
(\|\tilde{\bZ}\|, \tilde{\bQ}) & \eind (\|\bY\|, \bY/\|\bY\|) \; \big| \; A^{fm}(\bY)=\bo \\
& = (|Y_{\bo}|, \bTheta) \; \big| \; A^{fm}(\bTheta)=\bo \, .
\end{align*}
The result now follows by the properties of the tail process.
\end{proof}
\begin{remark}\label{rem:shift_invariance}
Let $\bZ = (Z_{\bi})_{\bi\in \Z^d}$ be any representative of $\tilde{\bZ}$ and set $\bQ = (Q_{\bi})_{\bi} = (Z_{\bi}/ \|\bZ\|)_{\bi} $.
Since the function $\bx \mapsto \|\bx\|$ on $l_0$ is shift-invariant, the previous result implies that $\pr(\|\bZ\| \ge y)=y^{-\alpha}$, $y\geq 1$. On the other hand, $\|\bZ\|$ and $\bQ$ (as an element in $l_0$) are in general not independent. Still, if $h:l_0 \to [0,\infty)$ is measurable and shift-invariant then
\begin{align*}%\label{eq:polarZ}
\ex[h(\bZ)] = \ex \left[\int_{1}^{\infty} h(y \bQ) \alpha y^{-\alpha - 1} dy \right] \, .
\end{align*}
\end{remark}

\begin{example}\label{exa:moving_averages}
	Let $(\xi_{\bi} : {\bi}\in\Z^d)$ be i.i.d.\ random variables with regularly varying distribution
	with index $\alpha >0$, i.e.\
\begin{align*}
\lim_{u\toi} \frac{\pr(|\xi_{\bo}|>uy)}{\pr(|\xi_{\bo}|>u)} =y^{-\alpha} \, , \, y>0 \, ,
\end{align*}	
	and for some $p\in [0,1]$,
		\begin{align*}
\lim_{u\toi} \pr(\xi_{\bo}>0 \mid |\xi_{\bo}|>u) =p \, , \,\lim_{u\toi} \pr(\xi_{\bo}<0 \mid |\xi_{\bo}|>u) =1-p \, .
\end{align*}	
	 Consider the infinite order moving average process $\bX=(X_{\bi}:\bi\in\Z^d)$ defined by
	\begin{align*}%\label{eq:linProc}
	X_{\bi} = \sum_{\bj \in\Z^d} c_{\bj} \xi_{\bi - \bj}  \, , 
	\end{align*}
	where $(c_{\bj} : \bj\in\Z^d)$ is a field of real numbers satisfying
%	\begin{align}   
%	\label{eq:condition-linear}
\begin{align*}%\label{eq:MA_condOnCoeff}
	 0<\sum_{\bj\in\Z^d} |c_{\bj}|^\delta < \infty \, ,
	 \end{align*}
	% \mbox{ with }
	%\begin{cases}
	%\delta<\alpha & \mbox{ if } \alpha\in      (0,1]  \; , \\
	%\delta<\alpha & \mbox{ if } \alpha\in      (1,2] \mbox{ and }  \EE[\xi_0]=0 \; , \\
	%\delta=2 & \mbox{ if } \alpha>2 \mbox{ and }  \EE[\xi_0]=0 \; , \\
	%\delta = 1 & \mbox{ if } \alpha>1 \mbox{ and }  \EE[\xi_0]\ne0 \; .
	%\end{cases}
%	\end{align}
	for some $\delta>0$ such that $\delta<\alpha$ and $\delta\leq 1$\,.
	It is easily shown (see e.g.~\cite[Section 4.5]{resnick:1987}) that this condition ensures that the series above is absolutely convergent. Note also that $\sum_{{\bj}\in\Z^d} |c_{\bj}|^\alpha<\infty$.	
	 Furthermore, it can be proved as in \cite[Lemma 4.24]{resnick:1987}
	that 
	\begin{align*}
	\lim_{u\to\infty} \frac{\PP(|X_{\bo}|>u)}{\PP(|\xi_{\bo}|>u)}=\sum_{\bj\in\Z^d} |c_{\bj}|^\alpha \; .  %\label{eq:tail-equivalence-linear}
	\end{align*}
	Moreover, extending the arguments of Meinguet and Segers~\cite[Example 9.2]{meinguet:segers:2010}, one can show that the stationary field $\bX$ is jointly regularly varying with index $\alpha$ and the spectral tail field given by 
\begin{align*}
%\label{eq:spectralProc_linearProc}
(\Theta_{\bi})_{\bi \in \Z^d} \eind (\spectral c_{\bi + J}/|c_{J}|)_{\bi \in \Z^d}
\end{align*}		
where $\spectral$ is a $\{-1,1\}$-valued random variable with  
% distribution equal to the spectral measure of $\xi_{\bo}$
$\pr(\spectral=1)=p$, 
	and $J$ an $\Z^d$-valued random variable, independent of $\spectral$, such that $\pr(J=\bj)=|c_{\bj}|^{\alpha}/\sum_{{\bi}\in\Z^d} |c_{\bi}|^\alpha$ for all $\bj \in \Z^d$. 
	
	In particular, $\pr(\bTheta\in l_0)=\pr(\bY\in l_0)=1$. Choosing $A^{fm}$ as the anchoring function (see (\ref{eq:extremal_index_maximum}) and (\ref{eq:Q's})) yields that
	\begin{align*}%\label{eq:linProc_theta}
\vartheta = \frac{\max_{\bj\in \Z^d} |c_{\bj}|^\alpha}{\sum_{\bj \in \Z^d} |c_{\bj}|^\alpha}\, , \quad \tilde{\bQ}\eind\left( \frac{\spectral c_{\bj}}{\max_{\bi\in \Z^d}|c_{\bi}|}  \right)_{\bj\in\ZZ^d} \; \text{in } \lo \, .
\end{align*}

\end{example}

\subsubsection{Intensity convergence}

%Recall that, for each $n\in \N$, $J_{r_n}=\{1,\dots,r_n\}^d$ and $\bX_{r_n}=\bX_{J_{r_n}}$.

The following result is an extension of the case $d=1$ shown in \cite[Lemma 3.3]{basrak:planinic:soulier:2018}. The proof is based on \cite[Theorem 4.3]{basrak:segers:2009} and can be found in~\cite[Section 2.3.5]{thesis} (note that $\tilde{\bQ}$ is there denoted by $\bQ$).

\begin{proposition}\label{prop:intensity_convergence}
If $(r_n)_{n\in \N}$ is a sequence of positive integers satisfying $r_n\toi$, $r_n/n\to 0$, and such that (\ref{eq:AC}) holds, then 
%\begin{equation}\label{eq:v_ntov}
%k_n^d \pr(a_n^{-1}\bX_{r_n}\in \cdot)\vto \nu(\, \cdot \, )=\vartheta\int_0^\infty \pr(y\bQ \in \cdot) \alpha y^{-\alpha - 1} dy    \;\; \text{in} \; \mathcal{M}(\loo) \, .
%\end{equation}
\begin{align}\label{eq:intensity_convergence_limiting_measure}
k_n^d \pr(a_n^{-1}\bX_{r_n}\in \cdot)\vto 
\nu(\, \cdot \, )=\vartheta\int_0^\infty \pr(y\tilde{\bQ} \in \cdot) \alpha y^{-\alpha - 1} dy     \, ,
\end{align}
as $n\toi$ in $\mathcal{M}(\loo)$, 
where $\vartheta\in (0,1]$ and the anchored spectral tail process $\tilde{\bQ}$ are defined in \Cref{subs:anchor}.

% with $\vartheta\in (0,1]$ and random element $\bQ$ in $\lo$  being determined by the tail process $\bY$, see (\ref{eq:extremal_index}) and (\ref{eq:Q's_definition}) below.
%the measure $\nu$ in $\mx$ is given by
%\begin{equation}\label{eq:nu}
%\nu(\cdot) = \vartheta\int_0^\infty \pr(y\bQ \in \cdot) \alpha y^{-\alpha - 1} dy \; .
%\end{equation}
\end{proposition}

\begin{remark}
Note that $\nu$ is a proper element of $\mathcal{M}(\loo)$. Indeed, since $\|\tilde{\bQ}\|=1$, $\nu(\{\bx\in \loo : \|\bx\|>\epsilon\})=\vartheta \epsilon^{-\alpha}<\infty$ for all $\epsilon>0$.
\end{remark}

\begin{remark}
%[\textbf{Limiting cluster of extremes}]
\label{rem:typical_cluster}
Observe, since $\nu(\{\bx : \|\bx\|=u\})=0$ for all $u>0$, 
%(\ref{eq:v_ntov}) with $\nu$ defined in 
(\ref{eq:intensity_convergence_limiting_measure}) implies that 
\begin{align}\label{eq:RVofMrn}
k_n^d \pr(M_{r_n}>a_n u)\to \vartheta u^{-\alpha} \, , \, u>0 \, , 
\end{align} 
as $n\toi$, where $M_{r_n}=\|\bX_{r_n}\|$ is the maximum of the block $\bX_{r_n}$. Moreover, for every $u>0$,
\begin{align*}
\pr((a_n u)^{-1} \bX_{r_n} \in \cdot \mid M_{r_n}>a_n u ) &= \frac{k_n^d \pr((a_n u)^{-1} \bX_{r_n} \in \cdot \, , \, M_{r_n}>a_n u )}{k_n^d \pr(M_{r_n}>a_n u)} \\
&\wto \frac{u^{\alpha}}{\vartheta} \vartheta \int_{u}^{\infty} \pr(u^{-1} y \tilde{\bQ} \in \cdot \,)\alpha y^{-\alpha-1} dy \\
&=\int_{1}^{\infty} \pr( y \tilde{\bQ} \in \cdot \, )\alpha y^{-\alpha-1} dy =\pr(\tilde{\bZ}\in  \cdot \, ) \, ,
\end{align*}
where $\wto$ denotes weak convergence of finite measures and the last line follows from \Cref{lem:spectral_Z}. Hence, for  all $u>0$,
\begin{align}\label{eq:AsympOfCluster}
(a_n u)^{-1} \bX_{r_n} \mid M_{r_n}>a_n u \dto \tilde{\bZ}  \:\: \text{in} \; \lo \, .
\end{align}
Thus, the distribution of the anchored tail process $\tilde{\bZ}$ is the asymptotic distribution of a cluster of extremes of $\bX$, i.e.\ block of size $r_n^d$ with at least one exceedance over the level $a_n u$. Also, we identify the anchored spectral process $\tilde{\bQ}$ by 
\begin{align*}%\label{eq:AsympOfCluster_Q}
M_{r_n}^{-1} \bX_{r_n} \mid M_{r_n}>a_n u \dto \tilde{\bQ}  \:\: \text{in} \; \lo \, .
\end{align*}

%Conversely, (\ref{eq:RVofMrn}) and (\ref{eq:AsympOfCluster}) for some $\vartheta>0$ and  some random element $\tilde{\bZ}$ imply (\ref{eq:v_ntov}) for $\tilde{\bQ}:=\tilde{\bZ}/\|\tilde{\bZ}\|$, this is actually the approach of \cite[Theorem 2.2, Lemma 3.3]{basrak:planinic:soulier:2018}.

%To shed some light on $\vartheta$ and $\bQ$ (and convergence in (\ref{eq:intensity_convergence_1})), note that for such $\nu$, denoting $M_{r_n}=\|\bX_{r_n}\|_{\infty}$, (\ref{eq:intensity_convergence_1}) is equivalent to convergences
%\begin{align*}
%k_n^d \pr(M_{r_n}>a_n u)\to \vartheta u^{-\alpha}
%\quad \text{ and } \quad 
%%\end{align*} 
%%and
%%\begin{align*}
%\frac{\bX_{r_n}}{M_{r_n}} \; \big| \; M_{r_n}>a_n u \dto \bQ  \; \text{in } \lo \, ,
%\end{align*}
%for all $u>0$, see \cite[Lemma 4.2]{planinic:soulier:2018}. 

%\end{remark}
%
%\begin{remark}
In fact, convergences (\ref{eq:RVofMrn}) and (\ref{eq:AsympOfCluster}) 
%(or (\ref{eq:AsympOfCluster_Q})) 
imply 
%the intensity convergence (\ref{eq:v_ntov}) with the limiting measure $\nu$ as in 
(\ref{eq:intensity_convergence_limiting_measure}), this is actually the approach in~\cite[Lemma 3.3]{basrak:planinic:soulier:2018}. 
\end{remark}

\subsubsection{Point process convergence}
\label{subs:PP_conv}
 Following \cite{basrak:planinic:soulier:2018} we give a convenient convergence determining family for point processes on $[0,1]^d\times \loo$ (see \Cref{defn:conv_det_fam}).
For an element $\bx\in \lo$  and any $\delta>0$ denote by $\bx^\delta \in \lo$ the equivalence class of the sequence $(x_{\bi}\ind{|x_{\bi}|>\delta})_{\bi}$, where $(x_{\bi})_{\bi}\in l_0$ is an arbitrary representative of $\bx$. 
%Let $\F_0$ be the family of all functions $f\in CB_b^+(\loo)$ such that for some $\delta>0$, $f(\bx)=f(\bx^\delta)$ for all $\bx\in\lo$, where we set $f(\bo)=0$, i.e.\ $f$ depends only on coordinates greater than $\delta$ in absolute value. 
Let $\F_0'$ be the family of all functions $f\in CB_b^+(\loo)$ such that for some $\delta>0$, $f(\bx)=f(\bx^\delta)$ for all $\bx\in\lo$, where we set $f(\bo)=0$, i.e.\ $f$ depends only on coordinates greater than $\delta$ in absolute value. As shown in \cite[Lemma 2.5.2 and Remark 2.5.4]{thesis}, $\F_0'$ is convergence determining in the sense of \Cref{defn:conv_det_fam}.

In view of \Cref{prop:intensity_convergence}, our main result now follows by an application of \Cref{cor:poissonapprox_stationarycase}.

%the following point process convergence result is a direct application of \Cref{cor:poissonapprox_stationarycase} and generalizes \cite[Theorem 3.6]{basrak:planinic:soulier:2018} to the $d$-dimensional setting. The only thing left to verify is that a Poisson point process with mean measure $\nu$ can be represented as in (\ref{eq:PPconvinLo}) below, for details see~\cite[Theorem 2.3.4]{thesis}. 

\begin{theorem}\label{thm:PPconv}

 Let $\bX$ be a stationary regularly varying random field with tail index $\alpha>0$. Assume that $(r_n)_{n\in \N}$ is  a sequence of positive integers satisfying $r_n\toi$, $r_n/n\to 0$, such that (\ref{eq:AC}) holds 
  %If a sequence $(r_n)_n$ is such that Assumption \ref{hypo:AC} holds 
  and the family $((\bi/k_n,\bX_{n,\bi}/a_n): n\in \N, \: \bi \in I_n)$ is $AI(\F_0')$.
  %for some convergence determining family $\F'\subseteq CB_b^{+}([0,1]^d \times \loo)$. 

Then 
\begin{align}\label{eq:PPconvinLo}
N_n' = \sum_{\bi\in I_n} \delta_{(\bi/k_n,\bX_{n,\bi}/a_n)} \dto N'=\sum_{i\in \N}\delta_{(\bT_i, P_i \tilde{\bQ}_i)} 
\end{align}  
in $\mathcal{M}_p([0,1]^d \times \loo)$, where $N'\sim \PPP(\mbox{Leb}\times \nu)$ and
%  Then the point processes $N_n' = \sum_{\bi\in I_n} \delta_{(\bi/k_n,\bX_{n,\bi})}$ converge in distribution in $\mathcal{M}_p([0,1]^d \times \lo\setminus\{\bo\})$
%to a Poisson point process with intensity measure $Leb\times\nu$. Moreover, the limit has the same distribution as the point process
%\begin{equation}\label{eq:PPPrepr}
%N'=\sum_{i\in \N}\delta_{(\bT_i, P_i(Q^i_{\bj})_{\bj} )} 
%\end{equation}
%  where
  \begin{enumerate}[(i)]
  \item $\sum_{i\in \N}\delta_{(\bT_i,P_i)}$ is a Poisson point process on $[0,1]^d\times(0,\infty)$
   with intensity measure $\vartheta \mbox{Leb} \times \alpha y^{-\alpha-1} dy$;
  \item $(\tilde{\bQ}_i)_{i\in \N}$ is a sequence of i.i.d.\ elements in $\lo$, independent of
    $\sum_{i\in \N}\delta_{(\bT_i, P_i)}$ and with common distribution equal to the distribution of the anchored spectral process
    $\tilde{\bQ}$.
    % defined in \eqref{eq:Q's}.
  \end{enumerate}
\end{theorem}

\begin{proof}
%In view of \Cref{prop:intensity_convergence}, \Cref{cor:poissonapprox_stationarycase} implies that $\sum_{\bi\in I_n} \delta_{(\bi/k_n,\bX_{n,\bi}/a_n)} \dto N'$ in $\mathcal{M}_p([0,1]^d \times \loo)$ with $N'\sim \PPP(Leb\times \nu)$. Since 
%\begin{align*}
%\sup_{\bi\in I_n} |\bi/k_n - \bi r_n/n| \leq \frac{r_n}{n} \to 0 \, , \, n\toi \, ,
%\end{align*}
%it follows easily that also $N_n'\dto N'$ in $\mathcal{M}_p([0,1]^d \times \loo)$, cf. the proof of \cite[Corollary 2.3.15]{thesis}. 
The only thing left to verify is that $N'\sim \PPP(\mbox{Leb}\times \nu)$ can be represented as in (\ref{eq:PPconvinLo}) but this follows easily using standard arguments; for details see~\cite[Theorem 2.3.4]{thesis}.
\end{proof}

\begin{remark}
\label{rem:explanation}
If $(Q_{\bj}^i)_{\bj \in \Z^d}$, $i\in \N$ is a sequence of  independent elements of $l_0$ which are representatives of $\tilde{\bQ}$ and independent of $\sum_{i\in \N}\delta_{(\bT_i,P_i)}$, one can construct the limiting process $N'$ simply by considering $\sum_{i\in \N}\delta_{(\bT_i, P_i (Q_{\bj}^i)_{\bj \in \Z^d})}$ as a point process on $[0,1]^d \times \loo$.
%\begin{align*}
%N'\eind  \sum_{i\in \N}\delta_{(\bT_i, P_i (Q_{\bj}^i)_{\bj \in \Z^d})} \; \; \text{in } \mathcal{M}_p([0,1]^d \times \loo) \, .
%\end{align*}
\end{remark}

%%%%%% Stavti negdje kasnije u tekst
To illustrate the meaning of the result in \Cref{thm:PPconv} set $P_i=(\Gamma_i/  \vartheta)^{-1/ \alpha}$ where $\Gamma_i=E_1+\dots+E_i$, $i\in \N$,  with $(E_i)_{i\in \N}$ being i.i.d.\ standard exponential random variables, and let $(\bT_i)_{i\in \N}$ be i.i.d.\ uniform random vectors in $[0,1]^d$ independent of the sequence $(P_i)_{i\in \N}$. Then $\sum_{i\in \N} \delta_{(\bT_i,P_i)}$ is a $\PPP(\vartheta \mbox{Leb} \times \alpha y^{-\alpha-1} dy)$ which in addition satisfies $P_{1}> P_{2} > \dots$\, almost surely. 
%Assume first that the points of the limiting point process in (\ref{eq:PPconvinLo}) are such that  $P_{1}\geq P_{2} \geq \dots$\,. For this simply set $P_i=(\Gamma_i/  \vartheta)^{-1/ \alpha}$ where $\Gamma_i=E_1+\dots+E_i$, $i\in \N$,  with $(E_i)_{i\in \N}$ being i.i.d.\ standard exponential random variables, and let $(\bT_i)_{i\in \N}$ be i.i.d.\ uniform random vectors in $[0,1]^d$ independent of the sequence $(P_i)_{i\in \N}$. 
Consequently, if $\bX_{n,(i)}$ and $T_{n,(i)}$, $i=1,2,\dots, k_n^d$, denote the original blocks $\bX_{n,\bi}$ and their positions  $\bi/k_n$, $\bi \in I_n$, but relabeled so that
\begin{align*}
\|\bX_{n,(1)}\| \geq \|\bX_{n,(2)}\| \geq \dots \geq \|\bX_{n,(k_n^d)}\|  \, , 
\end{align*} 
the continuous mapping theorem applied to (\ref{eq:PPconvinLo}) for every $k\in \N$ yields the convergence
\begin{align*}
\left(T_{n,(i)}, \bX_{n,(i)}/a_n \right)_{i=1,2,\dots,k} \dto \left(T_{i}, P_{i} (Q^{i}_{\bj})_{\bj\in \Z^d} \right)_{i=1,2,\dots,k} \, ,
\end{align*}
in the space $([0,1]^d \times \lo)^{k}$ (to show that the corresponding mapping is a.s.\ continuous w.r.t.\ the limit in (\ref{eq:PPconvinLo}) use \cite[Proposition 2.8]{basrak:planinic:2019}).

Furthermore, by applying the continuous mapping theorem to (\ref{eq:PPconvinLo}) and using similar arguments as in \cite[Proposition 1.34]{krizmanic:2010}, one obtains the following convergence of point processes on a simpler state space;
% \Cref{cor:PPconv_back_to_R} which is a generalization of \cite[Theorem 3.1]{basrak:tafro:2016} to $d$-dimensions; 
 the details can be found in \cite[Corollary 2.3.15]{thesis}.

\begin{corollary}
\label{cor:PPconv_back_to_R}
{In the notation of \Cref{rem:explanation}}, if there exists a sequence $r_n\toi,\, r_n/n \to 0$ for which (\ref{eq:PPconvinLo}) holds, then, with $J_n=\{1,\dots,n\}^d$,
\begin{align}
\label{eq:ppconvinR}
\sum_{\bj \in J_n}\delta_{(\bj / n , X_{\bj}/a_n)} \dto \sum_{i\in \N} \sum_{\bj\in \Z^d}\delta_{(\bT_i, P_iQ^i_{\bj} )}
\end{align}
in $\mathcal{M}_p([0,1]^d \times (\R\setminus\{0\}))$  with bounded sets being those which are bounded away from $[0,1]^d \times \{0\}$.
% with bounded sets being those which are bounded away from $[0,1]^d \times \{0\}$.
%% point processes  $\sum_{\bj \in J_n}\delta_{(\bj / n , X_{\bj}/a_n)}$ converge in distribution in the space $\mathcal{M}_p([0,1]^d \times (\R\setminus\{0\}))$ to the Poisson cluster process
%%$\sum_{i\in \N} \sum_{\bj\in \Z^d}\delta_{(\bT_i, P_iQ^i_{\bj} )}$.
\end{corollary}

%The proof of the previous result is in \Cref{subs:ppconvinR}; \tcb{it first applies continuous mapping theorem to convergence (\ref{eq:PPconvinLo}) and then uses the fact that (assume for simplicity that $d=1$) time instances $j/n$ and $i/k_n$ for $i\in I_n=\{1,\dots,k_n\}, j\in J_{n,i}=\{(i-1)r_n +1,\dots,ir_n\}$, differ by at most $2 r_n/n$ which tends to zero as $n\toi$.}
%%Note, for $d=1$, the previous result corresponds to \cite[Theorem 3.1]{basrak:tafro:2016}.

Observe that in this convergence one loses the information about the structure of the cluster in the limit, see \cite{basrak:planinic:soulier:2018} for a detailed discussion. 

\begin{remark}
\label{rem:extremal_index}
As noted by \cite[Remark 4.7]{basrak:segers:2009}, when convergence in (\ref{eq:ppconvinR}) holds, the quantity $\vartheta$ is  the \textit{extremal index} of the field $(|X_{\bj}|)_{\bj\in \Z^d}$ since $n^d\pr(|X_{\bo}|>a_n u)\to u^{-\alpha}$ and 
\begin{align*}
\pr(\max_{\bj \in J_n}|X_{\bj}|\leq a_n u)\to \pr\left(\sum_{i\in \N} \1{\{P_i>u\}}=0\right)=e^{-\vartheta u^{-\alpha}} \, ,
\end{align*}
as $n\toi$, for all $u>0$.
\end{remark}

The assumptions of Theorem \ref{thm:PPconv} are straightforward to check in the case of $m$-dependent stationary fields. In general, however, checking these assumptions is not trivial. 
%To assure the asymptotic independence condition, one can impose some form of mixing on the field $\bX$ from the literature, see \cite{basrak:planinic:2018} for the case of time series. However,
% such an approach proves to be too restrictive in the case of random fields, for $d>1$ that is, cf. Theorem 2.1 in~\cite{bradley:2005}, so we will not pursue such an approach here.
Still, one can extend the convergence in (\ref{eq:PPconvinLo}) to fields which can be approximated by $m$-dependent fields, such as spatial infinite order moving average processes from \Cref{exa:moving_averages} as explained in the following remark.

\begin{remark} \label{rem:mdep}

%Note, $m$-dependent regularly varying fields by definition satisfy  all the conditions of \Cref{thm:PPconv} (see \cite{thesis} for details).
%		In general,	establishing $AI(\F_0')$ condition for a given random field and  a family of functions $\F_0'$ directly might be very technical. However, 
%Some random fields allow for a natural and convenient approximation by $m$-dependent regularly varying fields.
%		%	Recall, a random field $\bX=(X_{\bi}: \bi \in \Z^d)$ is said to be $m$--dependent for some $m\in \N$ if for all {finite} $I,J\subseteq \Z^d$ such that $\inf\{|\bi - \bj| : \bi\in I, \bj \in J\}>m$, $\sigma$-algebras $\sigma(X_{\bi} : \bi \in I)$ and $\sigma(X_{\bj} : \bj \in J)$ are independent.
%		To make this more precise, 
		Assume that $\bX=(X_{\bi}: \bi \in \Z^d)$ is a stationary random field  such that  there exists a sequence of stationary regularly varying $m$-dependent fields $\bX^{(m)}=(X_{\bi}^{(m)}: \bi \in \Z^d)$, $m\in \N$, and two sequences of strictly positive real numbers $(b_n)$ and $(d^{(m)})_m$ such that for all $m\in \N$ 
		%\begin{equation}\label{eq:mdep-an} 
		$n^d \pr(|X^{(m)}_{\bo}|> b_n)\to d^{(m)} >0 \, ,$
		%	\end{equation}
		while also for any $u >0$
		\begin{align*}
		%\label{eq:lemmaDR1985}
		\lim_{m\to\infty}\limsup_{n\to\infty} \PP(\max_{\bone\leq \bi  \leq \bone \cdot n}|X_{\bi}^{(m)}-X_{\bi}|>b_n u)=0 \; . 
		\end{align*}	
		Provided that the tail processes of the approximating random fields   $\bX^{(m)}$ behave reasonably as $m\toi$, the process $\bX$ satisfies the Poissonian limiting relation in \eqref{eq:PPconvinLo}, see \cite[Section 2.4.1]{thesis} for details, cf.\ also {Kulik and Soulier~\cite{kulik:soulier:2020}} who study the problem in the time series setting.

\end{remark}

In Section \ref{sec:PrThmAl} below
we show that Theorem \ref{thm:PPconv} can be applied to the random field of (exponentially transformed) scores from the sequence alignment problem. In particular, this is an example of a field with a nontrivial dependence structure, but for which the {asymptotic $\F_0'$-independence property} can be shown to hold. For this purpose we apply Proposition \ref{prop:AI_condition} and for convenience, we rephrase it in this setting and in the form suitable for our needs. 

%First, order elements of $I_n=\{1,\dots,k_n\}^d$ with respect to lexicographic or any other order.
%\begin{proposition}\label{prop:CheckingAI_regVarFields}
%Let $(\bX_{n,\bi} : n\in \N , \bi \in I_n)$ be random elements in $\lo$. 
%If there exists a neighborhood structure $(B_n(\bi) : n\in \N, \, \bi\in I_n)$ such that for all $\epsilon>0$ 
%\begin{align*}
%b_{n,1}^\epsilon&=\sum_{\bi\in I_n}\sum_{\bj\in \tilde{B}_n(\bi)} \pr(\|X_{n,\bi}\|>a_n\epsilon)\cdot \pr(\|X_{n,\bj}\|>a_n\epsilon)\to 0\;, \\
%b_{n,2}^\epsilon &=\sum_{\bi\in I_n}\sum_{\bj\in \tilde{B}_n(\bi)} \pr(\|X_{n,\bi}\|>a_n\epsilon ,\|X_{n,\bj}\|>a_n\epsilon)\to 0\;,
%\end{align*}
%  and for all $f\in \F'$
%  \begin{multline*}
%b_{n,3}(f)=\sum_{\bi\in I_n}\big|\ex\big[e^{-f(\bi/k_n,\bX_{n,\bi}/a_n)} \prod_{\bj\in \tilde{B}^c_n(\bi)} e^{-f(\bj/k_n,\bX_{n,\bj}/a_n)}\big]\\
%-\ex\big[e^{-f(\bi/k_n,\bX_{n,\bi}/a_n)}\big]\cdot \ex\big[\prod_{\bj\in \tilde{B}^c_n(\bi)} e^{-f(\bj/k_n,\bX_{n,\bj}/a_n)}\big]\big| \to 0\; ,
%\end{multline*}
%then the family $((\bi/k_n,\bX_{n,\bi}/a_n): n\in \N, \: \bi \in I_n)$ is $AI(\F')$.
%\end{proposition}

\begin{corollary}\label{prop:CheckingAI_regVarFields}
Let for each $n\in \N$, $(\tilde{\bX}_{n,\bi} : \bi \in I_n)$ be identically distributed random elements in $\lo$ and such that for all $\epsilon>0$,
\begin{align}\label{eq:CheckingAI_regVarFields_1}
\limsup_{n\toi}k_n^d \pr(\|\tilde{\bX}_{n,\bone}\|>a_n\epsilon)<\infty \, . 
\end{align}
If there exists a neighborhood structure $(B_n(\bi) : n\in \N, \, \bi\in I_n)$ such that, denoting $\|B_n\|=\max_{\bi \in I_n} |B_n(\bi)|$,
\begin{enumerate}[(i)]
\item As $n\toi$, $\|B_n\|/k_n^d \to 0$ and for all $\epsilon>0$,
\begin{align}\label{eq:CheckingAI_regVarFields_2}
k_n^d \|B_n\| \max_{\substack{\bi \in I_n \\ \bi\neq \bj\in B_n(\bi)}}\pr(\|\tilde{\bX}_{n,\bi}\|>a_n\epsilon , \,\|\tilde{\bX}_{n,\bj}\|>a_n\epsilon) \to 0\, ;
\end{align}
\item \label{it:CheckingAI} For $n$ big enough, $\tilde{\bX}_{n,\bi}$ is independent of $\sigma(\tilde{\bX}_{n,\bj} : \bj \notin B_n(\bi))$ for each $\bi \in I_n$.
\end{enumerate}
Then the family $((\bi/k_n,\tilde{\bX}_{n,\bi}/a_n): n\in \N, \: \bi \in I_n)$ is $AI(\F'_0)$.
\end{corollary}

\begin{proof}
First, observe that for any sequence $\epsilon_m\searrow 0$ sets $K_m'=[0,1]^d \times \{\bx\in \loo : \|\bx\|>\epsilon_m\}$, $m\in \N$, form a base for the family of bounded sets of $[0,1]^d\times \loo$. Next, regardless of ordering of $I_n=\{1,\dots,k_n\}^d$, $|\tilde{B}_n(\bi)|\leq |B_{n}(\bi)| $ for all $\bi \in I_n$. 
%$|\tilde{B}_n(\bi)|\leq |B_{n}(\bi)|-1 $ for all $\bi \in I_n$ (the $-1$ is since $B_n(\bi)$ always contains $\bi$, and $\tilde{B}_n(\bi)$ does not). 
Since $\tilde{\bX}_{n,\bi}$'s are identically distributed, 
\begin{align*}
b_{n,1}^m &= \sum_{\bi\in I_n}\sum_{\bj\in \tilde{B}_n(\bi)} \pr((\bi/k_n, \tilde{\bX}_{n,\bi}/a_n) \in K_m')\cdot \pr((\bj/k_n, \tilde{\bX}_{n,\bj}/a_n) \in K_m') \\
&\leq k_n^d \|B_n\| \pr(\|\tilde{\bX}_{n,\bone}\|>a_n \epsilon_m)^2 \, .
\end{align*} 
In view of (\ref{eq:CheckingAI_regVarFields_1}), 
$\limsup_{n\toi}b_{n,1}^m\leq (const.) \limsup_{n\toi} \|B_n\|/k_n^d = 0$ for all $m\in \N$. Similarly, (\ref{eq:CheckingAI_regVarFields_2}) implies that $\lim_{n\toi} b_{n,2}^m = 0$ for all $m\in \N$, and by \ref{it:CheckingAI}, $b_{n,3}(f) =0$ for every measurable function $f\geq 0$ on $[0,1]^d\times \loo$ and $n$ big enough.  Applying Proposition \ref{prop:AI_condition} finishes the proof.
\end{proof}

\section{Sequence alignment problem} \label{sec:PrThmAl}
This section is devoted to the proof of \Cref{thm:PPconv_Alignments_intro}. We will use the notation introduced in Section \ref{sec:locSeqAlig_intro} and assume throughout that Assumptions \ref{hypo:negative_drift} and \ref{hypo:E'} hold. 
In particular, $(A_i)_{i\in \Z}$ and $(B_i)_{i\in \Z}$ are independent i.i.d.\ sequences, $S_{i,j}^m=\sum_{k=0}^{m-1}s(A_{i-k},B_{j-k})$ for $i,j\in \Z$ and $m\geq 0$, and 
$S_{i,j}=\sup\{ S_{i,j}^m : m\geq 0 \}$ for $i,j \in \Z$.  

For some of the key technical results in our analysis we are indepted to Hansen~\cite{hansen:2006} who even allows sequences $(A_i)$ and $(B_i)$ to be Markov chains. In the i.i.d.\ setting the corresponding proofs, which rely on change of measure arguments, are much less involved. For an alternative approach based on combinatorial arguments see Dembo et. al.~\cite{dembo:karlin:zeitouni:1994}.

\subsection{The tail field}
Consider the positive stationary field $\bX=(X_{i,j}:i,j\in\Z)$ defined by
$$X_{i,j}=e^{S_{i,j}}, \: i,j\in\Z \, .$$
Observe that by \eqref{eq:exp_tail}, for  $\ts>0$ satisfying $\ex[e^{\ts s(A,B)}]=1$,
\begin{equation}\label{eq:reg_var_tail}
\pr(X_{i,j}>u)\sim \tailcons u^{-\ts}, \, \text{ as } u\toi \, ,
\end{equation}
i.e.\ the marginal distribution of $\bX$ is regularly varying. Moreover, the transformed  field  $\bX$ has a tail field and therefore fits into the framework of Section \ref{sec:RegVarF}.

%More importantly, the transformed  field  $\bX$  has a tail field, i.e.\ it is jointly regularly varying. \tcb{Observe, (\ref{eq:Lind}) implies that $\bX$ satisfies 
%\begin{align*}
%X_{m,m}=\max\{X_{m-1,m-1}e^{s(A_m,B_m)},1\} \, ,
%\end{align*} 
%for each $m \in \Z$, where $X_{m-1,m-1}$ is independent of the i.i.d.\ sequence $(e^{s(A_k,B_k)} : k\geq m)$, cf.~\Cref{exa:SRE}.}

\begin{proposition}\label{prop:multivariate_reg_var}
%Under Assumptions \ref{hypo:negative_drift} and \ref{hypo:E'}, 
The field $\bX$ is regularly varying with tail index $\ts$ and with the spectral tail field $\bTheta=(\Theta_{i,j}:i,j\in\Z)$ satisfying 
\begin{itemize}
\item[(i)] $\Theta_{i,j}=0$ for $i,j\in\Z$, $i\neq j$.
\item[(ii)] $\Theta_{m,m}=e^{S_{m}^{\bz}}$ for $m\in \Z$, 
%where $(S_{m}^{\bz})$ is defined in (\ref{eq:defOfSm}).
where $S^\bz _ 0 = 0$ and 
\begin{align*}
%\label{eq:defOfSm}
   S^\bz_m = \sum_{i=1}^m \bz_i\,, \mbox{ for } m \geq 1 \quad \mbox{ and  }   
   S^\bz_m = - \sum_{i=1}^{-m}  \bz^*_i\,, \mbox{ for } m \leq -1\,,
\end{align*}
for independent i.i.d.\ sequences
$ (\bz_i)_{i \geq 1}$ and $ (\bz^*_i)_{i \geq 1}$ whose distributions correspond to
the distributions of $s(A,B)$ under the product measure $\mu_A \times \mu_B$ and 
under the tilted measure $\mus$ from (\ref{eq:tilted_measure}), respectively.
\end{itemize}

\end{proposition} 

Before proving \Cref{prop:multivariate_reg_var} we give one expression for the constant $\vartheta$ and one representative of the anchored spectral tail process $\tilde{\bQ}$, both defined in \Cref{subs:anchor}.

\begin{corollary}
The tail field $\bY$ of $\bX$ satisfies $\pr(\bY\in l_0)=1$ with 
\begin{align}\label{eq:extremalIndex_align}
\vartheta=\pr(\Gamma + \max_{m\geq 1} S_{m}^{\bz}\leq 0)>0 \, ,
\end{align}
where $\Gamma$ is independent of $(S_m^{\bz})_{m\geq 1}$ and satisfies $\pr(\Gamma \geq x)=e^{-\ts x}$, $x\geq 0$. A representative $\bQ=(Q_{i,j})_{i,j\in \Z}$ of the anchored spectral tail process $\tilde{\bQ}$ is given by 
\begin{align}\label{eq:Q's_align}
Q_{i,j}=0 \: \text{ for } i\neq j\, , \: (Q_{m,m})_{m\in \Z} \eind \big(e^{S_{m}^{\bz}}, m\in \Z \, \mid \, \sup_{m\leq -1}S_{m}^{\bz} <0 , \, \sup_{m\geq 1}S_{m}^{\bz} \leq 0\big) \, .
\end{align}
\end{corollary}

\begin{proof}

The tail field $ \bY=(Y_{i,j})_{i,j\in \Z}$ of $\bX$ is given by  $Y_{i,j}=Y \cdot  \Theta_{i,j}$ where $Y$ satisfies $\pr(Y\geq y)=y^{-\ts}$ for $y \geq 1$ and is independent from $\bTheta$. Observe, $\ex[\bz_1]=\ex[s(A,B)]<0$ and since the moment generating function $m(\alpha)=\ex[e^{\alpha s(A,B)}]$ is strictly convex and $m(0)=m(\ts)=1$, 
\begin{align*}
\ex[\bz_1^*]=\ex[s(A,B)e^{\ts s(A,B)}]=\tfrac{\mathrm{d}m}{\mathrm{d}\alpha}(\ts)>0 \, .
\end{align*} 
This implies that $\pr(\lim_{|m|\toi}S_{m}^\bz = -\infty)=1$ so $\bTheta$ and $\bY$ are elements of $l_0$ almost surely.
In particular, by (\ref{eq:extremal_index}),
\begin{align*}
%\label{eq:extremalIndex_align}
0<\vartheta=\pr(\sup_{(i,j) \lgreater (0,0)} Y_{i,j}\leq 1)=\pr(Y\max_{m\geq 1}\Theta_{m,m}\leq 1)=\pr(\log Y + \max_{m\geq 1} S_{m}^{\bz}\leq 0) \, ,
\end{align*}
where $\log Y$ is a standard exponential random variable with index $\ts$. This yields (\ref{eq:extremalIndex_align}) and (\ref{eq:Q's_align}) follows directly from (\ref{eq:Q's}).
%By (\ref{eq:Q's}), a representative $\bQ=(Q_{i,j})_{i,j\in \Z}$ of the anchored spectral tail process $\tilde{\bQ}$ is given by 
%\begin{align}
%\label{eq:Q's_align}
%Q_{i,j}=0 \: \text{ for } i\neq j\, , \: (Q_{m,m})_{m\in \Z} \eind \big(e^{S_{m}^{\bz}}, m\in \Z \, \mid \, \sup_{m\leq -1}S_{m}^{\bz} <0 , \, \sup_{m\geq 1}S_{m}^{\bz} \leq 0\big) \, .
%\end{align}
\end{proof}

To prove Proposition \ref{prop:multivariate_reg_var} we need two auxiliary lemmas. The first one is a rough estimate using Markov inequality, see Section \ref{sub:alignments_appendix} for the proof.
%The first one is a rough estimate using Markov inequality, see Section \ref{sub:alignments_appendix} for the proof. the second one can be proved as \cite[Lemma 5.11]{hansen:2006} by an application of the Azuma-Hoeffding inequality. Moreover, the proof is even simpler since in \cite{hansen:2006}, sequences $(A_i)$ and $(B_i)$ are assumed to be Markov chains.  
\begin{lemma}\label{lem:1}
There exist a constant $c_0>0$ such that
\[
\lim_{u\toi}e^{2\ts u} \pr\left(\max_{m > c_0 u}S_{0,0}^m \geq 0\right)=0 \, .
\]
\end{lemma}
Before we state the second lemma, observe first that, using $\ex[e^{\ts s(A,B)}]=1$, for all $u\geq 0$ and any integer $m\geq 0$,
\begin{align*}
\pr(S_{0,0}^m\geq u)=\ex[e^{-\ts S_{0,0}^m} e^{\ts S_{0,0}^m} \ind{S_{0,0}^m\geq u}] \leq e^{-\ts u} \pr^*(S_{0,0}^m\geq u) \leq e^{-\ts u} \, ,
\end{align*}
where the tilted measure $\pr^*$ makes pairs $(A_k,B_k)$ for $k=-m+1,\dots,0$, independent and distributed according to the measure $\mu^*$. The following result is proved in \cite[Lemma 5.11]{hansen:2006} using change of measure arguments and the Azuma-Hoeffding inequality for martingales. The key fact is that, whenever $\mu^*\neq \mu_A^*\times \mu_B^*$ (which holds under (\ref{eq:E'})), 
\begin{align*}
\ex_{\nu_A \times \nu_B}[s(A,B)]< \ex_{\mu^*}[s(A,B)] 
\end{align*}
for all $\nu_A\in \{\mu_A,\mu_A^*\}$ and $\nu_B\in \{\mu_B,\mu_B^*\}$, where $\ex_{\mu}$ denotes the expectation assuming $(A,B)$ is distributed according to $\mu$, see \cite[beginning of Section 3]{dembo:karlin:zeitouni:1994a}.  The proof of {\cite[Lemma 5.11]{hansen:2006}} is much simpler in the i.i.d.\ setting and can be found in \cite[Lemma 4.2.3]{thesis}.
\begin{lemma}[{\cite[Lemma 5.11]{hansen:2006}}]\label{lem:2}
There exists an $0<\epsilon_0<1$ such that for all $u>0$,
\[
\sup_{\substack{i,j\in\Z,\, i\neq j \\ 
m,l\geq 0}}\pr(S_{0,0}^m>u, S_{i,j}^l>u) \leq 2 e^{-(1+\epsilon_0)\ts u} \, .
\]
\end{lemma}
\begin{proof}[Proof of Proposition \ref{prop:multivariate_reg_var}]
Let 
$\bTheta$ be from the statement of the proposition. We first show that, as $u\toi$, 
\begin{equation}\label{eq:fidi}
X_{0,0}^{-1}\bX_{I} \;\big| \;X_{0,0}>u \dto \bTheta_I \,,
\end{equation}
for all $I\subseteq \Z^2\setminus{\{ (m,m) : m \leq -1\}}$. Since $X_{0,0}$ is regularly varying with index $\ts$, this will prove the regular variation property of $\bX$ and show that the spectral tail field $\bTheta'=(\Theta_{i,j}')_{i,j\in \Z}$ of $\bX$ satisfies 
\begin{align}\label{eq:tailFieldProof_inter0}
(\Theta_{i,j}' : (i,j)\in \Z^2\setminus{\{ (m,m) : m \leq -1\}})\eind (\Theta_{i,j} : (i,j)\in \Z^2\setminus{\{ (m,m) : m \leq -1\}})\, ,
\end{align}
see Remark \ref{rem:convToSpectralIsSuff}. 
%Let $Y$ be a Pareto random variable with index $\ts$, i.e. $\pr(Y\geq y)=y^{-\ts}$ for $y \geq 1$, independent from 
%$\bTheta$ from the statement of the proposition. Define $\bY=(Y_{i,j})_{i,j\in \Z}$ by  $Y_{i,j}=Y \Theta_{i,j}$. We first show that 
%\begin{equation}\label{eq:fidi}
%u^{-1}\bX_{I} \;\big| \;X_{0,0}>u \dto \bY_I \,,
%\end{equation}
%for all $I\subseteq \Z^d_{\lgeq}$. By Theorem \ref{thm:tail_process}, this will prove the regular variation property of $\bX$ and show that the tail process $\bY'=(Y_{i,j}')_{i,j\in \Z}$ of $\bX$ satisfies $(Y_{i,j}')_{(i,j)\lgeq (0,0)}\eind (Y_{i,j})_{(i,j)\lgeq (0,0)}$.

Observe, by (\ref{eq:Lind}), for each $m \geq 1$,
\begin{align*}
X_{m,m}=\max\{X_{m-1,m-1}e^{s(A_m,B_m)},1\} \, .
\end{align*}
Now since $X_{0,0}$ is regularly varying and independent of the i.i.d.\ sequence $(e^{s(A_k,B_k)})_{k\geq 1}$, \cite[Theorem 2.3]{segers:2007} implies that for all $m\geq 0$, as $u \toi$,
\begin{align*}
X_{0,0}^{-1}\left(X_{0,0}, X_{1,1},\dots,X_{m,m}\right)\; \big| \; X_{0,0}>u \dto  &\big(1, e^{s({A}_1,{B}_1)},\dots, \prod_{k=1}^{m} e^{s({A}_k,{B}_k)}\big)\\
 &\eind (\Theta_{0,0},\Theta_{1,1},\dots, \Theta_{m,m})\, .
\end{align*}
Since $\Theta_{i,j}=0$ for all $i,j \in \Z$, $i\neq j$, (\ref{eq:fidi}) will follow if we show that for all such $i,j$, 
\begin{align}
\pr(X_{i,j}>X_{0,0}\eta \mid X_{0,0}>u)&\leq\pr(X_{i,j}>u\eta \mid X_{0,0}>u)\nonumber \\
&=\pr(S_{i,j}>\log u +\log \eta \mid S_{0,0}>\log u) \to 0 \, , \text{ as } u\toi\, , \label{al:tailFieldProof_inter1}
\end{align}
for all $\eta\in (0,1)$.

Fix now $i, j\in \Z$ such that $i\neq j$. Using (\ref{eq:exp_tail}) and Lemmas \ref{lem:1} and \ref{lem:2}, for every $M\geq 0$,
\begin{multline*}
\limsup_{u\toi}\pr(S_{i,j}>u-M \mid S_{0,0}>u)=\limsup_{u\toi}C^{-1} e^{\ts u} \pr(S_{0,0}>u, S_{i,j}>u-M)\\
\le \limsup_{u\toi} C^{-1} e^{\ts u} \pr\left(\max_{1\leq m \leq c_0 u}S_{0,0}^m >u-M,\max_{1\leq l \leq c_0 u}S_{i,j}^l >u-M\right)\\
\leq \limsup_{u\toi}2 C^{-1} e^{(1+\epsilon_0)\ts M} (c_0 u)^2 e^{-\epsilon_0\ts u} = 0\, ,
\end{multline*}
hence (\ref{al:tailFieldProof_inter1}) holds.

Finally, we extend (\ref{eq:tailFieldProof_inter0}) to equality in distribution on whole $\R^{\Z^2}$. 
%show that the distribution of the spectral tail field $\bTheta'$  of $\bX$ on  equals to the distribution of $\bTheta$. 
First, fix $m\geq 1$ and note that by (\ref{eq:time-change0}) and $\ex[e^{\ts s(A,B)}]=1$,
$$\pr(\Theta_{-m,-m}'> 0)=\ex[\Theta_{m,m}^{\ts}]=1 \, .$$ 
Further, for arbitrary bounded measurable function $h:\R^{2m+1}\to \R$, using (\ref{eq:time-change0}) and (\ref{eq:tailFieldProof_inter0}),
\begin{multline*}
\ex[h(\Theta_{-m,-m}',\dots,\Theta_{m,m}')] 
= \ex[h(\Theta_{m,m}^{-1}(\Theta_{0,0},\dots,\Theta_{2m,2m})) \Theta_{m,m}^{\ts}] \\
=\ex\big[h(e^{-\sum_{k=1}^m \bz_{k}},e^{-\sum_{k=2}^m \bz_{k}},\dots,e^{\bz_m},1,e^{\bz_{m+1}},\dots, e^{\sum_{k=m+1}^{2m}\bz_{k}})\prod_{k=1}^m e^{\ts \bz_{k}}\big] \, .
\end{multline*}
By definition of $(\Theta_{k,k})_{k\in \Z}$, this implies that 
\begin{align*}
\ex[h(\Theta_{-m,-m}',\dots,\Theta_{m,m}')] 
=\ex[h(\Theta_{-m,-m},\dots,\Theta_{m,m})] \, .
\end{align*}
\end{proof}

\subsection{Checking the assumptions of Theorem \ref{thm:PPconv}}
\label{sec:align_checkingAssumptions}
In view of \eqref{eq:reg_var_tail}, define the sequence $(a_n)$ by
\begin{equation*}
a_n=(\tailcons n^2)^{{1}/{\ts}}, \; n\in\N \, ,
\end{equation*}
so that $\lim_{n\toi} n^2 \pr(X_{0,0}>a_n)=1$.
%It follows from \eqref{eq:reg_var_tail} that the sequence
%\begin{equation*}
%a_n:=(\tailcons n^2)^{{1}/{\ts}}, \; n\in\N
%\end{equation*}
%satisfies  
%%\eqref{eq:a_n}, i.e.\ that 
%$\lim_{n\toi} n^2 \pr(X_{\bo}>a_n)=1$.
The proof of the following result is postponed to Section \ref{sub:ACforAlignments}.
\begin{proposition}\label{prop:assumptions}
The random field $\bX$ satisfies Assumption \ref{hypo:AC} for every sequence of positive integers $(r_n)$ such that $\lim_{n\toi} r_n=\infty$ and $\lim_{n\toi}r_n /n^\epsilon =0$ for all $\epsilon>0$.  
\end{proposition}
%%\hrnote{Actually there exists $0<\delta < 1$ such that $r_n = \lfloor n^ \delta\rfloor$ is good for Assumption \ref{hypo:AC}.}

Take now two sequences of positive integers $(l_n)$ and $(r_n)$ such that 
\[
\lim_{n\toi}\log n /l_n =\lim_{n\toi} l_n / r_n = \lim_{n\toi} r_n/n^\epsilon = 0  
\]
for all $\epsilon>0$ and set $k_n=\lfloor n/r_n \rfloor$. Recall the blocks of indices $J_{n,\bi}\subseteq \{1,\dots,k_n r_n\}^2$ of size $r_n^2$ from (\ref{eq:blocks_indices}) and the blocks $\bX_{n,\bi}:=\bX_{J_{n,\bi}}$ for $\bi \in I_n :=\{1,\dots, k_n\}^2$. To show that the $\bX_{n,\bi}$'s satisfy the asymptotic independence condition from Theorem \ref{thm:PPconv}, we will apply \Cref{prop:CheckingAI_regVarFields}. However, to use it we first need to alter the original blocks.

First, cut off the edges of the $J_{n,\bi}$'s by $l_n$, more precisely,  define
\begin{align*}%\label{eq:tildeJni}
\tilde{J}_{n,\bi}:= \{(i,j) : (\bi-\bone)\cdot r_n+ \bone \leq (i,j) \leq \bi \cdot r_n - l_n\cdot \bone \} , \: \bi \in I_n \, .
\end{align*}

Further, for all $i,j\in \Z$ and $m\in \N$ let $\varepsilon_{i,j}^m$ be the empirical measure on $E^2$ of the sequence $(A_{i-k},B_{j-k})$, $k=0,\dots,m-1$, i.e.\ 
\[
\varepsilon_{i,j}^m=\frac{1}{m}\sum_{k=0}^{m-1} \delta_{(A_{i-k},B_{j-k})} \, .
\]
%Notice, the score $S_{i,j}^m=\sum_{k=0}^{m-1}s(A_{i-k},B_{j-k})$ is then equal to the integral of the score function $s$ w.r.t.\ $\varepsilon_{i,j}^m$.
For every $\eta>0$ denote by $B_\eta$ the set of all probability measures $\nu$ on $E^2$ satisfying $||\nu - \mu^*||:=\sum_{a,b\in E}|\nu(a,b)-\mu^*(a,b)|<\eta$. 

Set $b_n=\log a_n$ for all $n\in \N$ and for all $\eta>0, i,j\in \Z$ define the random variable $\tilde{S}_{i,j}=\tilde{S}_{i,j}(n,\eta)$ by
\begin{equation}\label{eq:S_j'}
\tilde{S}_{i,j} =\max \{S_{i,j}^m : 1 \leq m \leq c_0 b_n,\:\varepsilon_{i,j}^m \in B_\eta\}
\end{equation}
with $c_0>0$ from Lemma \ref{lem:1} and $\max \emptyset :=0$. Further, define 
the modified blocks $\tilde{\bX}_{n,\bi} = \tilde{\bX}_{n,\bi}(\eta)$ in $\lo$ by
\begin{equation*}%\label{eq:X_n,i'}
\tilde{\bX}_{n,\bi}=(e^{\tilde{S}_{i,j}}:(i,j)\in \tilde{J}_{n,\bi}) \; .
\end{equation*}
It turns out that by restricting to the $\tilde{\bX}_{n,\bi}$'s one does not lose any relevant information. To understand the role of the $\tilde{\bX}_{n,\bi}$'s, 
observe that for any nonnegative and measurable function $f$ on $[0,1]^2\times \loo$,
\begin{multline}
\big| \ex\left[e^{-\sum_{\bi\in I_n}f (\bi /k_n, \bX_{n,i}/a_n)}\right] -\textstyle{\prod_{\bi\in I_n}} \ex\left[e^{-f (\bi /k_n, \bX_{n,i}/a_n)}\right] \big| \\
\leq\big| \ex\left[e^{-\sum_{\bi\in I_n}f (\bi /k_n, \bX_{n,i}/a_n)}\right] -  \ex\big[e^{-\sum_{\bi\in I_n}f (\bi /k_n, \tilde{\bX}_{n,i}/a_n)}\big] \big| \\
+\big|\textstyle{\prod_{\bi\in I_n}} \ex\left[e^{-f (\bi /k_n, \bX_{n,i}/a_n)}\right] -  \textstyle{\prod_{\bi\in I_n}}\ex\big[e^{-f (\bi /k_n, \tilde{\bX}_{n,i}/a_n)}\big] \big| \\
+ \big| \ex[e^{-\sum_{\bi\in I_n}f (\bi /k_n, \tilde{\bX}_{n,i}/a_n)}] -\textstyle{\prod_{\bi\in I_n}} \ex[e^{-f (\bi /k_n, \tilde{\bX}_{n,i}/a_n)}] \big| =: I_1 + I_2 + I_3 \, . \label{eq:equiv_tilde}
\end{multline} 

Recall now the convergence determining family $\F'_0$ from \Cref{subs:PP_conv}.
The proof of the following result is in Section \ref{sub:AIequivalence}.
\begin{lemma}\label{lem:AIequivalence}
For every $\eta>0$ and every $f\in \F'_0$, $I_1+I_2 \to 0$ as $n\toi$.
%For every $\eta>0$, the family $((\bi/k_n,\bX_{n,\bi}/a_n): n\in \N, \: \bi \in I_n)$ is $AI(\F'_0)$ if and only if $((\bi/k_n,\tilde{\bX}_{n,\bi}/a_n): n\in \N, \: \bi \in I_n)$ is $AI(\F'_0)$.
\end{lemma}

\begin{remark}\label{rem:originalScores}
In particular, since $I_1\to 0$ for all $f\in \F'_0$, point processes $\sum_{\bi \in I_n} \delta_{(\bi /k_n, \tilde{\bX}_{n,\bi})}$, which are based on the $\tilde{S}_{i,j}$'s, converge in distribution if and only if point processes $\sum_{\bi \in I_n} \delta_{(\bi /k_n, \bX_{n,\bi})}$, which are based on the $S_{i,j}$'s from (\ref{eq:stationary_scores}), do, and in that case their limits coincide. Similarly, one can show that the former (and therefore the latter) convergence is equivalent to convergence of point processes of blocks based on nonstationary scores from (\ref{eq:original_scores}). In particular, the point process convergence results given below hold even with the $S_{i,j}$'s from (\ref{eq:stationary_scores}) replaced with the ones from (\ref{eq:original_scores}).
\end{remark}

By (\ref{eq:equiv_tilde}) and Lemma \ref{lem:AIequivalence}, to show that the $(\bi/k_n,\bX_{n,\bi}/a_n)$'s are $AI(\F'_0)$, it is sufficient to find at least one $\eta>0$ such that $I_3\to 0$ for all $f\in \F'_0$, i.e.\ that the $(\bi/k_n,\tilde{\bX}_{n,\bi}/a_n)$'s are $AI(\F'_0)$. For that purpose, we apply \Cref{prop:CheckingAI_regVarFields}.

For every $\bi=(i_1,i_2)\in I_n$ define its neighborhood $B_n(\bi)$ by
\[
B_n(\bi)=\{\bj=(j_1,j_2)\in I_n : i_1=j_1 \; \text{or} \; i_2 = j_2\} \, .
\]
Observe, $|B_n(\bi)|=2k_n -1$ for all $\bi \in I_n$ and hence  $\lim_{n\toi}\|B_n\|/k_n^2 = 0$. Further, by (\ref{eq:reg_var_tail}) for all $\epsilon>0$,
\begin{align*}
\limsup_{n\toi}k_n^2 \pr(\|\tilde{\bX}_{n,\bone}\|>a_n\epsilon)&\leq \limsup_{n\toi} k_n^2r_n^2 \pr(e^{\tilde{S}_{0,0}}>a_n \epsilon) \\
&\leq \limsup_{n\toi} k_n^2r_n^2 \pr(X_{0,0}>a_n \epsilon) =\epsilon^{-\ts} <\infty \, .
\end{align*}

Next, recall that $S_{i,j}^m=\sum_{k=0}^{m-1}s(A_{i-k},B_{j-k})$ so by \eqref{eq:S_j'}, for every $n\in \N$, 
$$\tilde{S}_{i,j} \in \sigma(A_{i-\lfloor c_0 b_n \rfloor +1},\dots, A_{i},B_{j-\lfloor c_0 b_n \rfloor +1},\dots, B_{j}) \, . $$

By the construction of the $\tilde{J}_{n,\bi}$'s and  the choice of $(l_n)$ such that, in particular,
$\lim_{n\toi} c_0 b_n /l_n = \lim_{n\toi} l_n /r_n = 0$, this implies that,
  for $n$ large enough, $\tilde{\bX}_{n,\bi}$ and  the blocks $(\tilde{\bX}_{n,\bj}:\bj \notin B_n(\bi))$ are constructed from completely different sets of the $A_k$'s and the $B_k$'s, and therefore independent. 

Further, when $\bj \in B_n(\bi)$, $\bj\neq \bi$, 
arbitrary scores $S_{i,j}^m$ and $S_{i',j'}^l$  which build blocks $\tilde{\bX}_{n,\bi}$ and $\tilde{\bX}_{n,\bj}$, respectively (i.e.\ $(i,j)\in \tilde{J}_{n,\bi}, (i',j')\in \tilde{J}_{n,\bj}$ and $1\leq m,l \leq c_0 b_n$), for $n$ large enough, depend on completely different sets of variables from at least one of the sequences $(A_k)$ or $(B_k)$. Thus, the following result, which is \cite[Corollary 5.4]{hansen:2006}, applies.

\begin{lemma}[{\cite[Corollary 5.4]{hansen:2006}}]\label{lem:bound_incl_totalvariation}
There exist constants $\epsilon_2, \eta>0$ such that for all $u>0$
\begin{equation*}
\pr(S_{0,0}^m >u,  S_{i,j}^l >u ,\; \varepsilon_{0,0}^m, \varepsilon_{i,j}^l \in B_{\eta}) \leq e^{-(3/2+\epsilon_2)\ts u} 
\end{equation*}
uniformly over all $i,j\in \Z$ and $m,l\in \N$ such that 
$\min\{i,j\}<-m+1$ or $\max\{i-l,j-l\}>0$.
\end{lemma}

\begin{remark}
\cite[Corollary 5.4]{hansen:2006} follows from \cite[Lemma 5.3]{hansen:2006} under condition (12) in \cite{hansen:2006}, which, when $(A_i)$ and $(B_i)$ are i.i.d.\ sequences,  is equivalent to Assumption \ref{hypo:E'}, see \cite[Remark 3.8]{hansen:2006}; the proof can be found in \cite[Lemma 4.3.5]{thesis}. For a different and, in this i.i.d.\ setting, probably better argument, see \cite[pp.\ 2032--2033]{dembo:karlin:zeitouni:1994}. Note that the fact that $E$ is finite is here exploited.
\end{remark}

Take now the constant $\eta>0$ from the previous result and recall the corresponding $\tilde{\bX}_{n,\bi}$'s. For $n$ big enough and every $\epsilon>0$ we get that
\begin{multline*}
k_n^2 \|B_n\| \max_{\substack{\bi \in I_n \\ \bi\neq \bj\in B_n(\bi)}}\pr(\|\tilde{\bX}_{n,\bi}\|>a_n\epsilon , \,\|\tilde{\bX}_{n,\bj}\|>a_n\epsilon) \\
 \leq k_n^2 2k_n r_n^4 (c_0 b_n)^2 e^{-(3/2+\epsilon_2)\ts (b_n + \log\epsilon)} \sim (const.)\,  n^3 r_n b_n^2 n^{-3 - 2 \epsilon_2}\to 0 \, ,
\end{multline*}
as $n \toi$, by the choice of $(r_n)$ and since $b_n\sim 2\log n / \ts$.

Hence by \Cref{prop:CheckingAI_regVarFields}, for this $\eta$, the blocks $\tilde{\bX}_{n,\bi}$, and therefore the original blocks $\bX_{n,\bi}$, satisfy the asymptotic independence condition. We can now apply \Cref{thm:PPconv}: 
%\begin{theorem}
%\label{thm:PPconv_Alignments}
%Under Assumptions \ref{hypo:negative_drift} and \ref{hypo:E'}, for any sequence of positive integers $(r_n)$ such that $\lim_{n\toi} r_n = \infty$ and $\lim_{n\toi}r_n/n^\epsilon\to 0$ for all $\epsilon>0$,
the convergence
\begin{align}\label{eq:PPconv_alingm_pomocna}
\sum_{\bi\in I_n} \delta_{(\bi/k_n,\bX_{n,\bi}/(Cn^2)^{1/\ts})} \dto \sum_{k\in \N}\delta_{(\bT_k, P_k(Q^k_{i,j})_{i,j \in \Z} )} 
\end{align}  
holds in $\mathcal{M}_p([0,1]^2 \times \loo)$ where the limit is described in Theorem \ref{thm:PPconv} and \Cref{rem:explanation}, with $\vartheta$ given by (\ref{eq:extremalIndex_align}) and $(Q_{i,j}^k)_{i,j\in \Z}$, $k\in \N$ with the distribution given in (\ref{eq:Q's_align}).
%\end{theorem}

\Cref{thm:PPconv_Alignments_intro} stated in the introduction now follows from (\ref{eq:PPconv_alingm_pomocna}) by an application of the continuous mapping theorem since the mapping
\begin{align*}
\sum_{k\in \N} \delta_{(\bt_k,\bx_k)}\mapsto \sum_{k\in \N} \delta_{(\bt_k,\bx_k \cdot C^{1/\ts})} \, ,
\end{align*} 
is continuous (see \cite[Proposition 2.8]{basrak:planinic:2019}), and then applying standard Poisson process transformation arguments (see e.g.\ \cite[Proposition 3.7]{resnick:1987}).

Consider now the space $\mathcal{M}_p([0,1]^2 \times \R)$ with a set $B\subseteq [0,1]^2 \times \R$ being bounded if $B\subseteq [0,1]^2 \times (x,\infty)$ for some $x\in \R$.
\begin{corollary}\label{cor:sequence_alignment}
Under Assumptions \ref{hypo:negative_drift} and \ref{hypo:E'}, 
\[
\sum_{i,j=1}^n \delta_{\left(\tfrac{(i,j)}{n},\: S_{i,j}- \tfrac{2\log(n)}{\ts}\right)} 
\dto 
\sum_{k\in \N} \sum_{m\in \Z} \delta_{(\bT_k, \tilde{P}_k + \tilde{Q}^{k}_m  )} \, 
\] 
in $\mathcal{M}_p([0,1]^2 \times \R)$
 where
  \begin{enumerate}[(i)]
  \item $\sum_{k\in \N}\delta_{(\bT_k,\tilde{P}_k)}$ is a Poisson point process on $[0,1]^2\times \R$
   with intensity measure $\vartheta \tailcons \mbox{Leb} \times \ts e^{-\ts u} du$; 
%   where, for an exponential random variable $\Gamma$ with parameter $\ts$ independent of $(S_m^{\bz})$,
%   \begin{align*}
%   \vartheta=\PP ( \sup_{m \geq 1}S^\bz_m + \Gamma  \leq  0) \, ;
%   \end{align*}
%   \[
%   \vartheta=\int_{-\infty}^0 \pr\left(\sup_{m \geq 1}S^\bz_m \leq z \right)\ts e^{\ts z}dz \in (0,1] \; ; 
%   \]
  \item $(\tilde{Q}^{k}_{m})_{m\in \Z}, \: k\in \N$ are i.i.d.\ two-sided $\R$-valued sequences, independent of
    $\sum_{k\in \N}\delta_{(\bT_k,\tilde{P}_k)}$ and with common distribution equal to the distribution of the random walk $(S^\bz_m)_m$ conditioned on staying negative for $m<0$ and nonpositive for $m>0$.
  \end{enumerate}
\end{corollary}
\begin{proof}
An application of \Cref{cor:PPconv_back_to_R} to the convergence in (\ref{eq:PPconv_alingm_pomocna}) yields that 
\begin{align}\label{eq:PPconvInR_alignments}
\sum_{i,j=1}^n\delta_{((i,j) / n , X_{i,j}/(Cn^2)^{1/\ts})} \dto \sum_{k\in \N} \sum_{i,j\in \Z}\delta_{(\bT_k, P_k Q^k_{i,j} )} = \sum_{k\in \N} \sum_{m \in \Z}\delta_{(\bT_k, P_k Q^k_{m,m} )}
\end{align}
in $\mathcal{M}_p([0,1]^2 \times (0,\infty))$, where the last equality follows since $Q^k_{i,j}=0$ for $i\neq j$. It is easy to see that  
\begin{align*}
\sum_{k\in \N} \delta_{(\bt_k,x_k)}\mapsto \sum_{k\in \N} \delta_{(\bt_k, \,  \log (x_k C^{1/\ts}))} \, 
\end{align*} 
is a well defined mapping from $\mathcal{M}_p([0,1]^2 \times (0,\infty))$ to $\mathcal{M}_p([0,1]^2 \times \R)$ which is also continuous w.r.t.\  the vague topologies on these spaces. %(use \cite[Proposition 2.8]{basrak:planinic:2019}). 
The result now follows easily from (\ref{eq:PPconvInR_alignments}) via the continuous mapping theorem and using standard Poisson process transformation arguments (again, see e.g.\ \cite[Proposition 3.7]{resnick:1987}).
\end{proof}

\section{Postponed proofs} \label{sec:App}
%\subsection{Regularly varying random fields}
\subsection{Proof of Theorem \ref{thm:tail_process}}\label{subs:RVequiv}
We only prove \ref{it:tail_process2}$\Rightarrow$\ref{it:tail_process1} since \ref{it:tail_process1}$\Rightarrow$\ref{it:tail_process3} follows as in \cite[Theorem 2.1]{basrak:segers:2009} and \ref{it:tail_process3}$\Rightarrow$\ref{it:tail_process2} is obvious. Also, since we essentially adapt the arguments of \cite[Theorem 2.1]{basrak:segers:2009}, some details are omitted.

Observe first that (\ref{eq:conv_to_tail_forward}) with $I=\{\bo\}$ implies that   for all $\epsilon>0$,
\begin{align}\label{eq:inter_RV}
\lim_{u\toi}\frac{\pr(|X_{\bo}|>u\epsilon)}{\pr(|X_{\bo}|>u)}={\epsilon}^{-\alpha}\, ,
\end{align} 
and moreover that $X_{\bo}$ is a regularly varying random variable with index $\alpha$, see \cite[Theorem 2.1]{basrak:segers:2009}.

Take now an arbitrary finite $I\subseteq \Z^d$ such that $|I|\geq 2$ and consider the space $\R^{|I|} \setminus \{\bo\}$ with bounded sets being those which are contained in sets $B_{\epsilon}:=\{(x_{\bi})_{\bi \in I} \in \R^{|I|}: \sup_{\bi\in I}|x_{\bi}|>\epsilon\}$, $\epsilon>0$. In view of (\ref{eq:inter_RV}), multivariate regular variation (with index $\alpha$) of $\bX_{I}$ is equivalent to the existence of a nonzero measure $\mu_I\in \mathcal{M}(\R^{|I|} \setminus \{\bo\})$ such that 
\begin{align*}%\label{eq:RV_vague}
\mu_u^I(\, \cdot \,):= \frac{\pr( u^{-1} \bX_{I} \in  \cdot)}{\pr(|X_{\bo}|>u)}\vto \mu_I \, , \text{ as } \, u\toi \, , 
\end{align*} 
see~\cite[Definition 3.1, Proposition 3.1]{segers:zhao:meinguet:2017} (cf.~\cite[Equation (1.3)]{basrak:segers:2009}).

%Take now an arbitrary finite $I\subseteq \Z^d$ such that $|I|\geq 2$ and  for all $u>0$ consider the measure $\mu_u^{I}(\,\cdot\,)=\pr( u^{-1} \bX_{I} \in  \cdot)/\pr(|X_{\bo}|>u)$ as an element of the space $\mathcal{M}(\R^{|I|} \setminus \{\bo\})$ with bounded sets in $\R^{|I|} \setminus \{\bo\}$ being those which are contained in sets $\{(x_{\bi})_{\bi \in I} \in \R^{|I|}: \sup_{\bi\in I}|x_{\bi}|>\epsilon\}$, $\epsilon>0$. Multivariate regular variation (with index $\alpha$) of $\bX_{I}$ is now equivalent to the existence of a non--zero measure $\mu_I\in \mathcal{M}(\R^{|I|} \setminus \{\bo\})$ such that $\mu_u^I\vto \mu_I$ as $u\toi$, see \cite[Equation (1.3)]{basrak:segers:2009}.

Arguing exactly as in \cite[Theorem 2.1]{basrak:segers:2009} it follows that the vague limit of $\mu_{u}^{I}$, if it exists, is necessarily nonzero, and furthermore,  
that $\limsup_{u\toi} \mu_u^I(B_{\epsilon}) \leq |I| \epsilon^{-\alpha}<\infty$ for every $\epsilon>0$. Since sets $\{(x_{\bi})_{\bi \in I} \in \R^{|I|}: \sup_{\bi\in I}|x_{\bi}| \in [\epsilon,M]\}$ are compact for every $\epsilon,M >0$, by \cite[Theorem 4.2]{kallenberg:2017}  it follows that the set  $\{\mu_{u}^{I}: u>0\}$ is relatively compact in the vague topology of $\mathcal{M}(\R^{|I|} \setminus \{\bo\})$.
%Note that, since bounded sets in $R^{|I|} \setminus \{\bo\}$ do not coincide with the family of relatively compact sets in $R^{|I|} \setminus \{\bo\}$ 

Since $\mathcal{I}$ is encompassing, we can take $\bi^*\in I$ such that $I':=I-\bi^* \subseteq \mathcal{I}$.
%Let $\bi^*\in I$ be the minimal element of $I$ w.r.t.\ the lexicographic order. 
By \cite[Lemma 2.2]{basrak:segers:2009}, to show that measures $\mu_{u}^{I}$ vaguely converge as $u\toi$,
%in $\mathcal{M}(\R^{|I|} \setminus \{\bo\})$ as $u\toi$ (i.e.\ that (\ref{eq:def_of_mrv}) holds), 
%and hence that $\bX_I$ is multivariate regularly varying, 
it suffices to prove that $\lim_{u\toi} \mu_{u}^{I}(f)$ exists for  all $f\in \mathcal{F}$ where $\mathcal{F}=\mathcal{F}_1\cup \mathcal{F}_2\subseteq CB_b^+(\R^{|I|} \setminus \{\bo\})$ with
\begin{align*}
&\mathcal{F}_1=\{f : \, \text{for some} \; \epsilon>0, f((x_{\bi})_{\bi \in I})=0 \;  \text{if} \; |x_{\bi^*}| \leq \epsilon\} \, , \\
&\mathcal{F}_2=\{f :f((x_{\bi})_{\bi \in I}) \; \text{does not depend on} \; x_{\bi^*}\} \, .
\end{align*}
Note that families $\mathcal{F}_1$ and $\mathcal{F}_2$ depend on $I$ but we omit this in the notation. 

%Denote $I':=I-\bi^*$ so in particular, 
%where $\bi^*\in I$ is the minimal element of $I$ w.r.t.\ the lexicographic order. By construction, 
Since $I'\subseteq \mathcal{I}$, stationarity, (\ref{eq:conv_to_tail_forward}) and (\ref{eq:inter_RV}) imply that for every $f\in \F_1$ and $\epsilon>0$ as in the definition of $\F_1$,
\begin{align*}
\mu_u^{I}(f)=\frac{\pr(|X_{\bo}|>u\epsilon)}{\pr(|X_{\bo}|>u)}\cdot \ex[f(u^{-1}\bX_{I'})\mid |X_{\bo}|>u\epsilon] \to \epsilon^{-\alpha} \ex[f(\epsilon(Y_{\bi})_{\bi \in I'})]\, , \; \text{as} \; u\toi \, .
\end{align*}
Further, every $f\in \mathcal{F}_2$ naturally induces a function $\tilde{f}$ in  $CB_b^+(\R^{|I|-1} \setminus \{\bo\})$ and by stationarity
\begin{align*}
\mu_{u}^{I}(f)=\frac{\ex[\tilde{f}(u^{-1}\bX_{I\setminus\{\bi^*\}})]}{\pr(|X_{\bo}|>u)} =\mu_{u}^{I\setminus\{\bi^*\}}(\tilde{f}) \, .
\end{align*}
Hence, $\lim_{u\toi}\mu_{u}^{I}(f)$ exists for all $f\in \mathcal{F}_2$ if $\bX_{I\setminus\{\bi^*\}}$ is multivariate regularly varying. 

Observe, we have shown that for an arbitrary finite $I\subseteq \Z^d$ such that $|I|\geq 2$, $\bX_{I}$ is multivariate regularly varying if  $\bX_{I\setminus\{\bi^*\}}$ is, where $\bi^*\in I$ is 
such that $I-\bi^* \subseteq \mathcal{I}$.
%the minimal element of $I$ w.r.t.\ the lexicographic order. 
Therefore, \ref{it:tail_process1} now follows by regular variation of $X_{\bo}$ and since $\mathcal{I}$ is encompassing.

\subsection{Local sequence alignments}\label{sub:alignments_appendix}

\subsubsection{Proof of \Cref{lem:1}}
%\begin{lemma}[Lemma \ref{lem:1}]
%There exist a constant $c_0>0$ such that
%\[
%\lim_{u\toi}e^{2\ts u} \pr\left(\max_{m > c_0 u}S_{0,0}^m \geq 0\right)=0 \, .
%\]
%\end{lemma}
%\begin{proof}
By Markov inequality, for any $\lambda\geq 0$ and all $u>0$
\begin{align*}
\pr\left(\max_{m > c_0u}S_{0,0}^m \geq 0\right)\leq \sum_{l=0}^\infty \pr\left(S_{0,0}^{\lceil c_0u \rceil  +l} \geq 0\right)
\leq \sum_{l=0}^\infty \ex \left[e^{\lambda S_{0,0}^{\lceil c_0 u \rceil +l}}\right]
= \sum_{l=0}^\infty m(\lambda)^{\lceil c_0 u \rceil +l}\, ,
%= \sum_{l=0}^\infty e^{(\lceil c_0 u \rceil +l)\kappa(\lambda)}\, , 
\end{align*}
where $m(\lambda)=\ex[e^{\lambda s(A,B)}]$ is the moment generating function of $s(A,B)$.
%where $\kappa(\lambda)$ is the cumulant generating function from \eqref{eq:cgf}. 
Fix any $0<\lambda_0<\ts$.
%\hrnote{To optimize, take $\lambda_0$ which minimizes $m$.}
By strict convexity of $m$ and $m(\ts)=1$, 
$0<m(\lambda_0)<1$
% $\kappa(\lambda_0)<0$ 
and in particular
\[
\pr\left(\max_{m\geq c_0u}S_{0,0}^m \geq 0\right)\leq e^{c_0 u\log m(\lambda_0) } \sum_{l=0}^\infty m(\lambda_0)^l.  
\]
%\[
%\pr\left(\max_{m\geq c_0u}S_{0,0}^m \geq 0\right)\leq e^{ c_0 \kappa(\lambda_0) u } \sum_{l=0}^\infty \left( e^{\kappa(\lambda_0)}\right)^l.  
%\]
Since the series above is summable, taking 
%$c_0=-(2+\eta)\ts/\kappa(\lambda_0)$ for an $\eta>0$ arbitrary small 
$c_0$ strictly larger than $-2\ts/\log m(\lambda_0)$
finishes the proof.

\subsubsection{Proof of \Cref{prop:assumptions}}\label{sub:ACforAlignments}
%\begin{proposition}[Proposition \ref{prop:assumptions}]
%The random field $\bX$ satisfies Assumption \ref{hypo:AC} for every sequence of positive integers $(r_n)$ such that $\lim_{n\toi} r_n=\infty$ and $\lim_{n\toi}r_n /n^\epsilon =0$ for all $\epsilon>0$.  
%\end{proposition}
%\begin{proof}
%[Proof of Proposition \ref{prop:assumptions}]
Let $(r_n)$ be an arbitrary sequence of positive integers satisfying $r_n\toi$ and $r_n/n^\epsilon \to 0$ for all $\epsilon>0$. We have to show that for an arbitrary $u>0$   
\begin{align}\label{eq:ACforAlignments_inter0}
\lim_{m\toi}\limsup_{n\toi}\pr\left(\max_{m< |(i,j)|\leq r_n} X_{i,j}>a_nu\; \Big| \;  X_{0,0}>a_nu\right)=0 \; .
\end{align}
We deal with the diagonal elements using arguments from \cite[Lemma 4.1.4]{basrak:2000}. First, notice that by (\ref{eq:Lind}), for each $k\geq 1$ we can decompose
\begin{align*}
X_{k,k} = \max\{ e^{\max_{0\leq l \leq k} S_{k,k}^l}, X_{0,0} e^{S_{k,k}^k}\} \, ,
\end{align*}
with $(S_{k,k}^l)_{0\leq l\leq k}$ being independent of $X_{0,0}$. Hence, using stationarity, 
\begin{align*}
&\pr\left(\max_{m< |k|\leq r_n} X_{k,k}>a_n u \; \Big| \; X_{0,0}>a_n u\right) \leq 2 \sum_{k=m+1}^{r_n} \pr(X_{k,k}>a_n u\; | \; X_{0,0}>a_n u) \\
&\leq 2 r_n \pr(e^{\max_{0\leq l \leq r_n} S_{0,0}^l} >a_n u) + 2 \sum_{k=m+1}^{r_n} \pr(X_{0,0} e^{S_{k,k}^k}>a_n u \; | \; X_{0,0}>a_n u).
\end{align*}
Since $r_n/n^2 \to 0$, the choice of $(a_n)$ and (\ref{eq:reg_var_tail}) imply that
\begin{align*}
2 r_n \pr(e^{\max_{0\leq l \leq r_n} S_{0,0}^l} >a_n u) \leq 2 r_n \pr(X_{0,0} >a_n u) \to 0 \, , \text{ as } n\toi \, .
\end{align*}
For the second term, take an arbitrary $0<\lambda_0 < \ts$ so in particular 
$0<m(\lambda_0)=\ex[e^{\lambda_0 s(A,B)}]<1$ by strict convexity of $m$. Apply Markov's  inequality and use independence between $X_{0,0}$ and $S_{k,k}^k$ to obtain
\begin{align*}
\sum_{k=m+1}^{r_n} \pr(X_{0,0} e^{S_{k,k}^k}>a_n u \; | \; X_{0,0}>a_n u) \leq \frac{\ex[X_{0,0}^{\lambda_0}\ind{X_{0,0}>a_n u}]}{(a_n u)^{\lambda_0} \pr(X_{0,0}>a_n u)} \sum_{k=m+1}^{r_n} m(\lambda_0)^k \, .
\end{align*}
A variant of Karamata's theorem (see \cite[Appendix B.4]{buraczewski:damek:mikosch:2016}, also \cite[pp.\ 26--28]{bingham:goldie:teugels:1987}) now implies that 
\begin{align*}
\lim_{m\toi}\limsup_{n\toi}\pr\left(\max_{m< |k|\leq r_n} X_{k,k}>a_n u \; \Big| \; X_{0,0}>a_n u\right) \leq \frac{\ts}{\ts - \lambda_0} \lim_{m\toi}  \sum_{k=m+1}^{\infty} m(\lambda_0)^k =0 \, .
\end{align*}

It remains to deal with the non diagonal terms. More precisely, in order to obtain (\ref{eq:ACforAlignments_inter0}), we will show that, denoting $b_n=\log a_n$ and $M=\log u$,
\begin{multline*}
\limsup_{n\toi}\pr\left(\max_{|(i,j)|\leq r_n, \, i\neq j} S_{i,j}>b_n + M \; \Big| \; S_{0,0}>b_n + M \right)\\
=C^{-1}e^{\ts M}\limsup_{n\toi} e^{\ts b_n} \pr\left(\max_{|(i,j)|\leq r_n, \, i\neq j} S_{i,j}>b_n + M, \, S_{0,0}>b_n + M \right) =0 \, .
\end{multline*}
Notice that $e^{\ts b_n}=Cn^2$. First, since $r_n/n \to 0$, stationarity and Lemma \ref{lem:1} give 
\begin{align*}
\limsup_{n\toi} e^{\ts b_n}\pr\left(\max_{ \substack{|(i,j)|\leq r_n \\ k>c_0 b_n} } S_{i,j}^k\geq  0\right)&\leq \limsup_{n\toi} e^{\ts b_n} (2r_n +1)^2\pr\left(\max_{ k>c_0 b_n } S_{0,0}^k\geq 0\right)\\
&\leq \limsup_{n\toi} \frac{(2r_n +1)^2}{\tailcons n^2}=0 \, .
\end{align*}
Now by Lemma \ref{lem:2} there exist an $\epsilon_0>0$ such that
\begin{align*}
&\limsup_{n\toi} e^{\ts b_n} \pr\left(\max_{|(i,j)|\leq r_n, \, i\neq j} S_{i,j}>b_n + M, \, S_{0,0}>b_n + M \right)\\
&= \limsup_{n\toi} e^{\ts b_n}\pr\left(\max_{\substack{|(i,j)|\leq r_n,\, i\neq j\\ 1\leq l\leq c_0 b_n}} S_{i,j}^l>b_n + M, \max_{1\leq k\leq c_0 b_n}S_{0,0}^k>b_n + M \right) \\
&\leq \limsup_{n\toi} e^{\ts b_n} (2r_n +1)^2 (c_0 b_n)^2 2 e^{-(1+\epsilon')\ts b_n}\\
&=2 c_0^2 C^{-(1+\epsilon)} \limsup_{n\toi} \left(\frac{2r_n+1}{n^{\epsilon_0 /2}}\right)^2 \left(\frac{b_n}{n^{\epsilon_0/2}}\right)^2=0 \, ,
\end{align*}
where the last equality follows by the choice of $(r_n)$ and since $b_n\sim \tfrac{2}{\ts}\log n$.

\subsubsection{Proof of \Cref{lem:AIequivalence}}\label{sub:AIequivalence}
First, we need the following simple result proved  by a change of measure argument and a large deviation bound for empirical measures, cf.\ the proof of \cite[Lemma 5.14, Equation (54)]{hansen:2006}.
\begin{lemma}\label{lem:3}
For all $\eta>0$ there exists an $\epsilon_1>0$ such that
\[
\lim_{u \toi} e^{(1+\epsilon_{1})\ts u} \sup_{m\geq 1} \pr(S_{0,0}^m > u,\varepsilon_{0,0}^m \notin B_\eta ) = 0 \, .
\]
\end{lemma}
\begin{proof}
Fix $\eta>0$ and denote $A_m(u)=\{S_{0,0}^m > u,\varepsilon_{\bo}^m \notin B_\eta\}$ for $m\geq 1$ and $u>0$. Note that, since $S_{\bo}^m=\sum_{k=0}^{m-1} s(A_{-k},B_{-k})$, $\pr(A_m(u))=0$ whenever $m\leq u/\|s\|$, so for fixed $u>0$ we only need to deal with $\pr(A_m(u))$ for $m> u/\|s\|$. 

First, a change of measure yields
\begin{align*}
\pr(A_m(u))=\ex\left[\frac{\exp(\ts S_{0,0}^m)}{\exp(\ts S_{0,0}^m)} \mathbbm{1}_{A_m(u)}\right] \leq e^{-\ts u} \pr^*(\varepsilon_{0,0}^m \notin B_\eta) \, ,
\end{align*}
where $\pr ^*$ makes $(A_{-k},B_{-k}),\: k=0,\dots,m-1$, i.i.d.\ elements of $E^2$ with common distribution $\mu^*$. By Sanov's theorem (see \cite[Theorem 2.1.10]{dembo:zeitouni:2010})
\[
\limsup_{m\toi} \frac{1}{m} \log \pr^*(\varepsilon_{0,0}^m \notin B_\eta) \leq - \inf_{\pi \notin B_\eta} H(\pi \mid \mu^*) \, .
\]
%where the infimum is taken over all probability measures $\pi$ on $E^2$.
 Since, for a sequence of probability measures $(\pi_n)$ on $E^2$, $ H(\pi_n \mid \mu^*)\to 0$ implies that $\|\pi_n - \mu^*\|\to 0$, for we can find a constant $c=c(\eta)>0$ such that $\inf_{\pi \notin B_\eta} H(\pi \mid \mu^*)>c$. Hence, for all $m>u/\|s\|$ with $u$ large enough
\[
\pr^*(\varepsilon_{0,0}^m \notin B_\eta) \leq e^{-m c} \leq e^{-u c/\|s\|} \, .
\]
To finish the proof, it suffices to take $\epsilon_1:=\tfrac{c}{\|s\| \ts} >0$.
\end{proof}

%\begin{lemma}[Lemma \ref{lem:AIequivalence}]
%For every $\eta>0$, the family $((\bi/k_n,\bX_{n,\bi}/a_n): n\in \N, \: \bi \in I_n)$ is $AI(\F')$ if and only if $((\bi/k_n,\tilde{\bX}_{n,\bi}/a_n): n\in \N, \: \bi \in I_n)$ is $AI(\F')$.  
%\end{lemma}
\begin{proof}[Proof of \Cref{lem:AIequivalence}]
Take an arbitrary $f\in \F'_0\subseteq CB_{b}^+([0,1]^2 \times \loo)$ and let $\epsilon>0$ be such that  $f(\bt,(x_{i,j})_{i,j})=f(\bt,(x_{i,j}\1{\{|x_{i,j}|>\epsilon\}})_{i,j})$ for all $\bt \in [0,1]^2$ and $(x_{i,j})_{i,j\in \Z}\in\loo$ with $f(\bt,\bo)=0$.

By the elementary inequality $|\prod_{i=1}^k a_i - \prod_{i=1}^k b_i|\leq \sum_{i=1}^k |a_i-b_i|$ valid for all $k \geq 1$ and $a_i,b_i \in [0,1]$ (see e.g.\ \cite[Lemma 3.4.3]{durrett:2010}),
\begin{multline}
\big| \ex\left[e^{-\sum_{\bi\in I_n}f (\bi /k_n, \bX_{n,i}/a_n)}\right] -  \ex\big[e^{-\sum_{\bi\in I_n}f (\bi /k_n, \tilde{\bX}_{n,i}/a_n)}\big] \big| \\
+\big|\prod_{\bi\in I_n} \ex\left[e^{-f (\bi /k_n, \bX_{n,i}/a_n)}\right] -  \prod_{\bi\in I_n}\ex\big[e^{-f (\bi /k_n, \tilde{\bX}_{n,i}/a_n)}\big] \big| \\
 \leq 2 \sum_{\bi\in I_n} \ex\big|e^{-f (\bi /k_n, \bX_{n,i}/a_n)} - e^{-f (\bi /k_n, \tilde{\bX}_{n,i}/a_n)} \big| \, . \label{eq:AIequiv_inter0}
\end{multline} 
Further, denote by $J_{r_n}:=\{1,\dots,r_n\}^2=J_{n,\bone}$ and $\tilde{J}_{r_n}:=\{1,\dots,r_n-l_n\}^2 =\tilde{J}_{n,\bone}$. 
Using stationarity we get that
\begin{align}\label{eq:AIequiv_inter1}
%\big| \ex\left[e^{-\sum_{\bi\in I_n}f (\bi /k_n, \bX_{n,i}/a_n)}\right] -  \ex\big[e^{-\sum_{\bi\in I_n}f (\bi /k_n, \tilde{\bX}_{n,i}/a_n)}\big] \big| 
\sum_{\bi\in I_n} \ex\big|e^{-f (\bi /k_n, \bX_{n,i}/a_n)} - e^{-f (\bi /k_n, \tilde{\bX}_{n,i}/a_n)} \big|
\leq k_n^2 (A_1 + A_2 + A_3) \, ,
%\pr(\bigcup_{(i,j)\in J_{r_n}\setminus \tilde{J}_{r_n}}) 
\end{align}
where
\begin{align*}
&A_1=\pr(X_{i,j}>a_n \epsilon \; \text{ for some } (i,j)\in J_{r_n}\setminus \tilde{J}_{r_n}) \, , \\
&A_2= \pr(\max_{m>c_0 b_n}e^{S_{i,j}^m }> a_n \epsilon \; \text{  for some } (i,j) \in \tilde{J}_{r_n}) \, , \\
&A_3= \pr(e^{S_{i,j}^m }>a_n \epsilon \text{ and } \varepsilon_{i,j}^m \notin B_\eta \; \text{ for some } (i,j) \in \tilde{J}_{r_n},\, 1\leq m \leq c_0 b_n) \, .
\end{align*}
Observe,  $|J_{r_n}\setminus \tilde{J}_{r_n}|\leq 2 r_n l_n$ and $|\tilde{J}_{r_n}|\leq r_n^2$, and recall that $k_n r_n \sim n$ as $n\toi$, so using stationarity and then (\ref{eq:reg_var_tail}), Lemma \ref{lem:1} and Lemma \ref{lem:3}, respectively, 
\begin{align*}
\limsup_{n\toi} k_n^2 A_1 &\leq \limsup_{n\toi} 2 k_n^2 r_n l_n \pr(X_{0,0}>a_n\epsilon) = (const.) \limsup_{n\toi} l_n /r_n = 0 \, , \\
\limsup_{n\toi}k_n^2 A_2 &\leq\limsup_{n\toi} k_n^2 r_n^2 \pr(\max_{m>c_0 b_n}{S_{0,0}^m }\geq 0)\leq \limsup_{n\toi}n^{-2} = 0 \, , \\
\limsup_{n\toi} k_n^2 A_3 &\leq \limsup_{n\toi} k_n^2 r_n^2 c_0 b_n \sup_{m\geq 1} \pr({S_{0,0}^m }>b_n +  \log \epsilon, \, \varepsilon_{i,j}^m \notin B_\eta) \\
&\leq (const.) \limsup_{n\toi} b_n/n^{2\epsilon_1} = 0 \, .
\end{align*}
Therefore, the right hand side, and then also the left hand side, of (\ref{eq:AIequiv_inter1}) tends to $0$ as $n\toi$, and by (\ref{eq:AIequiv_inter0}) this proves the lemma. 
\end{proof}

\section*{Acknowledgements}
The research of both authors was supported in part by the HRZZ project "Stochastic methods in analytical and applied problems" (3526) and currently by the SNSF/HRZZ Grant "Probabilistic and analytical aspects of generalised regular variation" (180549). We also thank  the anonymous reviewers for their helpful comments and suggestions which lead to a significant improvement of the paper.

%%%%%%%%%%%%%%%%%%%%%%%%%%%%%%%%%%%%%%%%%%%%%%
%% Supplementary Material, if any, should   %%
%% be provided in {supplement} environment  %%
%% with title inside \textbf{} and short    %%
%% description below.                       %%
%%%%%%%%%%%%%%%%%%%%%%%%%%%%%%%%%%%%%%%%%%%%%%
%\begin{supplement}
%\textbf{???}.
%???.
%\end{supplement}

%%  imsart-number.bst  will be used to                     %%
%%  MR numbers will be added by VTeX.                      %%
%%                                                         %%
%%  Use \cite{...} to cite references in text.             %%
%%                                                         %%
%%%%%%%%%%%%%%%%%%%%%%%%%%%%%%%%%%%%%%%%%%%%%%%%%%%%%%%%%%%%%

%% if your bibliography is in bibtex format, uncomment commands:
%\bibliographystyle{imsart-number} % Style BST file
%\bibliography{bib-arrays}      % Bibliography file (usually '*.bib')

%% or include bibliography directly:
% \begin{thebibliography}{}
% \bibitem{b1}
% \end{thebibliography}

\end{document}